\documentclass[a4paper,10pt]{article}

\pdfoutput=1

\usepackage[utf8]{inputenc}
\usepackage{amsmath}
\usepackage{mathtools}
\usepackage{amssymb}
\usepackage{amsthm}
\usepackage{mathrsfs}
\usepackage{bm}
\usepackage[hidelinks]{hyperref}
\usepackage{bbm} 
\usepackage{subcaption} 
\usepackage{color}

\usepackage{tabu}

\newcommand{\me}{\mathrm{e}}
\newcommand{\ii}{\mathbbm{i}}
\newcommand{\tsmallbbw}{\tilde{\mathbbm{w}}}
\newcommand{\ges}{\geqslant} 
\newcommand{\les}{\leqslant} 
\newcommand{\R}{\mathbb{R}}
\newcommand{\C}{\mathbb{C}}

\newcommand{\N}{\mathbb{N}}
\newcommand{\Z}{\mathbb{Z}}
\newcommand{\dd}{\mathrm{d}}
\newcommand{\re}{\mathrm{Re\,}}
\newcommand{\im}{\mathrm{Im\,}}

\newcommand{\Vtilde}{\tilde{\mathbb{V}}}
\newcommand{\Vtildesep}{\tilde{V}}
\newcommand{\formalFlow}{\tilde{\mathfrak{W}}}
\newcommand{\formalSeparatrix}{\tilde{\mathscr{W}}}
\newcommand{\formalSeparatrixSing}{\tilde{W}}
\newcommand{\formalFlowSing}{\tilde{\mathbb{W}}}
\newcommand{\wcomb}{\overline{\mathscr{W}}}

\newcommand{\Vcaltilde}{\tilde{\mathcal{V}}}
\newcommand{\Vcalhat}{\hat{\mathcal{V}}}
\newcommand{\Ahat}{\hat{A}}
\newcommand{\Bhat}{\hat{B}}
\newcommand{\Xhat}{\hat{X}}
\newcommand{\Yhat}{\hat{Y}}

\newcommand{\myqed}{}

\newcommand{\Vector}[2]{
\begin{pmatrix}
 #1 \\
 #2 
\end{pmatrix}}

\newcommand{\bx}{\textbf{\textit{x}}}
\newcommand{\dt}{\text{t\hspace{-.35em}-}}

\newtheorem{theorem}{Theorem}[section]

\newtheorem{corollary}[theorem]{Corollary}

\newtheorem{lemma}[theorem]{Lemma}

\newtheorem{proposition}[theorem]{Proposition}

\theoremstyle{definition}
\newtheorem{definition}[theorem]{Definition}

\theoremstyle{definition}
\newtheorem*{remark}{Remark}

\allowdisplaybreaks

\title{Splitting of separatrices in a family of area-preserving maps that unfolds a fixed point at the resonance of order three}
\author{Giannis Moutsinas\thanks{This work was supported by the EPSRC  grant EP/J003948/1}
\\[6pt]
Mathematics Institute, University of Warwick\\[6pt]
{\small
E-mail: Giannis.Moutsinas@gmail.com}}

\begin{document}

\maketitle

\begin{abstract}

We study the exponentially small splitting of separatrices in an analytic one-parameter family of area-preserving
maps that generically unfolds a 1:3 resonance. Near the resonance the normal form theory predicts existence of
a small triangle formed by separatrices of a period three hyperbolic point. We prove that in a generic family the
separatrices split, provided that the Stokes constant of the map does not vanish. This constant describes the
distance between the analytical continuations of invariant manifolds associated with the degenerate saddle of
the map at the exact resonance. We provide an asymptotic formula which describes the size of the splitting.
The leading term of this asymptotic formula is proportional to the Stokes  constant.
\end{abstract}

\setcounter{tocdepth}{2}
\tableofcontents

\section{Introduction}

\subsection{Splitting of separatrices near the 1:3 resonance}

In this paper we study the splitting of separatrices in an analytic family of area-preserving maps close to 1:3 resonance. 
Such families naturally appear in the stability  analysis of periodic orbits  in a Hamiltonian system with 2
degrees of freedom. The phase space of the Hamiltonian system is of dimension 4. The level sets of the Hamiltonian
functions are invariant and are of dimension 3. A Poincaré section is used to reduce the problem to a family of
area-preserving maps parametrised by the values of the energy (e.g. \cite{arnold78}).

If multipliers of a fixed point of an area-preserving map are not real, then they belong to the unit
circle and the fixed point is called elliptic. In a neighbourhood of an elliptic fixed point, the map
can be formally embedded into an autonomous  Hamiltonian flow. This implies existence of a formal integral
for the original Hamiltonian system. In particular,  if the reduction to the normal form is analytic,
then the Hamiltonian system is integrable. A classification of possible cases and their normal
forms can be found in the book  \cite{AKN06}.

In Appendix 7E of the classical textbook  \cite{arnold78}, Arnold states that Hamiltonian systems
with resonant periodic orbits are in general non-integrable. He uses a Hamiltonian system with a
periodic orbit at 1:3 resonance as an example. Arnold conjectures that, in spite of the prediction
of the normal form theory, generically the separatrices of the hyperbolic points do not coincide but
intersect transversely. He also states that the distance between the separatrices has to be exponentially
small, since otherwise the normal form theory would be able to detect it.

Let $U$ be a neighbourhood of the origin in $\C^2$ and $V$ be a neighbourhood of the origin in $\mathbb C$.
Consider an analytic family of area-preserving maps $g_\mu:U\to\mathbb C^2$ defined for all $\mu\in V$.
We assume that the function $g_\mu$ is real when its arguments are real (including the parameter $\mu$).
Suppose that the origin is a fixed point of $g_0$, i.e. $g_0(0)=0$, and that $\lambda_0^\pm=\me^{\pm 2\pi\ii/3}$
are the eigenvalues of the Jacobian matrix $g_0'(0)$.
We say that $g_0$ is an \textit{area-preserving map at 1:3 resonance}.

The implicit function theorem implies that the fixed point of $g_\mu$ depends analytically on $\mu$ and
we can move it to the origin by an analytic change of coordinates. Therefore, without loss of generality
we assume that $g_\mu(0)=0$ without loosing in generality. We will see that when $\mu\ne0$ a hyperbolic
3-periodic orbit appears close to the origin. Let $G_\mu = g_\mu^3$ and let $\lambda_\mu^\pm$ denote the
eigenvalues of $G_\mu'$ evaluated on the periodic orbit, with $\lambda_\mu^-<1<\lambda_\mu^+$.

The Birkhoff normal form theorem \cite{Bir66} implies that there is a formal canonical change of coordinates $\Phi$
such that the map $N_\mu=\Phi\circ g_\mu \circ \Phi^{-1}$ commutes with the rotation
$R_{2\pi/3}$, i.e. $N_\mu\circ R_{2\pi/3} = R_{2\pi/3} \circ N_\mu$. The map $N_\mu$ is called
\textit{Birkhoff normal form} of $g_\mu$. Since the map $R_{4\pi/3} \circ N_\mu$ is tangent to identity,
there  is a formal Hamiltonian $H_\mu$ such that
\begin{align*}
 N_\mu = R_{2\pi/3} \circ \phi^1_{H_\mu},
\end{align*}
where $\phi^1_H$ is a formal time one flow of the Hamiltonian vector field that corresponds to the Hamiltonian $H$.
The corresponding  Hamiltonian vector field is usually called \textit{Takens normal form}, \cite{Tak73}.
The Hamiltonian $H_\mu$ is a formal integral of the map $N_\mu$. Reverting back to the original coordinates
we get a formal integral of the map $g_\mu$. This implies that if $g_\mu$ is not integrable then at least one
of the formal series $\Phi$ and $H_\mu$ has to be divergent.

We note that the formal series $H_0$ do not contain neither linear nor quadratic terms and they are $\mathbb Z_3$-symmetric,
i.e., $H_0\circ R_{2\pi/3}=H_0$. We say that $g_0$  is \textit{non-degenerate} if the terms of order 3 in
the Hamiltonian $H_0$ do not vanish.  The symmetry of $H_0$ implies that vanishing of the third order terms
is a condition of co-dimension two  \cite{Gel2009}.

If $\left|\frac{d\lambda_\mu^\pm}{d\mu}\right|_{\mu=0}\ne0$ we say that $g_\mu$ unfolds the 1:3 resonance {\em generically}. 
The level lines of a normal form Hamiltonian $H_\mu$ are sketched on  Figure \ref{fig-normal_form_flow}. 
Of course, the series cannot be assumed convergent, so the sketch represents a truncated series.
\begin{figure}[t]
	\centering
	\begin{subfigure}[b]{0.31\textwidth}
		\centering
		\includegraphics[width=\textwidth]{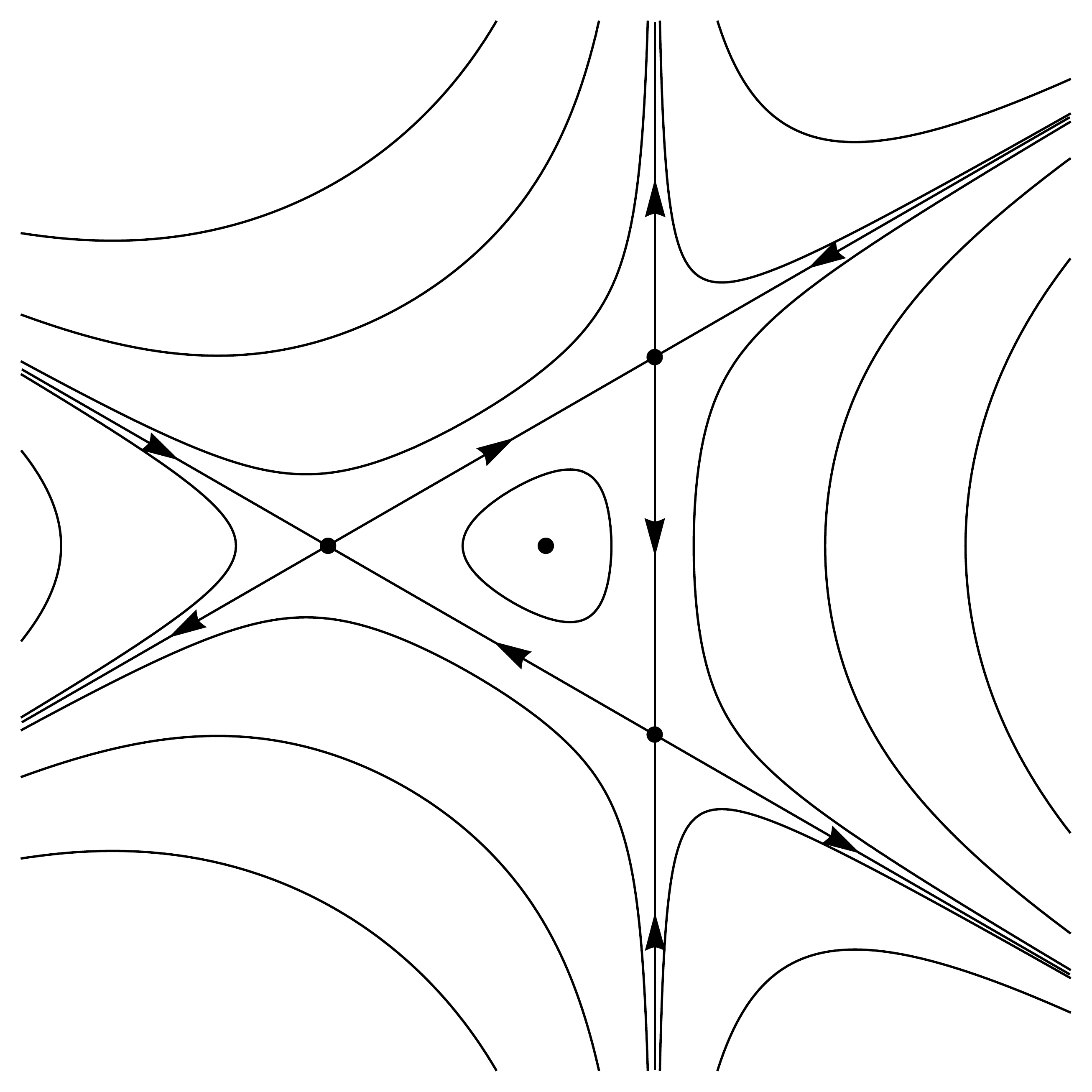}
	\end{subfigure}
	~
	\begin{subfigure}[b]{0.31\textwidth}
		\centering
		\includegraphics[width=\textwidth]{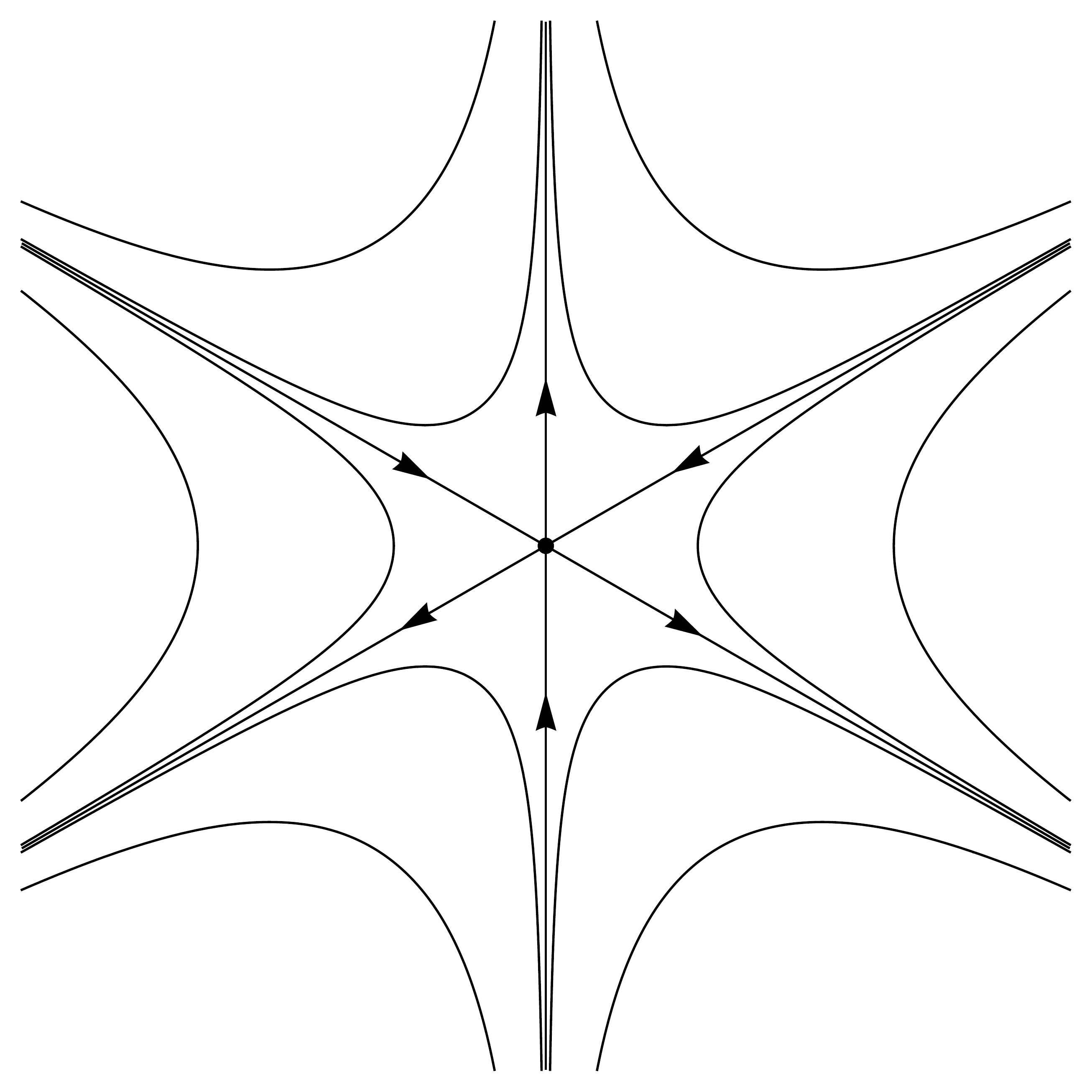}
	\end{subfigure}%
	~
	\begin{subfigure}[b]{0.31\textwidth}
		\centering
		\includegraphics[width=\textwidth]{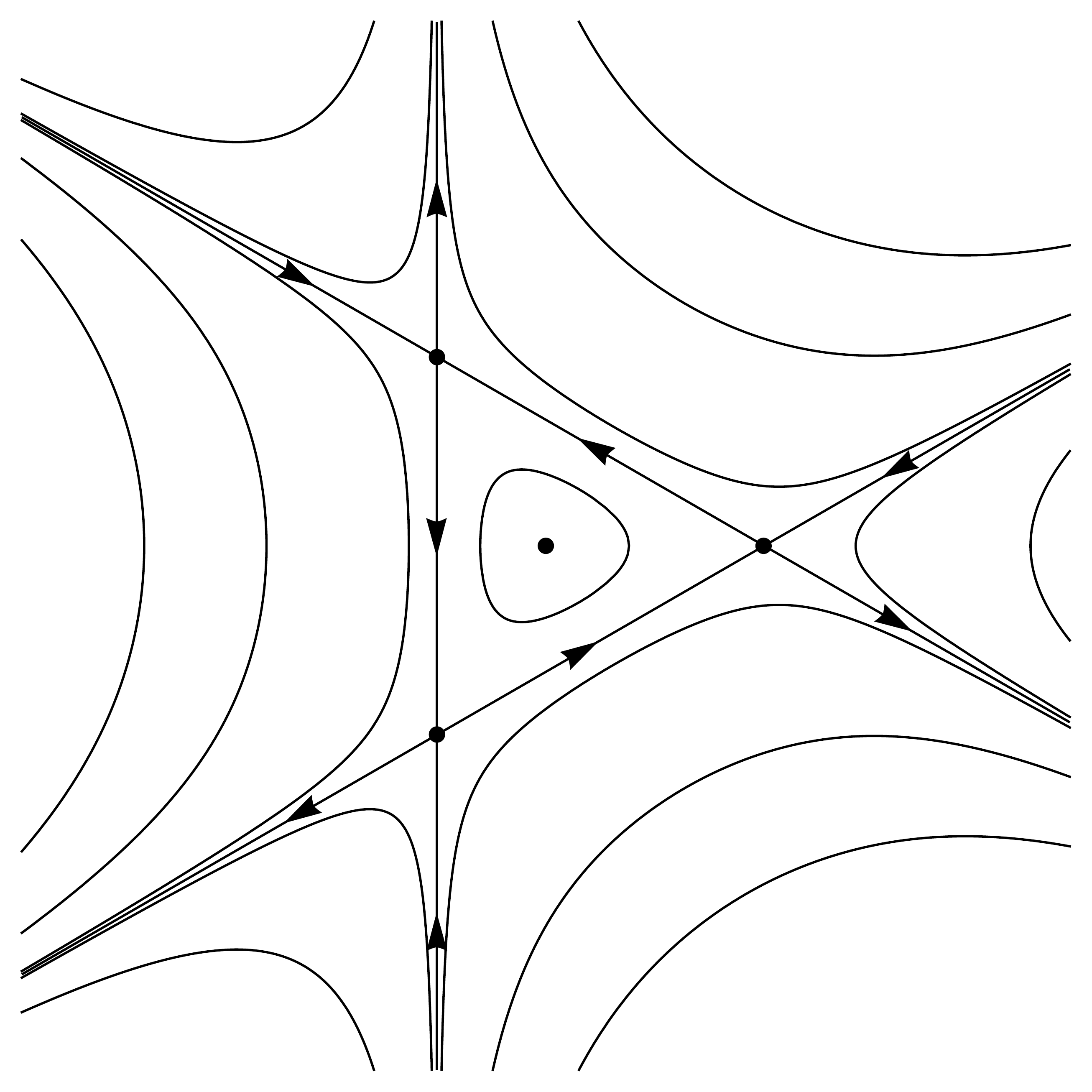}
	\end{subfigure}%
	\caption{The normal form of the generically unfolded 1:3 resonant map.}
	\label{fig-normal_form_flow}
\end{figure}
We see that close to the resonance there exists a 3-periodic orbit with homoclinic connections. The
separatrices of this orbit define an invariant triangle. In the present paper we prove that if
the Stokes constant does not vanish, then the homoclinic connections of the normal form do not
exist in the dynamics of the original map and the separatrices intersect in a non-trivial way, as
shown in Figure~\ref{fig-split_triangle}.

In order to measure the splitting of the separatrices we will use \textit{Lazutkin's homoclinic
invariant}, introduced in \cite{GLS94}. Notice that the points of intersection of the two separatrices
form a homoclinic orbit. We define the homoclinic invariant to be the area of the parallelogram that
is formed by the tangent vectors of the separatrices at a homoclinic point.

\begin{figure}[t]
  \centering
  \includegraphics[width=0.5\textwidth]{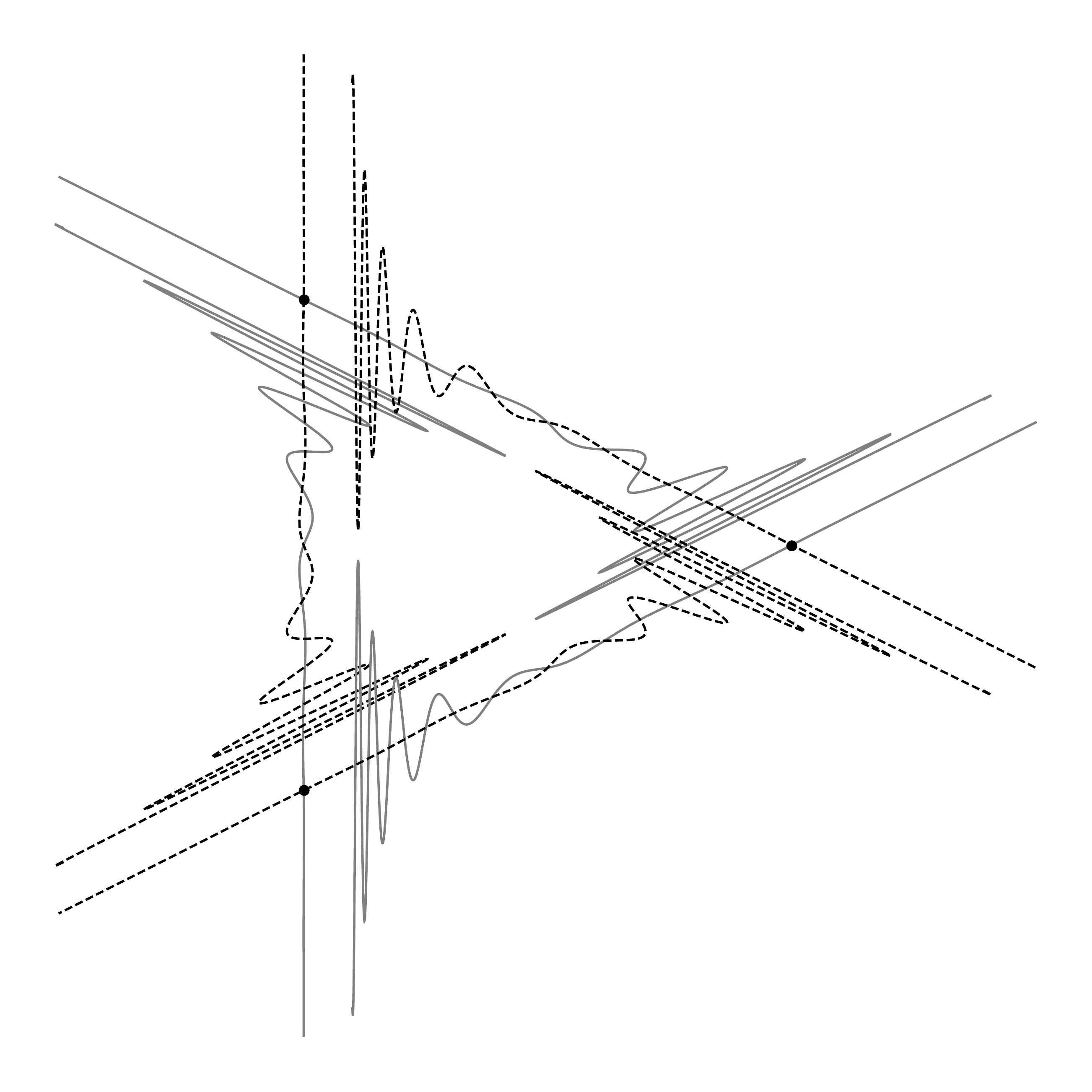}
  \caption{The splitting of the separatrices of a map close to 1:3 resonance.}
  \label{fig-split_triangle}
\end{figure}

The origin is a degenerate saddle of the map $g_0$ and its stable
and unstable manifolds can be parametrised by real analytic functions $V^+_0(t)$ and $V^-_0(t)$.
These functions share an asymptotic expansion in a sectorial neighbourhood of infinity and
satisfy the equation
$$V^\pm_0(t+1)=G_0(V^\pm_0(t)),$$
where $G_0=g_0^3$.  
We will call these two separatrices a \textit{compatible pair of separatrices of the map $G_0$}. 
Then the limit $$ \theta = \lim_{\im t\to-\infty} \me^{2\pi\ii t}\, \omega\Big( V^+_0(t)-V^-_0(t),
\dot{V}^-_0(t) \Big) $$ exists and it is called the \textit{Stokes constant} of the map $G_0$.
In this equation $\omega$ denotes the standard symplectic form in $\mathbb C^2$.
 
Let $\varepsilon\ne0$ be real and $v_s$, $v_u$ be two of the three saddle points of the map $G_\varepsilon$.
Moreover, let $V^+(\varepsilon,t)$ be the parametrization of the stable separatrix of the point $v_s$ and
$V^-(\varepsilon,t)$ be the parametrization of the unstable separatrix of the point $v_u$, such that
$V^+(0,t)$ and $V^-(0,t)$ form a compatible pair of separatrices of the map $G_0$. Let $t_h\in\R$ be
such that $V^+(\varepsilon,t_h)=V^-(\varepsilon,t_h)$, i.e. the point $V^+(\varepsilon,t_h)$ is on a
homoclinic trajectory of the map $G_\varepsilon$. We define the Lazutkin homoclinic invariant by
\[ \Omega = \omega(\dot{V}^+(\varepsilon,t_h),\dot{V}^-(\varepsilon,t_h)).\]

The main result of the paper is the following theorem.

\begin{theorem}
	If $g_\mu$ is an analytic family of area-preserving maps which generically unfolds
	a non-degenerate 1:3 resonance and its Stokes constant $\theta$ does not vanish, then 
%
there exist $\mu_0,c_0>0$ and real constants $\vartheta_n$ such that for any $\mu\in(-\mu_0,\mu_0)\backslash\{0\}$
the map $g_\mu$ has a unique hyperbolic periodic orbit of period three located in a neighbourhood of the origin
and exactly two primary homoclinic orbits to the hyperbolic periodic orbit. Moreover,
 for any $M\in\N$ the Lazutkin homoclinic invariant of the homoclinic orbits
are equal to
$$\Omega(\mu) = \pm\Bigg( \sum_{n=0}^M \vartheta_n (\log \lambda_\mu )^n  + O\big((\log \lambda_\mu )^{M+1}\big) \Bigg)\me^{-\frac{2\pi^2}{\log \lambda_\mu}},$$
where $\vartheta_0 = 4 \pi |\theta|$
and 
$\lambda_\mu=c_0|\mu|+O(\mu^2)$ is the largest multiplier of the hyperbolic periodic point.
\label{thm_result_unfolded}
\end{theorem}

A non-vanishing homoclinic invariant implies that the tangent vectors on the separatrices at the homoclinic
point are not parallel, which of course means that the intersection is transversal. It leads to the usual
corollaries of the existence of a horseshoe, positive topological entropy and divergence of the series in
the normal form theory. The transversality depends on a non-vanishing Stokes constant. Of course, the
Stokes constant may vanish, for example,  if the maps are integrable. On the other hand we think that
generically the Stokes constant does not vanish. For a given map, the Stokes constant can be evaluated
numerically.

The Poincaré maps obtained from  a generic analytic Hamiltonian system with two degrees of freedom
satisfy the non-degeneracy assumptions of our theorem. Consequently, the theorem
implies the transversality  of invariant manifolds for the flows, provided the Stokes constant does not
vanish.

The advantage of the homoclinic invariant is in its invariance with respect to canonical coordinate changes.
In particular, it can be computed after the map is transformed to its normal form up to any finite
order. The angle between the separatrices at a homoclinic point is not invariant and 
 cumbersome computations are required to trace its value when a coordinate change is applied.
 On the other hand, the lobe area is invariant but its positivity does not imply the transversality of the separatrices.
 We point out that both the homoclinic angle and the lobe area can be computed by the methods
 used in the proof of the main theorem.

\subsection{Historical remarks}

The phenomenon we study here was first observed by the French mathematician Henri Poincar{\'e}
around 1890 when investigating the stability of the solar system. Poincar{\'e} considered the
system formed by three bodies: Sun, Earth and Moon, under the action of Newton’s laws of gravity.
In an attempt to prove the stability of the three body system, he used perturbation series and
realized its divergent character \cite{poi1890}. He also noticed that a small differences in the
initial positions or velocities of one of the bodies would lead to a radically different state when
compared to the unperturbed system, what is now commonly known as deterministic chaos. Poincar{\'e}
realized that a small perturbation can destroy a homoclinic connection and its place is taken by
a region where the stable and the unstable manifolds intersect in a highly non-trivial way. He was
even able to prove for a concrete example that the width of this region was exponentially small
with respect to the size of the perturbation.

This splitting of separatrices is the phenomenon we are interested in here. We will give
some brief historical remarks and we encourage the reader to see the survey by Gelfreich and
Lazutkin \cite{GL01} for a more detailed exposition of the theory until 2000.

\subsubsection{Non-autonomous perturbation of flows}

One way to address the above question of stability is to embed the first return map into
the flow of a non-autonomous Hamiltonian system of one degree of freedom. This enables the
usage of methods developed for differential equations and more results are available. This
flow can be written as a periodic time dependent perturbation of a one degree of freedom
Hamiltonian system. More precisely, we describe this system with the help of the Hamiltonian
function
$$H(\mu,\varepsilon;x,y,t) = H_0(x,y) + \mu \,H_1(x,y,\tfrac{t}{\varepsilon},\varepsilon,\mu),$$
with $H_1(x,y,t,\varepsilon,\mu)$ periodic in $t$.

The natural question in this setting is whether homoclinic or heteroclinic connections that
exists in the unperturbed system persist in the perturbed one.

The case where only $\mu$ is considered to be a small parameter was solved using the so called
Melnikov method, see \cite{Mel63}. In this case we can reparametrise time such that $\varepsilon=1$.
Then for the separatrices of the system it holds
\begin{align*}
 W_{\mu,\varepsilon}^\pm(t_0,t) = W_0(t-t_0) + \mu\,W_1^\pm(t_0,t)+O(\mu^2), \quad \pm t\in[t_0,\infty).
\end{align*}
We define the \textit{Melnikov function} by
$$M(t_0):=\int_{-\infty}^\infty \{H_0,H_1\}|_{W_0(t-t_0),t} \dd t , $$
where $\{H,G\}$ is the Poisson bracket of $H$ and $G$. For the difference between the two
separatrices at $t_0$, measured in a coordinate system that uses $H_0$ as the first of its
coordinates, we get that the difference in the first component is
$$ d(t_0) = W_1^+(t_0) - W_1^-(t_0) = M(t_0) + O(\mu).$$
However, when both $\mu$ and $\varepsilon$ are considered small, then $\varepsilon$ cannot be
ignored. These systems are called \textit{rapidly forced} systems since the period of the
perturbation becomes arbitrarily small.

Nekhoroshev, \cite{Nekh77}, showed that in many degrees of freedom Hamiltonian systems, the phase
space can be covered by domains where the system behaves as if it was integrable for some time.
He showed that this time is exponentially large with the size of the perturbation. Neishtadt
showed in \cite{Neis84}  that $d$ actually admits an upper bound that is exponentially small
with $\varepsilon$. Neishtadt's results were refined by Treshchev in \cite{Treshchev97}.
Fontich showed using Lazutkin's ideas that the exponent depends on the location
of the singularities in the parameter of the unperturbed separatrix, \cite{Fon95}.

In rapidly forced systems the Melnikov function can become exponentially small with $\varepsilon$,
but since the error term is polynomially small in $\mu$, the error can become bigger than the
approximation. This situation can be avoided when $\mu$ is a function of $\varepsilon$ which
decreases exponentially as $\varepsilon$ goes to 0. Then the error is also exponentially small
and the Melnikov method can still be applied. It was shown in \cite{gel97} and \cite{DS97} that
in systems where $H_0(x,y)=y^2/2+V'(x)$ and $V''(0)\ne0$ the dependence of $\mu$ on $\varepsilon$
can be relaxed to a polynomial one, $|\mu|\les C \varepsilon^p $, with $p$ big enough. A similar
result was proved in \cite{BF04} for such systems but with $V''(0)=0$.

Stronger results have been proved in specific systems. Poincar{\'e}, \cite{Poi93}, discovered the
phenomenon of splitting by looking at the system described by the Hamiltonian
$$\frac{y^2}{2} + \cos x + a \sin x \cos \frac{t}{\varepsilon}. $$
He proved that in this system the splitting is exponentially small and he derived the first
term of the asymptotic expansion. Poincar{\'e}'s arguments require $a$ to be exponentially
small in $\varepsilon$ and his result is the same as the one that the Melnikov method provides. However,
for an $\varepsilon$-independent $a$, the  Melnikov's method provides a wrong estimate.
Treshchev \cite{Treshchev96} and Gelfreich \cite{Gel97a} independently showed this by obtaining
a different asymptotic formula using the averaging method with a continuous parameter.

The most studied system has been the rapidly perturbed pendulum with a perturbation only depending on time,
$$\ddot x = \sin x + \mu \varepsilon^\eta \sin\frac{t}{\varepsilon}.$$
Many authors have published on this, gradually strengthening the result, see
\cite{HMS88,Sch89,DS92,Ang93,EKS93,Gel94,Swa95} for results using the Melnikov method and \cite{Tre97,Gel00,GOS10}
for results beyond the Melnikov method.

Recently Gaiv{\~a}o and Gelfreich \cite{GaiGel11} used the generalized Swift-Hohenberg equation as an
example to show the transversality of the homoclinic solutions near a Hamiltonian-Hopf bifurcation.

Baldoma, Fontich, Guardia and Seara \cite{BFGS12} showed that in systems where $H_0 = \frac{y^2}{2} + V(x)$
with $V$ an algebraic or trigonometric polynomial and $|\mu| \les C \varepsilon^\eta$, the Melnikov method
can be applied if $\eta>0$. Moreover, they also showed that the Melnikov method fails when $p$ becomes
zero and they derived the first term of the asymptotic series in this case. Guardia \cite{Gua13} showed
that this result also holds close to the resonances of this Hamiltonian.

\subsubsection{Splitting of separatrices in area-preserving maps}

The other way to address the above question of stability of periodic orbits in two degrees
of freedom systems is to look directly at the first return map.

Historically the first map to be treated was the Chirikov standard map, defined on the torus by
\begin{align*}
 \Vector{x}{y}\mapsto \Vector{x+y+\varepsilon\,\sin(x)}{y+\varepsilon\,\sin(x)}.
\end{align*}
For $\varepsilon=0$ the standard map is integrable but for $\varepsilon>0$ the homoclinic
separatrix splits. An asymptotic formula for the splitting in the standard map was published
by Lazutkin in 1984 in a pioneering article, see \cite{Laz84} for the english translation.
However the proof was incomplete and it was completed and published by Gelfreich in \cite{Gel99}.
The same asymptotic formula was derived by Hakim and Mallick, \cite{HM93}, using Ecalle's
theory of resurgent functions. However their work was based on formal arguments.

Neishtadt, \cite{Neis84}, was the first to prove that the splitting in the difference of
the two separatrices of analytic maps close to identity admits an exponentially small
upper bound. Later Fontich and Sim\'o, \cite{FS90}, using Lazutkin's methods gave a
sharp upper bound.

Using the theory of resurgent functions, Gelfreich and Sauzin proved for an instance of
the H{\'e}non map at 1:1 resonance that the splitting of separatrices is exponentially
small and provided the first asymptotic term for it, see \cite{GS01} for more details.

More recently Mart\'in, Sauzin and Seara have studied the splitting of separatrices in
perturbations of the McMillan map, see \cite{MSS11a,MSS11b}. Their approach combined the
theory of resurgent functions with Lazutkin's original ideas.

A paper, \cite{Gel02}, was published by Gelfreich stating the first asymptotic term for
the resonances 1:1, 1:2 and 1:3 without proof in 2002. The only proof on these results published until
now is a preprint by Br{\"a}nnstr{\"o}m and Gelfreich  \cite{BG08}. There the authors derive
and prove the asymptotic formula for area-preserving maps near a Hamiltonian saddle-centre bifurcation.

\subsubsection{Splitting of separatrices in physics}

The same phenomenon has been studied in physics although in a different framework. The technique
conists of truncating an asymptotic series in the optimal order and then showing that the remainder
is exponentially small. This technique is called asymptotics beyond all order or superasymptotics,
see \cite{LST90,Berry91,IL04,IL05}.

There exist many examples of problems for which asymptotic power series methods lead to divergent
series. Oppenheimer \cite{oppenheimer1928three} while investigating a phenomenon in quantum physics
known as the Stark effect, demonstrated that the lifetime of a certain quantum state was inversely
proportional to a quantity exponentially small with the strength of the electric field applied at
the system.

Kruskall and Segur \cite{kruskal1991asymptotics} demonstrated that the geometric model for dendritic
crystal growth fails to produce needle crystal solutions due to exponentially small effects,
a byproduct of the breakage of a heteroclinic connection. This work has influenced many others
in the field and the same technique has been applied at the formal level to prove the non-persistence
of homoclinic or heteroclinic solutions to certain singularly perturbed systems. Examples of
application of this method include surface tension and wave formation
\cite{grimshaw1995weakly,yang1997asymmetric,tovbis2000breaking,van2009vanishing,wavelessship}),
crystal growth \cite{chapman2005exponential}, and optics \cite{chapman2009exponential}. More information
about applications of exponentially small splitting to mechanics, fluids and optics can be found
in the survey of Champneys \cite{champneys1998homoclinic}.

In his book \cite{lombardi2000oscillatory}, Lombardi puts the superasymptotics into rigorous
arguments that can be used to solve many problems in exponentially small phenomena. He did
that by reducing the problem to the study of certain oscillatory integrals which describe the
exponentially small terms. He applied his method to water waves.

\subsection{Unique normal form for families of maps}
\label{ch_normal_form}

The formal Hamiltonian $H$ constructed for the normal form is not defined uniquely so there
is room for further normalization.

\begin{proposition}[\cite{Gel2009}]
 Let $f_0$ be a non-degenerate area-preserving map at 1:3 resonance. Then there is a
 formal Hamiltonian H and formal canonical change of variables which conjugates $f_0$
 with $R_{2\pi/3} \circ \phi^1_H$. Moreover, $H$ has the following form:
 \begin{equation}
  H(x,y) = (x^2+y^2)^3\, A(x^2+y^2)+(2x^3-6x y^2)\, B(x^2+y^2),
  \label{eq_second_norm_hamiltonian}
 \end{equation}
 where $A$ and $B$ are series in one variable with real coefficients
  \begin{equation*}
  A(I)=\mathop{\sum_{k\ges0}}_{k\ne2\mod3} a_k I^k,\;\;\;\;\;\;\;\;\; B(I)=\frac{b_0}{6}
  + \mathop{\sum_{k\ges1}}_{k\ne2\mod3} b_k I^k
 \end{equation*}
  and the coefficient of $A$ and $B$ are uniquely defined if $b_0\ne0$.
 \label{thm_form_of_res_map}
\end{proposition}

For the coefficient $b_0$ it holds $b_0=6|h_{30}|$, where $h_{30}$ is the 3rd order coefficient
in the Takens normal form Hamiltonian.

For a real analytic family of maps, $f_\mu$, such that $f_0$ is a map at 1:3 resonance
we have the following result.

\begin{proposition}[\cite{Gel2009}]
 \label{thm_normal_form}
 Let $f_\mu$ be an analytic family of area-preserving maps such that $f_0$ is at
 resonance 1:3 and let the coefficient $b_0$ for the map $f_0$ not vanish. Then
 there is a formal Hamiltonian H and formal canonical change of variables which
 conjugates $f_\mu$ with $R_{2\pi/3} \circ \phi^1_H$. Moreover, $H$ has the
 following form:
 \begin{align*}
  H(\mu;x,y) = (x^2+y^2)\, A(\mu,x^2+y^2)+(2x^3-6x y^2)\, B(\mu,x^2+y^2),
 \end{align*}
 where $A$ and $B$ are series in two variables with real coefficients
 \begin{align*}
  A(\mu,I)=\mathop{\sum_{k,m\ges0}}_{k\ne1\mod3} a_{k,m} I^k \mu^m,\;\;\;\;\;\;\;\;\;
  B(\mu,I)=\frac{b_{0,0}}{6} + \mathop{\sum_{k,m\ges1}}_{k\ne2\mod3} b_{k,m} I^k \mu^m,
 \end{align*}
 with $b_{0,0}=b_0$ and $a_{0,0}=a_{1,0}=0$. Moreover the coefficients of these series are unique.
\end{proposition}

\begin{remark}
 Since the formal Hamiltonian $H_\mu$ is invariant under the rotation $R_{2\pi/3}$ we have that
 $N^3 = (R_{2\pi/3} \circ \phi^1_H)^3 = \phi^3_H = \phi^1_{3H} $. Then from the relation
 $N_\mu^3=\Phi_\mu \circ f_\mu^3 \circ \Phi_\mu^{-1}$ we see that the third iterate of the
 map $f_\mu$ can be described by the same normal form as $f_\mu$.
\end{remark}

\begin{remark}
 The fact that the map unfolds the 1:3 resonance generically,
 $\left.\frac{d\lambda_\mu^+}{d\mu}\right|_{\mu=0}\ne0$, implies that $a_{0,1}\ne0$.
\end{remark}

\subsection{Stokes phenomenon at the resonance}

Let $U\subset\C^2$ be a neighbourhood of the origin. Let $F_0 = f_0^3$.
\begin{theorem}[\cite{GM17}]
 Let $f_0:U\to\mathbb C^2$ be a non-degenerate area-preserving map at 1:3 resonance
 such that its Taylor expansion agrees with its normal form   up to order 4. 
  Then there exists a unique formal series of the form 
 \begin{align*}
  \formalSeparatrixSing_0(t)=\Vector{0}{-\frac{1}{b_0 \, t}}+O(|t|^{-3}) \in \frac{1}{t} \R^2[[\frac{1}{t}]],
 \end{align*}
 which satisfies the equation 
 \begin{align}
  \formalSeparatrixSing_0(t+1) = F_0(\formalSeparatrixSing_0(t)).
  \label{eq_res_separatrix_equation}
 \end{align}
 Any other formal solution of the form $(0,-\frac{1}{b_0 \, t})+O(|t|^{-2})$ can
 be written as $\formalSeparatrixSing_0(t+c)$ for some $c\in\C$. Moreover, there
 exists a formal series, $\tilde{\Xi}_0\in t^2 \R[[\frac{1}{t}]]^2$, with real coefficients
 which satisfies the equations
 $$ \tilde{\Xi}_0(t+1)=F'_0(\formalSeparatrixSing_0(t))\cdot \tilde{\Xi}_0(t),$$
 \begin{align*}
  \tilde{\Xi}_0(t)=\Vector{b_0 \,t^2-18\frac{ b_1}{b_0^2} +\frac{24 b_1^2}{b_0^5} \,
  t^{-2}}{-\frac{8 a_0}{b_0^3}\,t^{-1} }+O(|t|^{-3}),
 \end{align*}
 and\footnote
 {
 Notice that $\dot{\formalSeparatrixSing}_0$ satisfies that same equation.
 } $\det(\tilde{\Xi}_0(t),\dot{\formalSeparatrixSing}_0(t))=1$.
 
 There exist two unique real-analytic solutions of the equation \eqref{eq_res_separatrix_equation},
 $W^+_0$ and $W^-_0$  such that $\lim_{t\to\pm\infty}W^\pm_0=0$, that admit
 $\formalSeparatrixSing_0$ as asymptotic in $D^\pm$, where $D^+$ is
 a sectorial neighbourhood of infinity of opening bigger than $\pi$ symmetric
 around the real axis and $D^-=-D^+$.  Moreover, there exist two complex constants,
 $\theta$ and $\rho$, such that
 \begin{align}
  W^+_0(t)-W^-_0(t) \asymp e^{-2\pi\ii t}\left( \theta\, \tilde{\Xi}(t) + \rho\, \dot{\formalSeparatrixSing}(t)
  \right) + O(t^7 e^{-4\pi\ii t})
  \label{eq_res_var_eq}
 \end{align}
 for large enough $t$ in the semistrip $\im t <0$, $|\re t| <2$
 and consequently
 $$\theta=\lim_{t\to+\infty}\me^{2\pi t}\omega(W^+_0(-\ii t)-W^-_0(-\ii t),\dot{W}^-_0(-\ii t)).$$
 \label{thm_existence_Borel_transform}
\end{theorem}

Even though the constants $\theta$ and $\rho$ have similar roles, we will focus on $\theta$ as it describes the normal
component of the separatrix splitting. We call $\theta$  \textit{a Stokes constant of $F_0$}.

\begin{lemma}
 Let $g_0$ and $f_0$ be two non-degenerate, real analytic, area-preserving maps at 1:3
 resonance. Let $G_0 = g_0^3$ and $F_0 = f_0^3$ and let $\Phi$ be an analytic symplectic
 transformation such that $G_0=\Phi^{-1}\circ F_0 \circ\Phi$. Let $\theta_G$ and $\theta_F$
 be the Stokes constants of the maps $G_0$ and $F_0$ respectively. Then $|\theta_G| = |\theta_F|$.
 \label{thm_stokes_const_invariant}
\end{lemma}

\begin{proof}
 Let $W^\pm_0(t)$ be parametrizations of the invariant manifolds of the map $F_0$, that satisfy the equation 
 $$W^\pm_0(t+1) = F_0 (W^\pm_0(t))$$
 and let $V^\pm_0(t)$ be parametrizations of the invariant manifolds of the map $G_0$, that satisfy the equation 
 $$V^\pm_0(t+1) = G_0 (V^\pm_0(t)).$$
 Then $V^\pm_0(t) = \Phi\circ W^\pm_0(t+c)$ for some $c\in\R$.
 
 We can choose the functions $V^\pm_0$ to be real analytic and then $c$ is a real constant. 
 We define $\delta_0(t) = W^+_0(t) - W^-_0(t)$ and we get 
 \begin{align*}
  &\omega\Big( V^+_0(t) - V^-_0(t), \dot{V}^-_0(t) \Big)  = \\
  &\quad = \omega\Big( \Phi\circ W^+_0(t+c) - \Phi\circ W^-_0(t+c),
  \Phi'\circ W^-_0(t+c) \cdot \dot{W}^-_0(t+c) \Big)  \\
  &\quad = \omega\Big( \Phi'\circ W^-_0(t+c) \cdot \delta_0(t+c) + O(\delta_0(t+c)^2),
  \Phi'\circ W^-_0(t+c) \cdot  \dot{W}^-_0(t+c) \Big) \\
  &\quad = \omega\Big(\delta_0(t+c),\dot{W}^-_0(t+c) \Big) + O(\delta_0(t+c)^2). \\  
 \end{align*}
 We know that $\delta_0(t) = O(\me^{-2\pi\ii t})$, so the term $O(\delta_0(t+c)^2)$ vanishes at the limit. This gives
 \begin{align*}
  \theta_G =& \lim_{\im t\to-\infty}\me^{2\pi\ii t} \omega\Big( V^+_0(t) - V^-_0(t), \dot{V}^-_0(t) \Big)\\
  =& \me^{-2\pi\ii c} \lim_{\im t\to-\infty}\me^{2\pi\ii (t+c)} \omega\Big(\delta_0(t+c) ,\dot{W}^-_0(t+c) \Big) \\
  =& \me^{-2\pi\ii c} \;\theta_F. \qedhere
 \end{align*}
\myqed
\end{proof}

\subsection{Notation and outline of the proof}

Consider a family of maps $g_\mu$ which satisfies the assumptions of Theorem \ref{thm_result_unfolded}.
Before proceeding with the proof we fix $M\in\N$ and make make a canonical change of coordinates which
transforms the family to $f_\mu$ for which the Taylor series coincide with the normal form of Proposition
\ref{thm_normal_form} up to the order
\begin{equation}
 N=6M+39.
 \label{eq_N_definition}
\end{equation}
As we saw in Lemma \ref{thm_stokes_const_invariant}, if the Stokes constant of $g_\mu$ does
not vanish then the Stokes constant of $f_\mu$ also does not vanish. Moreover since the
Lazutkin constant is a symplectic invariant, it is sufficient to prove Theorem \ref{thm_result_unfolded}
for the map $f_\mu$. This proof  takes up the rest of the paper.

We define $F_\mu=f_\mu^3$ and by choosing a small $\mu\ne0$ we get three saddle points with
$\lambda_\mu>1$ being their largest eigenvalue. Then we define a new parameter $\varepsilon = \log(\lambda_\mu)>0$.
The implicit function theorem tells us that we can write the parameter $\mu$ as a function of $\varepsilon$. 
We use $f_\varepsilon$ to denote the map written in terms of the new coordinates and the new parameter.

Note that by definition, $\varepsilon$ is always positive, however this assumption does not
restrict the genericity of the result as we can consider the cases $\mu\in(0,\mu_0)$ and
$\mu\in(-\mu_0,0)$ separately.

We define $F_\varepsilon = f_\varepsilon^3$. Since the map $F_\varepsilon$ is analytic
in $x$, $y$ and $\varepsilon$, it can be decomposed in two ways, namely
$ F_\varepsilon(x,y)=\sum_{n\ges0} \varepsilon^n F_n(x,y)$ and
$ F_\varepsilon(x,y)=\sum_{n\ges1}  \mathcal{F}_n(\varepsilon;x,y)$. Here $F_n$ are real analytic functions
independent of $\varepsilon$ and $\mathcal{F}_n$ are polynomials of degree $n$ homogeneous in its three variables.

Let $\tilde{H}(\varepsilon;x,y)$ be the formal Hamiltonian of the normal form of
Proposition \ref{thm_normal_form} for $F_\varepsilon$. Let  $\tilde{F}_\varepsilon$  be  its formal time-1 flow.
Similarly, we define $ \tilde F_\varepsilon(x,y)=\sum_{n\ges0} \varepsilon^n \tilde F_n(x,y)$
and $ \tilde F_\varepsilon(x,y)=\sum_{n\ges1}  \tilde{\mathcal{F}}_n(\varepsilon;x,y)$,
where $\tilde F_n(x,y)$ are formal series in $x$ and $y$ and $\tilde{\mathcal{F}}_n$
are homogeneous polynomials. Moreover $\mathcal{F}_n = \tilde{\mathcal{F}}_n$ for all $n\les N$.

The central objects in this analysis are the functions $W^-(\varepsilon;\tau)$
and $W^+(\varepsilon;\tau)$ which parametrize the stable and the unstable separatrix
that are close to the vertical separatrices of the normal form
shown in Figure \ref{fig-normal_form_flow}. They satisfy the equation $$
W^\pm(\varepsilon;\tau+1)=F_\varepsilon(W^\pm(\varepsilon;\tau)).
$$ 
Unless it is explicitly stated, it will be assumed from now on that the separatrices
are parametrized with step 1 as above. We fix the parametrization by asking that
$W^+(\varepsilon;0)$ is the point where the stable separatrix meets the horizontal
axis for the first time. Similarly $W^-(\varepsilon;0)$ is the point where the unstable
separatrix meets the horizontal axis for the first time.

There are four formal solutions of the separatrix equation to be considered:
$ \formalSeparatrix $, $\formalFlow$, $\formalSeparatrixSing$ and $\formalFlowSing$.
The first two are formal series in $\varepsilon$ and $\tanh(\tfrac{1}{2}\varepsilon \tau)$.
The former, $\formalSeparatrix$, satisfies the equation
$\formalSeparatrix(\varepsilon;\tau+1)=F_\varepsilon(\formalSeparatrix(\varepsilon;\tau))$
and the latter, $\formalFlow$, satisfies
$\formalFlow(\varepsilon;\tau+1)=\tilde{F}_\varepsilon(\formalFlow(\varepsilon;\tau))$.

The other two formal solutions, $\formalSeparatrixSing$ and $\formalFlowSing$, are formal solutions
``around the singularity  $\pi\ii/\varepsilon$'' and they are formal series in $\varepsilon$ and $t$.
They are obtained by $\formalSeparatrix$ and $\formalFlow$ respectively, by substituting
$\tanh(\tfrac{1}{2}(\varepsilon t - \pi\ii))$ by its Laurent series.
Then $\formalSeparatrixSing$ satisfies $\formalSeparatrixSing(\varepsilon;t+1)= 
F_\varepsilon(\formalSeparatrixSing(\varepsilon;t))$ and
$ \formalFlowSing $ satisfies $ \formalFlowSing(\varepsilon;t+1)=F_\varepsilon(\formalFlowSing(\varepsilon;t))$.

Due to the symmetry of the normal form, $\formalFlow$ has its first component even and its second component
odd in $\tau$.
Similarly $\formalFlowSing$ has its first component even and its second component odd in $t$.

There are two linear equations that play an important role in the proof. These are
\begin{align}
 U(\varepsilon;\tau+1) = A(\varepsilon;\tau) \cdot U(\varepsilon;\tau) , \label{eq_var_eq_01} \\
 V(\varepsilon;\tau+1) = D(\varepsilon;\tau) \cdot V(\varepsilon;\tau) , \label{eq_var_eq_02}
\end{align}
with
$$ A(\varepsilon;\tau)=\int_0^1 F_\varepsilon'\left( s\,W^+(\varepsilon;\tau) + (1-s)\,
W^-(\varepsilon;\tau) \right) \dd s $$
and
$$ D(\varepsilon;\tau) =  F_\varepsilon'\left( W^-(\varepsilon;\tau) \right).$$
Evidently, $\delta = W^+-W^-$ satisfies the first one and $\dot{W}^-$ satisfies the second.
We denote by $U$ the fundamental solution of the first and by $V$ the fundamental solution
of the second, normalized by $ \det U = \det V =1$.
Moreover we require that $ U \cdot (\begin{smallmatrix} 1 \\ 0 \end{smallmatrix} ) = \delta$
and  $ V \cdot  (\begin{smallmatrix} 0 \\ 1 \end{smallmatrix} ) = \dot{W}^-$.

If we look at equations \eqref{eq_var_eq_01} and \eqref{eq_var_eq_02} formally we see that the
formal matrices $\tilde{A}$ and $\tilde{D}$ coincide. The formal equation close to the singularity is
$$ \Vtilde(\varepsilon;t+1) = \tilde{D}(\varepsilon;t) \cdot \Vtilde(\varepsilon;t).$$
We denote by $\Vtilde$ its fundamental solution that satisfies $\det \Vtilde = 1$ and $\Vtilde
\cdot  (\begin{smallmatrix} 0 \\ 1 \end{smallmatrix} ) = \dot{\formalFlowSing}$.

We will see that there exists an open domain in variable $\tau$ that contains the origin and goes
close to the singularities $\pm\frac{\pi}{\varepsilon}\ii$ in which both $U$ and $V$ are analytic
and their difference is small.

Initially, in Section \ref{ch_form_sols}, we prove the existence of the various formal solutions.
In Section \ref{ch_form_sols_approximate} we prove that the formal solution around the singularity
approximates the separatrix close to the singularity. This approach to the proof is called
{\em complex matching\/}, see \cite{GL01}.

In Section \ref{ch_var_eqs} we introduce the solutions to the variational equations and prove
that the formal solution approximates both. In Section \ref{ch_sharper_bounds} we prove that on
the real line we have exponentially small upper bounds.

In Section \ref{ch_asym_exp_of_splitting} we introduce the function
$$\Theta^-(\varepsilon;\tau) = \omega(\delta(\varepsilon;\tau),\dot{W}^-(\varepsilon;\tau)).$$
We will see that $\Theta^-$ is approximately periodic and can be approximated using the formal
solution of the variational equation. This enables us to compute its ``first Fourier coefficient''.
The derivative of this  function at a homoclinic point gives the homoclinic invariant.
We estimate the value of this function  close to the singularity where it is polynomially
small with $\varepsilon$ and finally we translate the result to the real axis where we see
it is exponentially small.

Notice that throughout the proof we treat $\varepsilon_0$ as if it was fixed  but we are allowed
to decrease it if the need arises. We also need to choose a constant $\Lambda>1$ such that
$\Lambda^2\varepsilon_0<1$. We are also allowed to increase $\Lambda$ if the need arises making
sure that $\varepsilon_0$ will be decreased proportionately. The constant $\Lambda$ is used to
fine tune the domains where the separatrices are approximated by the formal solutions.

\section{Formal solutions}
\label{ch_form_sols}

\subsection{Formal separatrix close to the saddle point}
\label{ch_form_sep}

\begin{lemma}
 Let $\sigma=\tanh (\tfrac{1}{2}\varepsilon \tau)$ and let $\tilde{H}(\mu;x,y)$ be a formal
 Hamiltonian as described in Proposition \ref{thm_normal_form}. Then there exist a real
 formal power series $\mu(\varepsilon)=\sum_{n\ges 1} \mu_n \varepsilon^n$ and a real
 formal solution $\formalFlow(\varepsilon;\tau)=(\tilde{x}(\varepsilon;\tau),\tilde{y}(\varepsilon;\tau))$
 of Hamilton's equations\footnote{Here the dot denotes derivation with respect to $\tau$.}
 \begin{equation}
 \begin{aligned}
  \dot{x} &= \,\partial_y \tilde{H}(\mu(\varepsilon);x,y) , \\
  \dot{y} &= -\partial_x \tilde{H}(\mu(\varepsilon);x,y) ,
  \label{eq_norm_form_ham_eq_inner}
 \end{aligned}
 \end{equation}
 such that
 \begin{equation}
 \begin{aligned}
  \tilde{x}(\varepsilon;\tau) = \sum_{n\ges1} \varepsilon^n P_n(\sigma), \\
  \tilde{y}(\varepsilon;\tau) = \sum_{n\ges1} \varepsilon^n Q_n(\sigma),
  \label{eq_norm_formflow_inner}
 \end{aligned}
 \end{equation}
 with $P_n(\sigma)$ even polynomials of degree $2 \lfloor \frac{n}{2}\rfloor$ and $Q_n(\sigma)$
 odd polynomials of degree $2 \lfloor \frac{n+1}{2} \rfloor-1$ and in particular
 $P_1(\sigma)=\frac{1}{2\sqrt{3} b_{0,0}}$, $Q_1(\sigma)=\frac{\sigma}{2 b_{0,0}}$,
 $\mu_1=\frac{1}{2\sqrt{3} a_{0,1}}$. Moreover, for all $n>1$ $P_n$, $Q_n$ and $\mu_n$ are uniquely
 defined.
 \label{thm_formal_flow_existence}
\end{lemma}

Since the series $\tilde{H}$ is invariant under the transformation $(x,y)\mapsto(x,-y)$
we can look for a formal solution of the Hamiltonian equations such that its 
first component is even and the second is odd.

\begin{proof}
We note that $\sigma$ is odd and  $\dot{\sigma}=\frac{1}{2}\varepsilon(1-\sigma^2)$.
Then we substitute the series into the equations and collect the terms at the same order of $\varepsilon$. 
 
 The first term that appears in Hamilton's equations is of order 2. Let $P_1(\sigma)=A_{1,0}$
 and $Q_1(\sigma)=A_{1,1} \sigma$.
 Then we have for terms of order $\varepsilon^2$
 \begin{align*}
  & 0 = 2 b_{0,0} A_{1,0} A_{1,1}  \sigma - 2 a_{0,1} \mu_1 A_{1,1} \sigma, \\
  & \frac{1}{2} A_{1,1} (1-\sigma^2) = b_{0,0}(A_{1,0}^2-A_{1,1}^2 \sigma^2) +2 a_{0,1} \mu_1 A_{1,0} .
 \end{align*}
 From the possible solutions we choose the ones described in the lemma and we let
 $$
   P_n(\sigma)=\sum_{k=0}^{\lfloor\frac{n}{2}\rfloor} A_{n,2k}\, \sigma^{2k} , 
   \qquad
   Q_n(\sigma)=\sum_{k=0}^{\lfloor\frac{n}{2}\rfloor} A_{n,2k+1}\, \sigma^{2k+1}
$$
 for all $n>1$.
 Thus for each $n$ there are $n+2$ coefficients, counting $\mu_n$ as an unknown.
 By taking into account that at the order of $\varepsilon^n$, $P_n$ and $Q_n$
 appear only in the second order terms of the Hamiltonian equations, we find that
 we need to solve a linear system and the matrix $M$ of the system is of the form
  \begin{align*}
   M=\begin{pmatrix}
      A & B \\
      C & D
     \end{pmatrix}.
  \end{align*}
 We have two cases.
 \begin{itemize}
  \item $n=2m$
  
  We order the unknowns by $(\mu_{2m}$, $A_{2m,1},\dots$, $A_{2m,2m-1}$, $A_{2m,0},\dots$, $A_{2m,2m})$
  and we see that the matrices $A,B,C,D$ are $(m+1) \times (m+1)$ matrices
 \begin{align*}
  & A=\begin{pmatrix}
    d_0 & t_1 & & & & \\
    & d_1 & t_2 & & & \\
    & & \ddots & \ddots & \\
    & & & d_{m-1} & t_m \\
    & & & & d_m \\
   \end{pmatrix} , \\
  & B=-\frac{2}{\sqrt{3}} \,\mbox{Id}_{n+1} , \\
  & C=\begin{pmatrix}
    \frac{a_{0,1}}{b_{0,0}} & 0 & \cdots & 0 \\
    0 & 0 & \cdots & 0 \\
    \vdots & \vdots &  \ddots & \vdots \\
    0 & 0 & \cdots & 0 \\
   \end{pmatrix} , \\
  & D=\begin{pmatrix}
    -1 & 1 & & & & \\
    & -2 & 2 & & & \\
    & & \ddots & \ddots & \\
    & & & -m & m \\
    & & & & -m-1 \\
   \end{pmatrix} , \\
 \end{align*}
 with $d_0=\frac{a_{0,1}}{\sqrt{3}b_{0,0}}$ and for $j>0$ $d_j=\frac{1}{2}-j$, $t_j=\frac{1}{2}+j$.
 \item $n=2m+1$
 
  We use an order similar to the above: $(\mu_{2m+1}$, $A_{2m+1,1},\dots$, $A_{2m+1,2m+1}$, $A_{2m+1,0},\dots$, $A_{2m+1,2m})$.
  In this case the matrices $A,B,C,D$ are of dimensions $(m+2) \times (m+2)$,
  $(m+2) \times (m+1)$, $(m+1) \times (m+2)$ and $(m+1) \times (m+1)$ respectively.
  The matrix $D$ is as above and
 \begin{align*}
  & A=\begin{pmatrix}
    d_0 & t_1 & & & & \\
    & d_1 & t_2 & & & \\
    & & \ddots & \ddots & \\
    & & & d_{m} & t_{m+1} \\
    & & & & d_{m+1} \\
   \end{pmatrix} , \\
  & B=\begin{pmatrix}
    -\frac{2}{\sqrt{3}} \,\mbox{Id}_{n+1} \\
    0
   \end{pmatrix} , \\
  & C=\begin{pmatrix}
    \frac{a_{0,1}}{b_{0,0}} & 0 & \cdots & 0 & 0 \\
    0 & 0 & \cdots & 0 & 0 \\
    \vdots & \vdots &  \ddots & \vdots & \vdots \\
    0 & 0 & \cdots & 0 & 0 \\
   \end{pmatrix}.
 \end{align*}
 \end{itemize}
 Then we have\footnote
 {
 This is a direct implication of the equality $$\begin{pmatrix}
                                                 A & B \\
                                                 C & D
                                                \end{pmatrix} =
                                                \begin{pmatrix}
                                                 A-B D^{-1}C & B D^{-1} \\
                                                 0 & I_m
                                                \end{pmatrix}\cdot
                                                \begin{pmatrix}
                                                 I_n & 0 \\
                                                 C & D
                                                \end{pmatrix}. $$
 }
 $\det(M)=\det(A-B D^{-1}C) \det(D)$. Since
 \begin{align*}
  D^{-1}=-\begin{pmatrix}
    1 & \frac{1}{2} & \frac{1}{3} & \cdots & \frac{1}{m} & \frac{1}{m+1} \\
    & \frac{1}{2} & \frac{1}{3} & \cdots & \frac{1}{m} & \frac{1}{m+1} \\
    & & \frac{1}{3} & \cdots & \frac{1}{m} & \frac{1}{m+1} \\
    & & & \ddots & \vdots & \vdots \\
    & & & & \frac{1}{m} &  \frac{1}{m+1}\\
    & & & & & \frac{1}{m+1} \\
   \end{pmatrix} \\
 \end{align*}
 we get
 \begin{align*}
  B D^{-1}C=\begin{pmatrix}
    \frac{2 a_{0,1}}{\sqrt{3}b_{0,0}} & \cdots & 0 \\
    \vdots & \ddots & \vdots \\
    0 & \cdots & 0 \\
   \end{pmatrix},
 \end{align*}
 so $\det(A-B D^{-1}C)\ne0$.
 This means that the matrix is invertible and the linear system can be solved at each order.
 Then the coefficients can be computed inductively up to any degree.
\myqed
\end{proof}

\begin{lemma}
 The formal solution $\formalFlow(\varepsilon;\tau)$ satisfies the equation 
 $$\formalFlow(\varepsilon;\tau+1)=\tilde{F}_\varepsilon(\formalFlow(\varepsilon;\tau)).$$
 \label{thm_formal_flow_satisfies_map}
\end{lemma}

\begin{proof}
 Let $\mathcal{M}$ be the space of formal series of the form $\mu(\varepsilon) = \sum_{n\ges1}\mu_n \varepsilon^n$, with $\mu_n\in\C$.
 Let $\mathcal{A}$ and $\mathcal{B}$ be the spaces of formal series of the form
 \begin{align*}
  A(\mu,I)=\mathop{\sum_{k,m\ges0}}_{k\ne1\mod3} a_{k,m} I^k \mu^m,\quad a_{k,m}\in\C,\, a_{0,0}=a_{1,0}=0
 \end{align*}
 and
 \begin{align*}
  B(\mu,I)=\mathop{\sum_{k,m\ges0}}_{k\ne2\mod3} b_{k,m} I^k \mu^m,\quad b_{k,m}\in\C,
 \end{align*}
 respectively. Let $\mathcal{H}$ be the space of formal series of the form
 \begin{align*}
  H(\varepsilon,x,y) = (x^2+y^2)\, A(\mu(\varepsilon),x^2+y^2) + (x^3-3x y^2) \, B(\mu(\varepsilon),x^2+y^2)
 \end{align*}
 for some $\mu\in\mathcal{M}$, $A\in\mathcal{A}$ and $B\in\mathcal{B}$.
 We define $ \text{val}(I^k \mu^m) = 2k + m$ and we define $\text{val}(A)$ to be the minimum of
 the valuation of all its monomials, we define $\text{val}(B)$ in a similar way. Then we define
 $ \text{val}(H) = \min\{ \text{val}(A)+2, \text{val}(B)+3 \} $.
 
 Let $\mathcal{X}$ be the set of formal series $x=\sum_{n\ges1}P_n(\sigma)\varepsilon^n$ with $P_n$ a polynomial
 of degree at most $n$ and $\sigma=\tanh(\varepsilon\tau/2)$.  Since $\sigma'(\tau) = (1-\sigma^2(\tau))\varepsilon/2 $
 and $\sigma'(\tau+1) = \sum_{n\ges0}\frac{1}{n!}\sigma^{(n)}(\tau) $, if the series $X(\tau)$ is in $\mathcal{X}$,
 then the series $X(\tau+1)$ is also in $\mathcal{X}$.

 We define the valuation of $x\in\mathcal{X}$ to be the smallest non-vanishing order of $\varepsilon$.
 We define a metric on any space of formal series by
 $$ d(x,y) = 2^{-\text{val}(x-y)}. $$
 We define the map $\Phi:\mathcal{H}\times\mathcal{X}^2\to\mathcal{X}^2$, such that for $H\in\mathcal{H}$ and $X\in\mathcal{X}^2 $ $$\Phi(H,X) = X(\tau+1) - \phi_{H}^1(X(\tau)).$$
 Let $H_n$ be the sum of all terms in $H$ with valuation at most $n$. We define $X_n$ in a similar way.
 Then it is easy to see that
 $$ d\big(T(H,X), T(H_{n+1},X_n)\big) \les 2^{-n}. $$
 This implies that $\Phi$ is continuous.
 
 Since $H_n$ is a polynomial, its flow $X_{H_n}\in\mathcal{X}^2$ is analytic.
 Then $\Phi(H_n,X_{H_n})=0$.  Because $\Phi$ is continuous we can take the limit
 $n\to\infty$, which proves the lemma.
\myqed
\end{proof}

We denote $\mathcal{Z}_n(\tau)=(P_n(\sigma),Q_n(\sigma))$, so $ \formalFlow(\varepsilon;\tau)=
\sum_{n\ges1}\varepsilon^n \mathcal{Z}_n(\sigma)$, and $ \formalFlow_n(\varepsilon;\tau)=\sum_{m=1}^n\varepsilon^m
\mathcal{Z}_m(\sigma)$. The standard arguments about the remainder of a Taylor series imply the following statement.

\begin{corollary}
 Let $F_\varepsilon$ be an analytic map with Taylor series that agrees with $\tilde{F}_\varepsilon$ up
 to degree $n$. Then we have $\text{\emph{val}}\big(\formalFlow_n(\varepsilon;\tau+1)-
 F_\varepsilon(\formalFlow_n(\varepsilon;\tau))\big)=n+2$.
\end{corollary}

\begin{corollary}
 There exists a real formal series $\formalSeparatrix$ in $\varepsilon$ such that it satisfies
 the equation $\formalSeparatrix(\varepsilon;\tau+1) = F_\varepsilon(\formalSeparatrix(\varepsilon;\tau)) $
 and its components at the order of $ \varepsilon^n $ are polynomials in $\sigma$ of order at most $n$.
 We define  $ \formalSeparatrix (\varepsilon;\tau)=\sum_{n\ges1}\varepsilon^n Z_n(\sigma) .$
 \label{thm_form_sol_corollary}
\end{corollary}

\begin{proof}
 Since $F_\varepsilon$ and $\tilde{F}_\varepsilon$ are conjugated by a formal series, their respective
 formal separatrices are also conjugated by a formal series and from this the result follows.
\myqed
\end{proof}

Notice that unlike $\formalFlow$, $\formalSeparatrix$ does not have one even and one odd component.

\subsection{Approximation of the separatrix}

For the existence of the two separatrices we have the following theorem.

\begin{theorem}[\cite{BG08}]
 Let $\varepsilon\in(0,\varepsilon_0)$ for some $\varepsilon_0>0$ and let $\Gamma_\varepsilon(\bx)=\bx +
 \varepsilon H_\varepsilon(\bx)$ denote a real analytic family of area preserving maps defined on
 a bounded domain $\mathcal{D}\subset \C^2$ for all $\varepsilon$.
 Moreover, assume that $\Gamma_\varepsilon$ is analytic in $\varepsilon$ around $0$ and that $\det H_0'(0) <0$.
 Let the origin be a hyperbolic fixed point for every map and $\varepsilon$ be the logarithm of the largest eigenvalue.
 We consider the separatrix equation
 \begin{align*}
  \mathbf{X}^-(\varepsilon;s+\varepsilon) = \Gamma_\varepsilon(\mathbf{X}^-(\varepsilon;s)).
 \end{align*}
 Then the following statements are true.
 \begin{itemize}
  \item The separatrix equation has a solution tangent to the eigenvector of $\Gamma_\varepsilon'(0)$
  that corresponds to the eigenvalue that is bigger than 1.
  \item There exists a formal solution of the separatrix equation of the form
  $$\tilde{\mathbf{X}}(\varepsilon;s)=\sum_{k\ges0} \varepsilon^k \Psi_k (e^s),$$
  with $\Psi_k$ being analytic around $0$ and $\Psi_k(0)=0$.
  \item Let $\tilde{\mathbf{X}}_n(\varepsilon;s)=\sum_{k=0}^{n-1}\varepsilon^k \Psi_k(e^s)$ be one such formal solution.
  We fix $r>0$ and let $\mathcal{D}$ be the domain such that $|e^s|<r$ and all $\Psi_k(e^s)$ are bounded.
  Fix $L>0$ and let $D$ be the $L$-time extension of $D$ such that all $\Psi_k(e^s)$ are still bounded
  (see Figure \ref{fig-theorem_domain} for an example).
  Then there exist a unique parametrization of the unstable separatrix $\mathbf{X}^-(\varepsilon;s)$
  and a constant $C_n>0$ such that
  $$\left|\mathbf{X}^-(\varepsilon;s)- \tilde{\mathbf{X}}_n(\varepsilon;s) \right|\les C_n \varepsilon^n$$
  for all $s\in D$.
 \end{itemize}
 \label{th-vassili_approx}
\end{theorem}

\begin{figure}[t]
  \centering
  \includegraphics[width=0.5\textwidth]{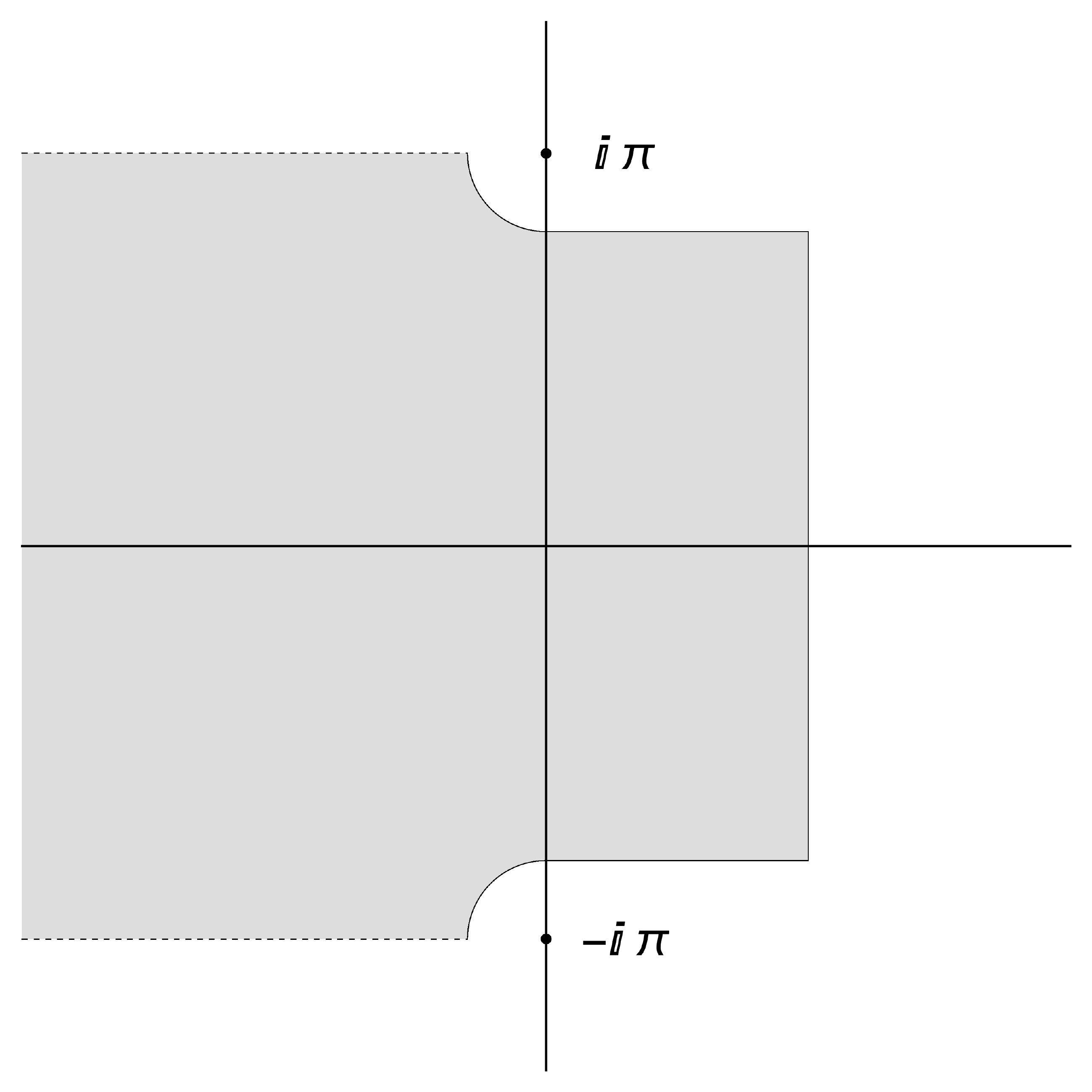}
  \caption{The domain in which we can apply the theorem.}
  \label{fig-theorem_domain}
\end{figure}

In order to apply the theorem we need to scale the map. We also move the fixed point to the origin
to make the process more transparent. Let $\varepsilon w_*$ be a hyperbolic fixed point, namely
$\varepsilon w_*= F_\varepsilon(\varepsilon w_*)$. We define the map 
$$\Gamma_\varepsilon (\bx)=\frac{1}{\varepsilon} F_\varepsilon(\varepsilon (\bx+w_*))-w_*.$$
Then $\Gamma_\varepsilon$ satisfies the conditions needed for the application of the theorem.
The condition $\det H_0'(0) <0$ is equivalent to our initial assumption that the map $g_\mu$
unfolds the resonance generically.

We define
\begin{align*}
 \tilde{\mathbf{X}}(\varepsilon;s)&=\frac{1}{\varepsilon} \formalSeparatrix (\varepsilon;\frac{s}{\varepsilon})
 =\sum_{n\ges1}\varepsilon^{n-1} Z_n\left(\tanh\left(\frac{s}{2}\right)\right)-w_*\\
 &= \sum_{n\ges1}\varepsilon^{n-1} Z_n\left(\frac{e^s-1}{e^s+1}\right)-w_*.
\end{align*}
By construction, $\tilde{\mathbf{X}}$ is a formal solution of the map $\Gamma_\varepsilon$ in the form
that the theorem predicts. Then we can apply the theorem and get the existence of a parametrization
of the unstable separatrix  $\mathbf{X}^-(\varepsilon; s )$. We define the unstable separatrix of
the map $F_\varepsilon$ by
$$ W^-(\varepsilon;\tau) = \varepsilon \mathbf{X}^-(\varepsilon;\varepsilon \tau)
+ \varepsilon w_*.$$

Let $D$ be a domain as shown in figure \ref{fig-theorem_domain} where $\tanh(s/2)$ is bounded, then
for all $\tau\in D$ it is true that
\begin{align}
 \left|W^-(\varepsilon;\tau)- \formalSeparatrix(\varepsilon;\tau) \right|\les C_n \varepsilon^{n+1}.
 \label{eq_first_approx}
\end{align}
Existence of the stable separatrix and similar estimates can be obtained
by applying the theorem on the inverse map.

\subsection{Formal separatrix close to the singularity}
\label{ch_form_sep_sing}

\begin{lemma}
 There exists a formal separatrix $\formalSeparatrixSing$ of the map $F_\varepsilon$ such that
 \begin{equation*}
  \formalSeparatrixSing(\varepsilon,t+1) = F_\varepsilon( \formalSeparatrixSing(\varepsilon,t) )
 \end{equation*}
 with
 \begin{equation*}
  \formalSeparatrixSing(\varepsilon,t) = \sum_{n\ges0}\varepsilon^n \formalSeparatrixSing_n(t)
 \end{equation*}
 and $\formalSeparatrixSing_n \in t^{n-1}\R^2[[t^{-1}]]$. Moreover $\formalSeparatrixSing_n$'s satisfy the equations
 \begin{align}
 & n=0:\quad \formalSeparatrixSing_0(t+1) = F_0( \formalSeparatrixSing_0(t) ), \nonumber \\
 & n>0:\quad \formalSeparatrixSing_n(t+1) = F_0'( \formalSeparatrixSing_0(t) )\cdot\formalSeparatrixSing_n(t) + B_n(t),
 \label{eq_Wn_equations}
\end{align}
with $B_n(t)$ depending on $\formalSeparatrixSing_m$ and $F_m$, $0\les m <n$.
\label{thm_form_sol_sing_coeffs_equations}
\end{lemma}

\begin{proof}
We defined the formal separatrice $\formalSeparatrix$ for the maps $F_\varepsilon$ and
we saw on each order of $\varepsilon$ is a polynomial of $\sigma = \tanh(\tfrac{1}{2}\varepsilon\tau)$.

Each component of $\formalSeparatrix$ has a singularity at $\pi\ii/\varepsilon$ since the hyperbolic
tangent has a simple pole there. We introduce a new parameter $t$ by translating the origin to the
singularity, $\tau=t+\pi\ii/\varepsilon$. Then we can substitute $\tanh(\tfrac{1}{2}(\varepsilon\tau+\pi\ii))$
by its Laurent series in $\formalSeparatrix$. Since the power of $\sigma$ in the order of $\varepsilon^n$
is at most $n$, the expansion does not have terms with negative powers of $\varepsilon$. The monomials
that appear in this expansion are summarized in Table \ref{tab-monomials}.

Recall that we have defined $\formalSeparatrix(\varepsilon;\tau) = \sum_{n\ges0} \varepsilon^n Z_n(\sigma)$, so
$\formalSeparatrix(\varepsilon;t+\tfrac{\pi}{\varepsilon}\ii) = \sum_{n\ges0} \varepsilon^n Z_n
\big( \tanh(\tfrac{1}{2}(\varepsilon\tau+\pi\ii)) \big)$. In Table \ref{tab-monomials} each row 
shows the monomials in the expansion of $\varepsilon^n Z_n$ without the coefficients. By changing 
summation order we can sum by columns so we have
$$ \formalSeparatrixSing(\varepsilon;t) := \formalSeparatrix(\varepsilon;t+\tfrac{\pi}{\varepsilon}\ii)
= \sum_{n\ges0}\varepsilon^n \formalSeparatrixSing_n(t), $$
with each $\formalSeparatrixSing_n(t)$ being a 2-vector of formal series in $t$.  From now on
$\formalSeparatrixSing(\varepsilon;t)$ will denote $\formalSeparatrix(\varepsilon;t+\tfrac{\pi}{\varepsilon}\ii)$
summed by columns. To conclude the proof we substitute the series in the equation
$\formalSeparatrixSing(\varepsilon;t+1)=F_\varepsilon(\formalSeparatrixSing(\varepsilon;t))$
and we gather terms in powers of $\varepsilon$.
\myqed
\end{proof}

\begin{table}[t]
\begin{center}
{\tabulinesep=1.2mm
\begin{tabu}{c|ccccccccc}
  & $\formalSeparatrixSing_0$ & $\varepsilon \formalSeparatrixSing_1$ & $\varepsilon^2 \formalSeparatrixSing_2$ & $\varepsilon^3 \formalSeparatrixSing_3$ & $\varepsilon^4 \formalSeparatrixSing_4$ & $\varepsilon^5 \formalSeparatrixSing_5$ & $\varepsilon^6 \formalSeparatrixSing_6$ & $\varepsilon^7 \formalSeparatrixSing_7$ & $\cdots$ \\
 \hline
 $\varepsilon Z_1 $ & $t^{-1}$ & $\varepsilon$ & $\varepsilon^2 \, t$ & $\varepsilon^3 \, t^2$ & $\varepsilon^4 \, t^3$ & $\varepsilon^5 \, t^4$ & $\varepsilon^6 \, t^5$ & $\varepsilon^7 \, t^6$ & $\cdots$ \\
 $\varepsilon^2 Z_2$ & $t^{-2}$ & $\varepsilon \, t^{-1}$ & $\varepsilon^2$ & $\varepsilon^3 \, t$ & $\varepsilon^4 \, t^2$ & $\varepsilon^5 \, t^3$ & $\varepsilon^6 \, t^4$ & $\varepsilon^7 \, t^5$ & $\cdots$ \\
 $\varepsilon^3 Z_3$ & $t^{-3}$ & $\varepsilon \, t^{-2}$ & $\varepsilon^2 \, t^{-1}$ & $\varepsilon^3$ & $\varepsilon^4 \, t$ & $\varepsilon^5 \, t^2$ & $\varepsilon^6 \, t^3$ & $\varepsilon^7 \, t^4$ & $\cdots$ \\
 $\varepsilon^4 Z_4$ & $t^{-4}$ & $\varepsilon \, t^{-3}$ & $\varepsilon^2 \, t^{-2}$ & $\varepsilon^3 \, t^{-1}$ & $\varepsilon^4$ & $\varepsilon^5 \, t$ & $\varepsilon^6 \, t^2$ & $\varepsilon^7 \, t^3$ & $\cdots$ \\
 $\varepsilon^5 Z_5$ & $t^{-5}$ & $\varepsilon \, t^{-4}$ & $\varepsilon^2 \, t^{-3}$ & $\varepsilon^3 \, t^{-2}$ & $\varepsilon^4 \, t^{-1}$ & $\varepsilon^5$ & $\varepsilon^6 \, t$ & $\varepsilon^7 \, t^2$ & $\cdots$ \\
 $\varepsilon^6 Z_6$ & $t^{-6}$ & $\varepsilon \, t^{-5}$ & $\varepsilon^2 \, t^{-4}$ & $\varepsilon^3 \, t^{-3}$ & $\varepsilon^4 \, t^{-2}$ & $\varepsilon^5 \, t^{-1}$ & $\varepsilon^6$ & $\varepsilon^7 \, t$ & $\cdots$ \\
 $\varepsilon^7 Z_7$ & $t^{-7}$ & $\varepsilon \, t^{-6}$ & $\varepsilon^2 \, t^{-5}$ & $\varepsilon^3 \, t^{-4}$ & $\varepsilon^4 \, t^{-3}$ & $\varepsilon^5 \, t^{-2}$ & $\varepsilon^6 \, t^{-1}$ & $\varepsilon^7$ & $\cdots$ \\
 $\varepsilon^8 Z_8$ & $t^{-8}$ & $\varepsilon \, t^{-7}$ & $\varepsilon^2 \, t^{-6}$ & $\varepsilon^3 \, t^{-5}$ & $\varepsilon^4 \, t^{-4}$ & $\varepsilon^5 \, t^{-3}$ & $\varepsilon^6 \, t^{-2}$ & $\varepsilon^7 \, t^{-1}$ & $\cdots$ \\
 $\vdots$ & $\vdots$ & $\vdots$ & $\vdots$ & $\vdots$ & $\vdots$ & $\vdots$ & $\vdots$  & $\vdots$ &
\end{tabu}}
\caption{Monomials in expansion close to the singularity}
\label{tab-monomials}
\end{center}
\end{table}

Using $\formalFlow$ we can show an analogous lemma.

\begin{lemma}
 There exists a formal separatrix $\formalFlowSing$ of the map $\tilde{F}_\varepsilon$.
 The first component of $\formalFlowSing$ is even in $t$ and the second is odd.
 Moreover
 $$\formalFlowSing(\varepsilon;t) = \sum_{n\ges0} \varepsilon^n \tsmallbbw_n(t) $$
 with $\tsmallbbw_n\in t^{n-1}\R^2[[t^{-1}]]$ and they satisfy the equations
 \begin{align}
 & n=0:\quad \tsmallbbw_0(t+1) = \tilde{F}_0( \tsmallbbw_0(t) ) \nonumber \\
 & n>0:\quad \tsmallbbw_n(t+1) = \tilde{F}_0'( \tsmallbbw_0(t) )\tsmallbbw_n(t) + \mathbbm{B}_n(t),
 \label{eq_bbmWn_equations}
\end{align}
with $\mathbbm{B}_n(t)$ depending on $\tsmallbbw_m$ and $\tilde{F}_m$, $0\les m <n$.
\label{thm_form_flow_sing_coeffs_equations}
\end{lemma}

The proof of the lemma is essentially the same as of the one above.
The evenness and oddness of the components are inherited by $\formalFlow$.

\subsection{Formal solution to the variational equation}

Let us denote $\tilde{H}_x(\varepsilon;x,y)=\partial_x \tilde{H}(\varepsilon;x,y)$, $\tilde{H}_y(\varepsilon;x,y)
=\partial_y \tilde{H}(\varepsilon;x,y)$, $\tilde{H}_{xy}(\varepsilon;x,y)=\partial_x \partial_y \tilde{H}(\varepsilon;x,y)$
and so on. Let $X_{\tilde{H}}$ be the Hamiltonian vector field of the formal Hamiltonian $\tilde{H}$, i.e. 
$X_{\tilde{H}} = (\tilde{H}_y,-\tilde{H}_x)$.

The formal separatrix $\formalFlowSing$, as defined in Lamma \ref{thm_form_flow_sing_coeffs_equations}, satisfies
the formal Hamilton's equation 
\begin{align*}
 \partial_t \formalFlowSing(\varepsilon;t) = X_{\tilde{H}}\big(\formalFlowSing(\varepsilon;t) \big).
\end{align*}

Let $J_{\tilde{H}}$ be the Jacobian of $X_{\tilde{H}}$, i.e.
$$ J_{\tilde{H}} = \begin{pmatrix}
                    \tilde{H}_{xy} & \tilde{H}_{yy} \\
                    -\tilde{H}_{xx} & -\tilde{H}_{xy}
                   \end{pmatrix}.
$$
Let $ \tilde{Z}(\varepsilon,t) = \dot \formalFlowSing(\varepsilon,t)$.
Then we have the following lemma.

\begin{lemma} 
 For all $n\in\N$ there exists $\tilde\xi_n\in t^{n+2}\C^2[[t^{-1}]]$ such that
 the formal matrix $ \Vtilde(\varepsilon;t) = (\Xi(\varepsilon,t),\tilde{Z}(\varepsilon,t)) $
 with $\Xi(\varepsilon,t) = \sum_{n\ges0}\varepsilon^n \xi_n(t)$ 
 satisfies the equations
 $$\partial_t \Vtilde(\varepsilon;t) = J_{\tilde{H}}\big(\formalFlowSing(\varepsilon;t) \big) \cdot \Vtilde(\varepsilon;t)$$
 and $\det \Vtilde(\varepsilon;t) = 1 $.
\end{lemma}

Notice that we can write 
$\Vtilde(\varepsilon;t) = \sum_{n\ges0} \varepsilon^n \Vtilde_n(t) $
with $\Vtilde_n(t)=(\tilde\xi_n(t),\dot{\tsmallbbw}_n(t))$.

\begin{proof}
We define a valuation of each monomial $\varepsilon^n t^m$ by $ \text{val}(\varepsilon^n t^m)=2n-m$.
For a series $\tilde A$ formal in $\varepsilon^n t^m$ we define $ \text{val} (\tilde A) $ to be the minimum
of the valuations of all its terms. We define this distance between two such formal series $\tilde A $
and $\tilde B$ by $d(A,B) = 2^{\text{val} (\tilde A -\tilde A)}$, and we consider the space of all the
formal series of this type and use the distance to define a topology on this space.

We know that $\tilde{H}(\varepsilon;x,-y)=\tilde{H}(\varepsilon;x,y)$. This implies that
\begin{align*}
 & \tilde{H}_x(\varepsilon;x,-y)=\tilde{H}_x(\varepsilon;x,y), \\
 & \tilde{H}_y(\varepsilon;x,-y)=-\tilde{H}_y(\varepsilon;x,y), \\
 & \tilde{H}_{xx}(\varepsilon;x,-y)=\tilde{H}_{xx}(\varepsilon;x,y), \\
 & \tilde{H}_{xy}(\varepsilon;x,-y)=-\tilde{H}_{xy}(\varepsilon;x,y), \\
 & \tilde{H}_{yy}(\varepsilon;x,-y)=\tilde{H}_{yy}(\varepsilon;x,y).
\end{align*}

We also define
$$\tilde{H}_\kappa(t):= \tilde{H}_\kappa(\varepsilon;\formalFlowSing_1(\varepsilon;t),\formalFlowSing_2(\varepsilon;t)), $$
with $\kappa\in\{x,y,xx,xy,yy\}$. Notice that the dependence on $\varepsilon$ is implied.

Since $\formalFlowSing_1(\varepsilon;t)$ is even in $t$ and $\formalFlowSing_2(\varepsilon;t)$ is odd in $t$ we have that
\begin{itemize}
 \item $\tilde{H}_x(t)$, $\tilde{H}_{xx}(t)$ and $\tilde{H}_{yy}(t)$ are formal series even in $t$,
 \item $\tilde{H}_y(t)$ and $\tilde{H}_{xy}(t)$ are formal series odd in $t$.
\end{itemize}

We are searching for a solution of the equation
$$\partial_t \Vtilde(\varepsilon;t) = \tilde{J}(t) \Vtilde(\varepsilon;t),$$
with
\begin{align*}
 \tilde{J}(t) = \begin{pmatrix}
                 \tilde{H}_{xy}(t) & \tilde{H}_{yy}(t) \\
                 -\tilde{H}_{xx}(t) & -\tilde{H}_{xy}(t)
                \end{pmatrix}.
\end{align*}

Let $\Vtilde(\varepsilon;t) = (\tilde{\Xi}(\varepsilon;t) , \tilde{Z}(\varepsilon;t))$.
We denote by $\tilde{\Xi}_{(1)}$ and $\tilde{\Xi}_{(2)}$ the first and the second component of
$\tilde{\Xi}$, respectively, and similarly for $\tilde{Z}$. Since we ask that $\det\Vtilde(\varepsilon;t) =1$
we get
$$\tilde{\Xi}_{(1)}(\varepsilon;t)  =\frac{1 + \tilde{\Xi}_{(2)}(\varepsilon;t)
\tilde{Z}_{(1)}(\varepsilon;t)}{\tilde{Z}_{(2)}(\varepsilon;t)} $$
We substitute this in the variational equation
$$\partial_t  \tilde{\Xi}_{(2)}(\varepsilon;t) = -\left(\tilde{H}_{xx}(t)
\frac{ \tilde{Z}_{(1)}(\varepsilon;t)}{\tilde{Z}_{(2)}(\varepsilon;t)} + \tilde{H}_{xy}(t) \right)
\tilde{\Xi}_{(2)}(\varepsilon;t)  - \frac{\tilde{H}_{xx}(t) }{\tilde{Z}_{(2)}(\varepsilon;t)}. $$
Since $ \tilde{Z}(\varepsilon;t) = \partial_t \formalFlowSing(\varepsilon;t) $, $\tilde{Z}_{(2)}(\varepsilon;t)$ 
is a solution for the homogeneous equation. We set 
$\tilde{\Xi}_{(2)}(\varepsilon;t) = C(\varepsilon;t) \tilde{Z}_{(2)}(\varepsilon;t)$
and substitute in the previous equation to finally get 
$$\partial_t  C(\varepsilon;t)  = - \frac{\tilde{H}_{xx}(t) }{\tilde{Z}_{(2)}(\varepsilon;t)^2}. $$
Both $\tilde{H}_{xx}(t)$ and $\tilde{Z}_{(2)}(\varepsilon;t)$ are even series.
This implies that the right-hand side of the equation is even
and it does not contain the term $\frac1t$.
This means that the above equation can be solved in the space of power series,
so $C(\varepsilon;t)$ is an odd series without logarithmic terms.
From that we get that $\tilde{\Xi}_{(2)}(\varepsilon;t)$ is odd and $\tilde{\Xi}_{(1)}(\varepsilon;t)$ is even.

By collecting terms in the same order of $\varepsilon$ in $\tilde{\xi}$ we write
\begin{align*}
 \tilde{\Xi}(\varepsilon;t) = \sum_{n\ges0} \varepsilon^n \tilde{\xi}_n(t).
\end{align*}
Straightforward computations show that the highest order of $t$ that appears in $\tilde{\xi}_n(t)$ is $n+2$.
\myqed
\end{proof}

\begin{corollary}
 The formal solution $ \Vtilde $ satisfies the equation
 $$ \Vtilde(\varepsilon;t+1) = \tilde{F}_\varepsilon(\formalFlowSing(\varepsilon;t))\cdot \Vtilde(\varepsilon;t).$$
 \label{thm_local_ref_00}
\end{corollary}

The proof of this corollary is similar to the proof of Lemma \ref{thm_formal_flow_satisfies_map}.

\begin{corollary}
 For any $n\in\N$ there exists a $2\times2$ matrix $\Vtildesep_n$ with entries in $t^{n+2}\C[[t^{-1}]]$ such that
 the formal matrix $\Vtildesep(\varepsilon;t) = \sum_{n\ges0} \varepsilon^n \Vtildesep_n(t) $ satisfies
 $$\partial_t \Vtildesep(\varepsilon;t) = F_\varepsilon'\big(\formalSeparatrixSing(\varepsilon;t) \big)
 \cdot \Vtildesep(\varepsilon;t)$$
 and $\det \Vtildesep(\varepsilon;t) = 1 $.
\end{corollary}

\begin{proof}
The corollary \ref{thm_local_ref_00} and the fact that $ F_\varepsilon'$ and $\tilde F_\varepsilon'$ are formally conjugated
imply that $\Vtildesep$ and $\Vtilde$ are formally conjugated. Moreover since $F_\varepsilon$ agrees with
the normal form up to order $N$, each $\Vtildesep_n$ will agree with $\Vtilde_n$ up to order $N$.
\myqed
\end{proof}

\subsection{Borel transform and linear difference equations}
\label{ch_borel_trans}

\begin{lemma}
 Let $\formalSeparatrixSing_n$ be as defined in Lemma \ref{thm_form_sol_sing_coeffs_equations}. The
 Borel-Laplace sum of $\formalSeparatrixSing_n$ defines two functions $ W_n^\pm $ that admit
 $\formalSeparatrixSing_n$ as asymptotic expansion in the sectorial neighbourhoods of infinity $D^\pm$
 defined in Theorem \ref{thm_existence_Borel_transform}.
 
 Using these we define the formal separatrices
 \begin{align*}
  \formalSeparatrixSing^\pm(\varepsilon,t) = \sum_{n\ges0} \varepsilon^n W^\pm_n(t).
 \end{align*}
 Let $\delta_n(t) = W_n^+(t) - W_n^-(t)$. Then $\delta_n(t) = O(t^{n+2}\me^{-2\pi\ii t})$ for all
 $t\in D^+\cap D^-$. 
 \label{thm_borel_trans_form_sol_sing}
\end{lemma}

We will present here a sketch of the proof, the full proof will be published in \cite{GM17}.

The Borel transform is defined as the formal inverse of the Laplace transform, i.e.
$$\mathcal B [t^{-n-1}] = \frac{s^n}{n!}.$$
This means that the Borel transform of a divergent series can be convergent.
If this is the case and if the Borel transform can be extended beyond a neighbourhood of the origin,
its Laplace transform will give the Borel-Laplace sum of the initial formal series. 
This method was generalized with the theory of resurgent functions by \'Ecalle in \cite{ecalle1981}.
For the purpose of the present text we will use the term \textit{resurgent function} to refer to
a function with a singularities at $2\pi\ii\Z$, which is of exponential type along paths that avoid
the set $2\pi\ii\Z$ and eventually go to infinity following a straight non-vertical line. Moreover
each singularity is a polar part and an integrable branching part. The formal Laplace transform of a resurgent
function will be called \textit{resurgent series}.

The Borel transform maps the product of two series to the convolution of their Borel sums.
This means that if we consider the space of all resurgent series as a ring under multiplication,
the space of their Borel sums is a ring under convolution and the Borel transform acts as a
ring homomorphism. On the space of the Borel sums a set of new operators that act as derivations
can be defined. These derivatives, the celebrated \textit{alien derivatives}, can be pulled back
on the space of resurgent series and do not have a classical counterpart. For each singularity of
the resurgent function there can be defined one alien derivative. These will be denoted by
$\Delta_{2\pi\ii n}$. An alien derivative of a resurgent function describes the corresponding singularity. 

We define $\Vcaltilde_0(t) = (\tilde{\Xi}_0(t),\dot{\formalSeparatrixSing}_0(t)) $. Due to
Theorem \ref{thm_existence_Borel_transform}, $\Vcaltilde_0$ satisfies
\begin{align*}
 \quad \Vcaltilde_0(t+1) = F_0'( \tilde{W}_0(t) )\cdot \Vcaltilde_0(t)
\end{align*}
and $\det \Vcaltilde_0(t) =1$.

Here we are interested in solutions of equations of the form
\begin{equation}
 X(t+1) = A(t)\cdot X(t) + B(t),
 \label{eq_non_hom_var_eq}
\end{equation}
with $A(t) = F_0'( \tilde{W}_0(t) )$ and $B\in\C^{2}[t][[t^{-1}]]$. We define $X(t) = \Vcaltilde(t)
\cdot Y(t)$ and from this we get
\begin{equation}
 Y(t+1) - Y(t) = \Vcaltilde^{-1}(t) \cdot A^{-1}(t) \cdot B(t).
 \label{eq_simple_diff_eq}
\end{equation}
Then $Y\in\C^{2}[t][[t^{-1}]]$ if and only if the formal series in the vector $\Vcaltilde^{-1} \cdot A^{-1}
\cdot B$ do not contain the term $t^{-1}$. If the term $t^{-1}$ appears then $Y$ contains also logarithmic terms.

For a resurgent series $A$ will will denote by $\hat{A}$ its the Borel transform. Since $F_0$ is convergent
and $\formalSeparatrix_0$ is a resurgent series, $\hat A$ is also a resurgent function. For a proof of
this see \cite{ecalle1981} or \cite{Sauz12b}. Suppose that the Borel transform of $B$ is a resurgent
function $\Bhat$. Then the Borel transform of equation \eqref{eq_non_hom_var_eq} is
\begin{equation}
 \me^{-s} \Xhat (s) = \Ahat * \Xhat (s) + \Bhat(s)
 \label{eq_non_hom_var_eq_Borel}
\end{equation}
and by defining $\Xhat(s) = \Vcalhat * \Yhat (s)$ we derive the equation
\begin{equation}
 \me^{-s} \Yhat(s) - \Yhat(s) = \Vcalhat^{-1} * \Ahat^{-1} * B(s).
 \label{eq_simple_diff_eq_Borel}
\end{equation}
This last equation can be solved trivially and gives
\begin{align*}
 \Yhat(s) = \frac{\me^{s}}{1-\me^{s}} \Big( \Vcalhat^{-1} * \Ahat^{-1} * B(s) \Big),
\end{align*}
so finally we get
\begin{align*}
 \Xhat(s) = \Vcalhat *  \Bigg( \frac{\me^{s}}{1-\me^{s}} \Big( \Vcalhat^{-1} * \Ahat^{-1} * B \Big) \Bigg)(s).
\end{align*}
From this we deduce that $\Xhat$ is a resurgent function.

We can now apply the above to equation \eqref{eq_Wn_equations}. In this case $B_n$ depends on $F_m$
and $\tilde{W}_m$ with $m<n$. Then using the results of \cite{Sauz15} it can be shown inductively
that the Borel transform of any $\tilde{W}_n$ defines a resurgent function $\hat{W}_n$.

Since the Borel transform $\hat{W}_n$ is resurgent, there are two Borel-Laplace sums for each
$\tilde{W}_n$, namely $W^+_n$ and $W^-_n$ and each one is the sum of a polynomial of at most
degree $n-1$ and a function decaying as $t^{-1}$ as $t$ goes to infinity. Both $W^+_n$ and
$W^-_n$ are analytic in the sectorial neighbourhoods of infinity in which $W^+_0$ and $W^-_0$
are defined.

Since $\tilde{W}_n$ are resurgent, it is a standard result of the theory that for all
$n\in\Z$
\begin{equation}
 W^+_n(t) - W^-_n(t) \asymp \me^{-2\pi\ii t} \Delta_{2\pi\ii} [\tilde{W}_n] (t)
 \label{eq_res_th_asym_diff}
\end{equation}
as $\im t \to -\infty$.
We can define the action of the alien derivative $\Delta_{2\pi\ii}$ on $ \formalSeparatrixSing $ by
$$\Delta_{2\pi\ii} [\formalSeparatrixSing] (\varepsilon;t) =
\sum_{n\ges0} \varepsilon^n \Delta_{2\pi\ii}[\tilde{W}_n](t). $$
Since $\Delta_{2\pi\ii}$ satisfies the Leibniz rule\footnote
{
This is because $\Delta_{2\pi\ii}[X_k Y_{n-k}](t) = \Delta_{2\pi\ii}[X_k](t) Y_{n-k}(t) + X_k(t)
\Delta_{2\pi\ii}[Y_{n-k}](t) $ implies that $\Delta_{2\pi\ii}[X Y](\varepsilon;t) =
\Delta_{2\pi\ii}[X](\varepsilon;t) Y(\varepsilon;t) + X(\varepsilon;t) \Delta_{2\pi\ii}[Y](\varepsilon;t) $.
},
then all $\Delta_{2\pi\ii} [\formalSeparatrixSing] (\varepsilon;t)$ satisfy the variational
equation, which means that there exist two formal series in $\varepsilon$, $\Theta_{2\pi\ii}$
and $q_{2\pi\ii}$, such that
\begin{align*}
 \Delta_{2\pi\ii} [\formalSeparatrixSing] (\varepsilon;t) = \Theta_{2\pi\ii}(\varepsilon)\,
 \tilde{\Xi}(\varepsilon;t) + q_{2\pi\ii}(\varepsilon)\,\tilde{Z}(\varepsilon;t)
\end{align*}
This relation combined with \eqref{eq_res_th_asym_diff} implies that
\begin{align*}
 \delta_n(t) \asymp \me^{-2\pi\ii t}(\theta_n \xi_n(t) + \rho_n \dot{\formalSeparatrixSing}_n(t)) + O(t^k \me^{-4\pi\ii t}),
\end{align*}
for some $k\in\N$. The term $O(t^k \me^{-4\pi\ii t})$ in the above equation comes from the singularity at
the point $4\pi\ii$. Taking into account the form of $\xi_n$ we find that
\begin{equation}
 \delta_n(t) = O(t^{n+2} \me^{-2\pi\ii t} ).
 \label{eq_bound_for_deltaN}
\end{equation}

\section{Asymptotic expansion for the separatrix and complex matching}
\label{ch_form_sols_approximate}

In this section we will show that the formal solutions $\formalSeparatrixSing^\pm$ defined in
Lemma \ref{thm_borel_trans_form_sol_sing} describe the asymptotic behaviour of $W^\pm$ close
to the singularity $\pi\ii/\varepsilon$.

Let $\mbox{SQ}(r),\mbox{SQ}_{-1}(r),\mbox{HP}(r)\subset\C$ be defined  as follows:
\begin{align*}
 &\mbox{SQ}(r):=\{ z\in\C : |\im(z)|<r, \re(z)>-r  \}, \\
 &\mbox{SQ}_{-1}(r):=\mbox{SQ}(r)-1, \\
 &\mbox{HP}(r):=\{ z\in\C : \re(z)>r  \}.
\end{align*}

Recall that we have assumed $\varepsilon<\varepsilon_0$ for some fixed $\varepsilon_0$. Since we
are interested in the asymptotic behaviour of the separatrices, we can choose $\varepsilon_0$ to
be as small as it is convenient. We choose $\Lambda > 2$ such that $\Lambda^2 \varepsilon_0 < 1$.
During the course of this proof we will see that it may be important to increase the value of
$\Lambda$. In this case we simultaneously decrease $\varepsilon_0$ such that the relation
$\Lambda^2 \varepsilon_0 < 1$ still holds. So $\Lambda$ is just a technical constant that can
be tuned to the needs of the proof. We choose $\Lambda$ such that $\{z\in\C:|\re(t)|\les2,\im(t)
\les\Lambda \}\subset \big( D^+ \cap D^- \big)$ with $D^\pm$ from Theorem \ref{thm_existence_Borel_transform}.

Let $\mbox{R}>4$ and we define the following domains:
\begin{align*}
 & \mathcal{D}_0:=\{t\in\C:|\im t| \les \pi/\varepsilon\}\backslash(\mbox{SQ}((\Lambda \varepsilon)^{-1}) \cup \mbox{HP}(\mbox{R})), \\
 & \mathcal{D}_1:=\mbox{SQ}_{-1}((\Lambda \varepsilon)^{-1}) \backslash (\mbox{SQ}(\varepsilon^{-\frac{1}{2}}) \cup \mbox{HP}(\mbox{R})), \\
 & \mathcal{D}_2:=\mbox{SQ}_{-1}(\varepsilon^{-\frac{1}{2}}) \backslash (\mbox{SQ}(\Lambda) \cup \mbox{HP}(\mbox{R})).
\end{align*}
These can be seen in Figure \ref{fig-asymt_domains}. Note that $\mathcal{D}_1$ intersects $\mathcal{D}_0$ in
a narrow strip of width 1 (Figure \ref{fig-asymt_domain0}) and that $\mathcal{D}_2$ intersects
$\mathcal{D}_1$ in another narrow strip of width 1 (Figure \ref{fig-asymt_domainN}).

\begin{figure}[t]
        \centering
        \begin{subfigure}[b]{0.49\textwidth}
                \centering
                \includegraphics[width=\textwidth]{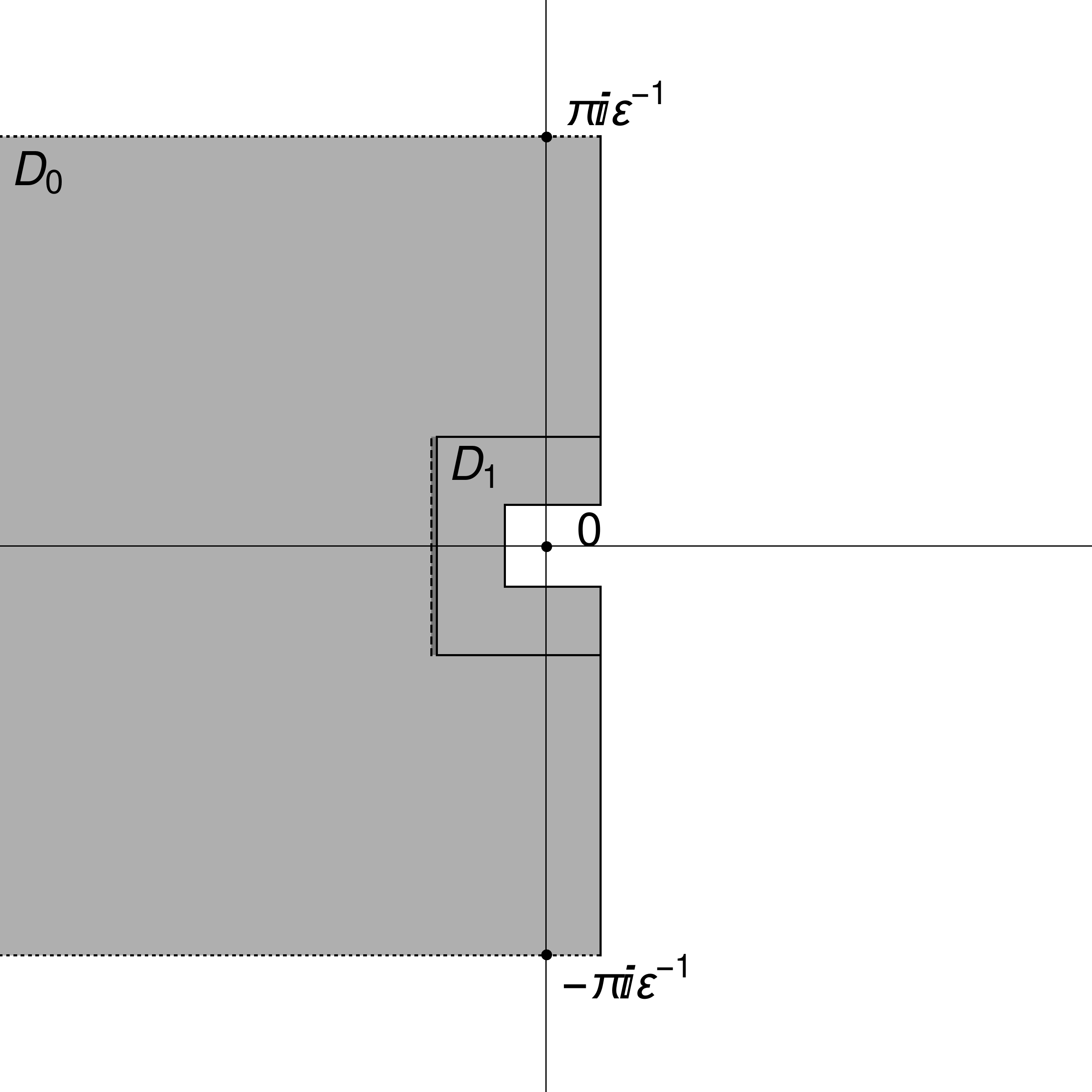}
                \caption{The domains $\mathcal{D}_0$ and $\mathcal{D}_1$.}
                \label{fig-asymt_domain0}
        \end{subfigure}%
        ~
        \begin{subfigure}[b]{0.49\textwidth}
                \centering
                \includegraphics[width=\textwidth]{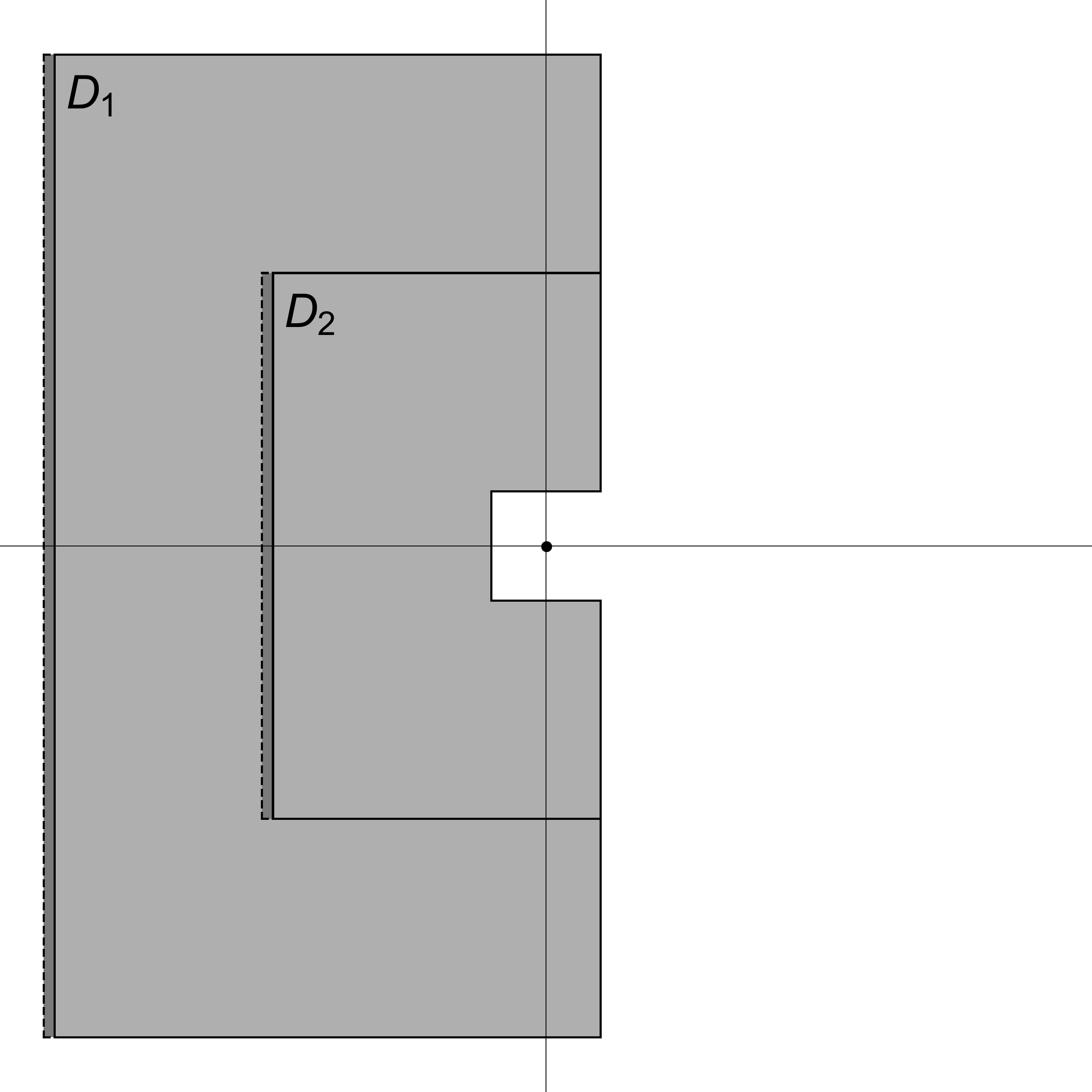}
                \caption{The domains $\mathcal{D}_1$ and $\mathcal{D}_2$.}
                \label{fig-asymt_domainN}
        \end{subfigure}%
        \caption{The domains considered in this section.}
        \label{fig-asymt_domains}
\end{figure}

Let $n\in\N$, $n\les N$, with $N$ defined by equation \eqref{eq_N_definition}. We define
\begin{align*}
 &\formalSeparatrixSing_n^\pm(\varepsilon;t) := \sum_{k=0}^{n-1}\varepsilon^k W^\pm_k(t), \\
 &\formalSeparatrix_n(\varepsilon;t) := \sum_{k=1}^{n}\varepsilon^k Z_k(\sigma),
\end{align*}
with $\varepsilon^k W^\pm_k$ defined in Lemma \ref{thm_borel_trans_form_sol_sing} and $Z_k$
defined in Corollary \ref{thm_form_sol_corollary}.

The main result of this section is the following lemma.

\begin{lemma}
 There exist $\Lambda>2$ and $\varepsilon_0>0$ such that for every $n\in\N$, $5\les n\les N$ and
 every $\varepsilon<\varepsilon_0$ there exists $C_2>0$ such that for all $t\in \mathcal{D}_2$
 \begin{align*}
  \left| W^-(\varepsilon;t+\tfrac{\pi}{\varepsilon}\ii) - \formalSeparatrixSing_n^-(\varepsilon;t)
  \right| \les C_2 \, \varepsilon^{\frac{n-1}{2}}
 \end{align*}
 and for all $t\in-\mathcal{D}_2$ it holds
 \begin{align*}
  \left| W^+(\varepsilon;t+\tfrac{\pi}{\varepsilon}\ii) - \formalSeparatrixSing_n^+(\varepsilon;t)
  \right| \les C_2 \, \varepsilon^{\frac{n-1}{2}},
 \end{align*}
 with the inequality and the absolute value interpreted componentwise.
 \label{thm_local_lemma_matching}
\end{lemma}

The rest of this section is devoted to the proof of this lemma. Before giving the proof
we need a few intermediate results.

For a $2\times2$ matrix $A$ we define $$|A|_\infty = \max\{ |A_{1,1}| + |A_{1,2}|, |A_{2,1}| + |A_{2,2}| \}.$$

\begin{lemma}
 Let $Q:\mathcal{D}_1\to\C^2$ and $ c>0$ such that $|Q(t)| \les c |t|^{-2}$ for all $t\in \mathcal{D}_1$. Then
 there exists $C_{1,1}>0$ such that 
 \begin{align*}
  \left| F'_\varepsilon ( \varepsilon Z_1(\sigma) + Q(t) ) \right|_\infty \les 1+\frac{2}{|t|}
  + C_{1,1} \varepsilon 
 \end{align*}
 for all $t\in \mathcal{D}_1$.
 \label{thm_jacobian_D1_bound}
\end{lemma}

\begin{proof}
 Throughout the proof we keep in mind that for all $t\in \mathcal{D}_1$, $|t|^2 \varepsilon\ges1$.
 
 Let $s\in\C$, $|s|<1/2$ then it holds
 \begin{align*}
  \tanh\left( \frac{\pi \ii}{2} + s \right) = \frac{1}{s} + s\, \phi(s),
 \end{align*}
  with $  |\phi(s)| \les 1 $
 and
 \begin{align*}
  \left|\tanh\left( \frac{\pi \ii}{2} + s \right) \right| \les \frac{2}{|s|}.
 \end{align*}
 So for $\tau = t+\pi\ii/\varepsilon$ we have
 \begin{align*}
  \varepsilon \sigma = \frac{2}{t} + \frac{\varepsilon^2 t}{2}\phi\left(\tfrac{\varepsilon t}{2}\right)
  \quad \text{ and } \quad
  |\varepsilon\sigma| \les\frac{4}{|t|}.
 \end{align*}
 Recall that we write $F_\varepsilon(x,y) = \sum_{n\ges1} \mathcal{F}_n(\varepsilon,x,y)$, with
 $\mathcal{F}_n$ a homogeneous polynomial of degree $n$.
 Since $F_\varepsilon$ is tangent to identity, $\mathcal{F}'_1$ is the identity
 (so trivially $\| \mathcal{F}'_1 \|_\infty = 1$). Also
 \begin{align*} \mathcal{F}'_2(\varepsilon, x,y)=
  \left(
  \begin{array}{cc}
   -2 b_{0,0} y &  -2 b_{0,0} x + \frac{\varepsilon}{\sqrt{3} }   \\
   -2 b_{0,0} x- \frac{\varepsilon}{\sqrt{3}}   &  2 b_{0,0} y \\
  \end{array}
 \right),
 \end{align*}
 which implies that $\|\mathcal{F}'_2(\varepsilon,Q(t))\|_\infty \les C_1\varepsilon$ and
 \begin{align*}
  \mathcal{F}'_2(\varepsilon,\varepsilon \mathcal{Z}_1(\sigma))=\begin{pmatrix}
                                                     \varepsilon \sigma & 0 \\
                                                     \frac{\sqrt{3}}{2}\varepsilon & -\varepsilon \sigma
                                                    \end{pmatrix}.
 \end{align*}
 The last relation implies that $\|\mathcal{F}'_2(\varepsilon,\varepsilon \mathcal{Z}_1(\sigma))\|_\infty
 \les \frac{2}{|t|} +  \frac{\sqrt{3}}{2 }\varepsilon$.

 For all $t\in \mathcal{D}_1$, the first component of $\varepsilon \mathcal{Z}_1(\sigma)$ is a
 constant times $\varepsilon$ and the second component is bounded by a constant over $|t|$. 
 From this we get 
 \begin{align*}
  \left|  \varepsilon \mathcal{Z}_1(\sigma) + Q(t)  \right| &= \frac{1}{|t|} \left|  t \varepsilon \mathcal{Z}_1(\sigma) + t Q(t)  \right| \\
  & \les \frac{1}{2 b_{0,0} |t|} \left( \left|  \Vector{\frac{\varepsilon t}{\sqrt{3}}}{2 + \frac{\varepsilon^2 t^2}{2}\phi\left(\frac{\varepsilon t}{2}\right)}  \right| + \Vector{\frac{2 b_{0,0} \, c}{|t|}}{\frac{2 b_{0,0} \, c}{|t|}} \right) \\ 
  & \les \frac{1}{2 b_{0,0} |t|}   \Vector{\frac{ \varepsilon |t|}{\sqrt{3}} + \frac{2 b_{0,0} \, c}{|t|}}{\left| 2 + \frac{\varepsilon^2 t^2}{2}\phi\left(\frac{\varepsilon t}{2}\right)\right|+ \frac{2 b_{0,0} \, c}{|t|}}   \\ 
  & \les \frac{1}{2 b_{0,0} |t|}   \Vector{\frac{ 1}{\sqrt{3} \Lambda} + \frac{2 b_{0,0} \, c}{\Lambda}}{ 2 + \frac{1}{2 \Lambda^2} \left| \phi\left(\frac{\varepsilon t}{2}\right)\right|+ \frac{2 b_{0,0} \, c}{\Lambda}}   \\ 
  & \les \frac{1}{2 b_{0,0} |t|}   \Vector{\frac{ 1}{\sqrt{3} \Lambda} + \frac{2 b_{0,0} \, c}{\Lambda}}{ 2 + \frac{1}{2 \Lambda^2} + \frac{2 b_{0,0} \, c}{\Lambda}}   \\ 
  & \les \frac{1}{2 b_{0,0} |t|}   \left( 2 + \frac{1}{ \Lambda} + \frac{2 b_{0,0} \, c}{\Lambda} \right)=\frac{C_2}{|t|},
 \end{align*}
 where inequality and absolute value interpreted componentwise. Notice that
 the constant $C_2$ is a decreasing function of $\Lambda$.
 
 For all $\mathcal{F}'_n(\varepsilon;x,y)$ with $n\ges 3$, each monomial of $\mathcal{F}'_n$
 is of degree $n-1$. We substitute $\varepsilon \mathcal{Z}_1(\sigma) + Q(t)$ in $\sum_{n\ges3}
 \mathcal{F}'_n$. Then all the monomials are in $O(\varepsilon)$.
 
 Collecting everything together we get the result. Notice that the constant $C_{1,1}$ is also a
 decreasing function of $\Lambda$.
\myqed
\end{proof}

\begin{lemma}
 Let $Q:\mathcal{D}_2\to\C^2$ and assume that there exists $ c>0$ such that $|Q(t)|\les c \varepsilon $.
 Then there exist $C_{2,1},C_{2,2}>0$ such that 
 \begin{align*}
  \left\| F'_\varepsilon ( W^-_0(t) + Q(t) ) \right\|_\infty \les 1+\frac{2}{|t|} + C_{2,1} \varepsilon
  + \frac{C_{2,2}}{|t|^2}.
 \end{align*}
 \label{thm_jacobian_D2_bound}
\end{lemma}

\begin{proof}
 For this proof we take into account that $|t|>\Lambda$ and $\Lambda^2 \varepsilon_0<1$.
 
 Recall from Section
 \ref{ch_form_sep_sing} that
 \begin{align*}
  W^-_0(t) = \Vector{0}{-\frac{1}{b_{0,0}\, t}}+ r(t),
 \end{align*}
 with $|r(t)| \les C_r |t|^{-2}$, which also implies trivially that $\|W^-_0(t)\|_\infty \les C_0 |t|^{-1}$.
 Expanding $F_0$ in Taylor series and taking into account that it agrees with the normal
 form up to the order $N$ which is bigger than $3$ we get
 \begin{align*}
  F_0(x,y)= \Vector{x - 2 b_{0,0} x y +   b_{0,0}^2 x^3 +  b_{0,0}^2 x y^2}{y -   b_{0,0} x^2 +  b_{0,0} y^2 +  b_{0,0}^2 x^2 y +  b_{0,0}^2 y^3} +O_4(x,y).
 \end{align*}
 Differentiating the above and substituting $W_0^-$ we get
 \begin{align*}
  F_0'( W^-_0(t) + Q(t) ) = \begin{pmatrix}
                                        1 + \frac{2}{t} & 0 \\
                                        0 & 1 - \frac{2}{t} 
                            \end{pmatrix} + R(t)
 \end{align*}
 with $\|R(t)\|_\infty \les C_R (\varepsilon + |t|^{-2})$. Moreover for all $ k\in\N$,
 $k\ges 1$ there exists $C_k$ such that $\| \varepsilon^k F'_k ( W^-_0(t) + Q(t) ) \|_\infty
 \les C_k \varepsilon^k \les C_k \varepsilon \Lambda^{2-2k} $ and since $F_\varepsilon$ is
 analytic around the origin. Summing these gives the result.  As before the constants $C_{2,1}$
 and $C_{2,2}$ are decreasing functions of $\Lambda$.
\myqed
\end{proof}

\begin{lemma}
 Let $\mu:\C\backslash (\mbox{SQ}(\Lambda) \cup \mbox{HP}(R)) \to\R_+$ with
 \begin{align*}
  \mu(t)\les 1 + \frac{2}{|t|} + c_1 \varepsilon + \frac{c_2}{|t|^2}
 \end{align*}
 for some $c_1,c_2>0$. Then for all $m\in\N$ with $m\les (\Lambda \varepsilon)^{-1}+2$
 and $t+m\in\C\backslash (\mbox{SQ}(\Lambda) \cup \mbox{HP}(\Lambda))$ it holds
 \begin{align*}
  \prod_{k=0}^{m} \mu(t+k) \les C  \frac{|t|^2}{|t+m|^2},
 \end{align*}
 with $$C=\Bigg(1 +\frac{2}{\Lambda} + \frac{c_1+c_2}{\Lambda^2} \Bigg)\cdot\exp\Bigg(2\pi
 + \frac{\pi}{\Lambda}  \left(c_2 + \left(4 + \frac{c_1}{\Lambda}  + \frac{c_2}{\Lambda}
 \right)^2 \right) + c_1\left(\frac{1}{\Lambda}+ \frac{2}{\Lambda^2}\right) \Bigg).$$
 \label{thm_prod_bound}
\end{lemma}

\begin{proof}
 For all $x\in\R$ with $x\ges 0$, it holds that $\log(1+x)=x+r(x)$ with $|r(x)|\les x^2$. So we have
 \begin{align*}
  & \log \left(1 + 2\frac{|\re(t)|+|\im(t)|}{|t|^2}+ c_1 \varepsilon + \frac{c_2}{|t|^2}\right)= \\
  & = 2\frac{|\re(t)|+|\im(t)|}{|t|^2}+ c_1 \varepsilon + \frac{c_2}{|t|^2} + r\left( \frac{1}{|t|} \left(2\frac{|\re(t)|+|\im(t)|}{|t|}+ c_1 \varepsilon |t| + \frac{c_2}{|t|} \right) \right) \\
  & \les 2\frac{|\re(t)|+|\im(t)|}{|t|^2}+ c_1 \varepsilon + \frac{c_2}{|t|^2} + \frac{1}{|t|^2} \left(2\frac{|\re(t)|+|\im(t)|}{|t|}+ c_1 \varepsilon |t| + \frac{c_2}{|t|} \right)^2 \\
  & \les 2\frac{|\re(t)|+|\im(t)|}{|t|^2}+ c_1 \varepsilon + \frac{1}{|t|^2} \left(c_2 + \left(4 + \frac{c_1}{\Lambda}  + \frac{c_2}{\Lambda} \right)^2 \right) \\
  & \les 2\frac{|\re(t)|+|\im(t)|}{|t|^2}+ c_1 \varepsilon + \frac{C_2}{|t|^2}.
 \end{align*}
 
 Then by standard integration we get
 \begin{align*}
  \int_t^{t+m} 2\frac{|\re(\dt)|+|\im(\dt)|}{|\dt|^2} \dd \dt &=  \log\left( \frac{|t|^2}{|t+m|^2}
  \right) - 2 \arctan\left(\frac{|\re(t)|-m}{|\im(t)|} \right) \\
  & + 2 \arctan\left(\frac{|\re(t)|}{|\im(t)|} \right) \les \log\left( \frac{|t|^2}{|t+m|^2} \right) +2\pi
 \end{align*}
 and
 \begin{align*}
  \int_t^{t+m} \frac{C_2}{|\dt|^2} \dd \dt &= \frac{C_2}{|\im(t)|} \arctan\left(\frac{|\re(t)|-m}{|\im(t)|} \right)
  - \frac{C_2}{|\im(t)|} \arctan\left(\frac{|\re(t)|}{|\im(t)|}\right) \les \frac{C_2}{\Lambda}  \pi.
 \end{align*}
 Also note that 
 \begin{align*}
  c_1 m \varepsilon \les c_1\left(\frac{1}{\Lambda}+2\varepsilon \right) .
 \end{align*}
 So collecting everything together we get
 \begin{align*}
  \int_t^{t+m} 2\frac{|\re(\dt)|+|\im(\dt)|}{|\dt|^2} + \frac{C_2}{|\dt|^2} \dd \dt
  + c_1 \varepsilon m \les \log\left( \frac{|t|^2}{|t+m|^2} \right) + 2\pi + \frac{C_2}{\Lambda}
  \pi + c_1\left(\frac{1}{\Lambda}+2\varepsilon \right) ,
 \end{align*}
 By the above we get 
 \begin{align*}
  \log & \left( \prod_{k=1}^{m} \mu(t+k) \right) = \sum_{k=0}^{m-1} \log(\mu(t+k)) \\
  & \les \sum_{k=1}^{m} 2\frac{|\re(t+k)|+|\im(t+k)|}{|t+k|^2} + \frac{C_2}{|t+k|^2} + c_1 \varepsilon \\
  & \les \int_t^{t+m} 2\frac{|\re(\dt)|+|\im(\dt)|}{|\dt|^2} + \frac{C_2}{|\dt|^2} \dd \dt + c_1 \varepsilon m\\
  & \les \log\left( \frac{|t|^2}{|t+m|^2} \right) + 2\pi + \frac{C_2}{\Lambda}  \pi
  + c_1\left(\frac{1}{\Lambda}+2\varepsilon \right) .
 \end{align*}
 Note that trivially $\mu(t) \les 1 +\frac{2}{\Lambda} + \frac{c_1+c_2}{\Lambda^2}$.
 Then we need to exponentiate the last relation and multiply one last time by $\mu(t)$.
 Using the bound of $\mu(t)$ we get the result.
\myqed
\end{proof}

\begin{lemma}
 There exists $\Lambda>1$ and $\varepsilon_0>0$  such that for every $n\in\N$, $5\les n\les N$
 and every $\varepsilon<\varepsilon_0$ there exists $C_1>0$ such that for all $t\in \mathcal{D}_1$ it holds
 \begin{align*}
  \left\| W^-(\varepsilon;t+\tfrac{\pi}{\varepsilon}\ii) - \formalSeparatrix_n(\varepsilon;t)
  \right\|_\infty \les \frac{C_1}{|t|^{n+1}}.
 \end{align*}
 \label{thm_loc_lemma_D1}
\end{lemma}

\begin{proof}
 Let
 \begin{align*}
  \xi_n(\varepsilon;t) : &= W^-(\varepsilon;t+\tfrac{\pi}{\varepsilon}\ii) -
  \formalSeparatrix_n(\varepsilon;t+\tfrac{\pi}{\varepsilon}\ii), \\
  R_n(\varepsilon;t): &=  F_\varepsilon(\formalSeparatrix_n(\varepsilon;t-1+\tfrac{\pi}{\varepsilon}\ii) -
  \formalSeparatrix_n(\varepsilon;t+\tfrac{\pi}{\varepsilon}\ii) ).
 \end{align*}
 As we saw in Section \ref{ch_form_sep}  $ \formalSeparatrix_1(\varepsilon;t+1+\tfrac{\pi}{\varepsilon}\ii)
 - F_\varepsilon(\formalSeparatrix_1(\varepsilon;t+\tfrac{\pi}{\varepsilon}\ii) )=O(\varepsilon^{3} \sigma^{3})$.
 Then each order in $\formalSeparatrix_n$ cancels an order of the difference.
 This implies that $ \formalSeparatrix_n(\varepsilon;t+1+\tfrac{\pi}{\varepsilon}\ii)
 - F_\varepsilon(\formalSeparatrix_n(\varepsilon;t+\tfrac{\pi}{\varepsilon}\ii) )=O(\varepsilon^{n+2} \sigma^{n+2})$.
 So for all $t\in \mathcal{D}_1$ it holds $R_n(\varepsilon;t)=O(|t|^{-n-2})$.
 
 Then we have
 \begin{align*}
  \xi_n(\varepsilon;t+1) &= W^-(\varepsilon;t+1+\tfrac{\pi}{\varepsilon}\ii) -
  \formalSeparatrix_n(\varepsilon;t+1+\tfrac{\pi}{\varepsilon}\ii) \\
  &= F_\varepsilon(W^-(\varepsilon;t+\tfrac{\pi}{\varepsilon}\ii)) -
  F_\varepsilon(\formalSeparatrix_n(\varepsilon;t-1+\tfrac{\pi}{\varepsilon}\ii) ) + R_n(\varepsilon;t+1) \\
  &= \left( \int_0^1 F_\varepsilon' ( \formalSeparatrix_n(\varepsilon;t+\tfrac{\pi}{\varepsilon}\ii) +
  \dt\,\xi_n(\varepsilon;t) ) \dd \dt \right) \xi_n(\varepsilon;t) + R_n(\varepsilon;t+1)
 \end{align*}
 and consequently
 \begin{align}
  \xi_n(\varepsilon;t+k+1) &= \left( \int_0^1 F_\varepsilon ( \formalSeparatrix_n(\varepsilon;t+k+\tfrac{\pi}{\varepsilon}\ii)
  + \dt\,\xi_n(\varepsilon;t+k) ) \dd \dt \right) \xi_n(\varepsilon;t+k) \nonumber \\
  & \phantom{=} + R_n(\varepsilon;t+k+1). \label{eq_local_eq_00}
 \end{align}
 Let
 \begin{align*}
  \delta_k :&= \left| \xi_n(\varepsilon;t+k) \right| , \\
  \alpha_k :&= \left\| \int_0^1 F_\varepsilon ( \formalSeparatrix_n(\varepsilon;t+k+\tfrac{\pi}{\varepsilon}\ii) +
  \dt\,\xi_n(\varepsilon;t+k) ) \dd \dt \right\|_\infty , \\
  \beta_k :&= \left| R_n(\varepsilon;t+k+1) \right| . \\
 \end{align*}
 Then by taking the absolute value of equation \eqref{eq_local_eq_00} we have
 $ \delta_{k+1} \les \alpha_k \delta_k+\beta_k $, which implies that
 \begin{align*}
  \delta_k\les \left( \prod_{i=1}^{n-1} \alpha_i \right)\delta_0 + \sum_{i=0}^{k-1} \left(\prod_{j=i+1}^{k-1} \alpha_j\right) \beta_i.
 \end{align*}
 This equation allows induction to be used.
 
 We know from Theorem \ref{th-vassili_approx} that for all $ t\in \mathcal{D}_0\cap \mathcal{D}_1$ it holds
 \begin{align*}
  \left|W^-(\varepsilon;t+\tfrac{\pi}{\varepsilon}\ii)- \formalSeparatrix_n(\varepsilon;t+\tfrac{\pi}{\varepsilon}\ii) \right|\les \frac{C_0}{|t|^{n+1}},
 \end{align*}
 so we get $\delta_0\les C_0 |t|^{-n-1}$ and $\beta_k \les C_\beta |t+k+1|^{-n-2}$ from Taylor's theorem.
 
 Assume that for all $j < k$ it holds $\delta_j\les C_1 |t+j|^{-n-1}$ with some constant $C_1>2\exp(2\pi+1)(C_0+C_\beta)$.
 Then we can use Lemma \ref{thm_jacobian_D1_bound} with $\xi_j$ in the place of $Q$ to get
 \begin{align*}
  \alpha_j \les 1+\frac{2}{|t+j|}+ C_{1,1} \varepsilon + \frac{C_{1,2}}{|t+j|^2}.
 \end{align*}
 This implies that
 \begin{align*}
  \delta_k& \les \left( \prod_{i=1}^{k-1} \alpha_i \right)\delta_0 + \sum_{i=0}^{k-1} \left(\prod_{j=i+1}^{k-1} \alpha_j\right) \beta_i \\
  & \les C\frac{|t|^2}{|t+k|^2} \frac{C_0}{|t|^{n+1}} + \sum_{i=0}^{k-1} C \frac{|t+i+1|^2}{|t+k|^2} \frac{C_\beta}{|t+i+1|^{n+2}} \\
  & \les \frac{C_0 C}{|t+k|^2|t|^{n-1}} +  \frac{C_\beta C}{|t+k|^2}\sum_{i=0}^{k-1}\frac{1}{|t+j+1|^{n}} \\
  & \les \frac{C_0 C}{|t+k|^{n+1}} + \frac{C_\beta C}{|t+k|^2}\frac{1}{|t+k|^{n-1}} \\
  & \les \frac{C(C_0+C_\beta) }{|t+k|^{n+1}},
 \end{align*}
 with $C$ given by Lemma \ref{thm_prod_bound}.
 
 We choose $\Lambda$ big enough to have
 $C(C_0+C_\beta) < C_1$. Then we get that the inductive hypothesis holds also for $k$.
 To extend the bound to the whole $\mathcal{D}_1$ we need to apply the same technique for
 $\mbox{R}$ more steps which changes only the constants.
\myqed
\end{proof}

The above actually proves that the bound is true in $\mbox{SQ}_{-1}(\Lambda)\backslash \mbox{HP}(0)$.
Of course the bound becomes arbitrarily big close to the origin so it will be used only in $\mathcal{D}_1$.

\begin{proof}[Proof of Lemma \ref{thm_local_lemma_matching}]
 Let
 \begin{align*}
  \xi_n(\varepsilon;t) : &= W^-(\varepsilon;t+\tfrac{\pi}{\varepsilon}\ii) - \formalSeparatrixSing_n^-(\varepsilon;t), \\
  R_n(\varepsilon;t): &= F_\varepsilon(\formalSeparatrixSing_n^-(\varepsilon;t-1) ) - \formalSeparatrixSing_n^-(\varepsilon;t) .
 \end{align*}
 We saw in Section \ref{ch_form_sep_sing} that $\formalSeparatrixSing_0^- (\varepsilon;t+1)
 - F_\varepsilon(\formalSeparatrixSing_0^- (\varepsilon;t) )=O(\varepsilon)$ and then each
 order in $\formalSeparatrixSing_n^-$ cancels one row of the Table \ref{tab-monomials}.
 This implies that $\formalSeparatrixSing_n^-(\varepsilon;t+1) -
 F_\varepsilon(\formalSeparatrixSing_n^-(\varepsilon;t) )=O(\varepsilon^{n+1} t^{n})$.
 So for all $t\in \mathcal{D}_2$ it holds $R_n(\varepsilon;t)=O(\varepsilon^{n+1} t^{n})$.
 
 Similarly to the previous proof we have
 \begin{align*}
  \xi_n(\varepsilon;t+1) = \left( \int_0^1 F_\varepsilon ( \formalSeparatrixSing_n^- (\varepsilon;t) + \dt\,\xi_n(\varepsilon;t) ) \dd \dt \right) \xi_n(\varepsilon;t) + R_n(\varepsilon;t+1),
 \end{align*}
 which implies
 \begin{align*}
  \xi_n(\varepsilon;t+k+1) = \left( \int_0^1 F_\varepsilon ( \formalSeparatrixSing_n^- (\varepsilon;t+k) + \dt\,\xi_n(\varepsilon;t+k) ) \dd \dt \right) \xi_n(\varepsilon;t+k) + R_n(\varepsilon;t+k+1).
 \end{align*}
 We define
 \begin{align*}
  \delta_k :&= \left\| \xi_n(\varepsilon;t+k) \right\|_\infty , \\
  \alpha_k :&= \left\| \int_0^1 F_\varepsilon ( \formalSeparatrixSing_n^- (\varepsilon;t+k) + \dt\,\xi_n(\varepsilon;t+k) ) \dd \dt \right\|_\infty , \\
  \beta_k :&= \left\| R_n(\varepsilon;t+k+1) \right\|_\infty.
 \end{align*}
 Then again we have $\delta_{k+1} \les \alpha_k \delta_k+\beta_k$
 and
 \begin{align*}
  \delta_k\les \left( \prod_{i=1}^{n-1} \alpha_i \right)\delta_0 + \sum_{i=0}^{k-1}
  \left(\prod_{j=i+1}^{k-1} \alpha_j\right) \beta_i.
 \end{align*}
 From now on we assume that $t\in \mathcal{D}_1\cap \mathcal{D}_2$ and for such $t$ by Lemma \ref{thm_loc_lemma_D1} it holds
 \begin{align*}
  \left\|W^-(\varepsilon;t+\tfrac{\pi}{\varepsilon}\ii)- \formalSeparatrixSing_n^-(\varepsilon;t) \right\|_\infty
  \les \frac{C_1}{|t|^{n+1}}\les C_1 \varepsilon^{\frac{n+1}{2}}.
 \end{align*}
 Assume that for all $j < k$ it holds $\delta_j\les C_2 \varepsilon^{\frac{n-1}{2}}$ with a constant $C_2> 2\exp(2\pi+1)$. 
 Then using Lemma \ref{thm_jacobian_D2_bound} with $\delta_j$ being $Q$ we get that
 \begin{align*}
  \alpha_j \les 1+\frac{2}{|t+j|}+ C_{1,1} \varepsilon + \frac{C_{1,2}}{|t+j|^2}
 \end{align*}
 and
 \begin{align*}
  \delta_k& \les \left( \prod_{i=1}^{k-1} \alpha_i \right)\delta_0 + \sum_{i=0}^{k-1} \left(\prod_{j=i+1}^{k-1} \alpha_j\right) \beta_i \\
  & \les C\frac{|t|^2}{|t+k|^2} \frac{C_1}{|t|^{n+1}}, + \sum_{i=0}^{k-1} C \frac{|t+j+1|^2}{|t+k|^2} C_\beta \varepsilon^{n+1} |t+j+1|^{n-1} \\
  & \les \frac{C_1 C}{|t+k|^2|t|^{n-1}} +  \frac{C_\beta C \varepsilon^{n+1}}{|t+k|^2 }\sum_{i=0}^{k-1}|t+j+1|^{n+1} \\
  & \les \frac{C_0 C}{|t+k|^{2}} \varepsilon^{\frac{n-1}{2}} + \frac{C_\beta C \varepsilon^{n+1} }{|t+k|^2} |t+k|^{n+2}\\
  & \les \frac{C_0 C}{|t+k|^{2}} \varepsilon^{\frac{n-1}{2}} + C_\beta C \varepsilon^{n+1}  |t+k|^{n+1}\\
  & \les \frac{C_0 C}{|t+k|^{2}} \varepsilon^{\frac{n-1}{2}} + C_\beta C \varepsilon^{n+1}  \left( \frac{\sqrt{2}}{\sqrt{\varepsilon}} \right)^{n+1}\\
  & \les \frac{C_0 C}{|t+k|^{2}} \varepsilon^{\frac{n-1}{2}} + 2^{\frac{n+1}{2}} C_\beta C \varepsilon^{\frac{n+1}{2}} \\
  & \les C \left( \frac{C_0}{|t+k|^{2}}  + 2^{\frac{n+1}{2}} C_\beta \varepsilon \right) \varepsilon^{\frac{n-1}{2}} \\
  & \les C\left( \frac{C_0}{\Lambda^{2}}  + 2^{\frac{n+1}{2}} \frac{C_\beta}{\Lambda^{2}} \right) \varepsilon^{\frac{n-1}{2}} 
 \end{align*}
 Similarly to the previous proof we can choose $\Lambda$ big enough to get $\delta_{k}\les C_2 \varepsilon^{\frac{n-1}{2}}$.
 Then by induction we get the result.
 To extend the bound to the whole $\mathcal{D}_2$ we need to apply the same technique for $\mbox{R}$
 more steps which changes only the constants.
\myqed
\end{proof}

Using the inverse map we get a similar result for the stable separatrix.

\section{Variational equations}
\label{ch_var_eqs}

There are two variational equations that are important in this analysis. In this section we
will show that the solutions of both can be approximated by the same formal series. 

\subsection{Linear difference equations in a rectangular domain}

We consider rectangular symmetric domains around the origin, i.e. $D=\{z\in\C : |\re(z)|\les \alpha,
|\im(z)|\les \beta \} $ for some $\alpha,\beta>1$. Let $\mathcal{O}(D)$ be the space
of functions analytic in the interior of $D$ and continuous at its boundary equipped
with the supremum norm over $D$.  

Let $g\in\mathcal{O}(D)$. We will examine the equation
\begin{align}
 X(z+1)-X(z)=g(z).
 \label{eq_lin_diff_eq_model}
\end{align}
We define the operator
\begin{align*}
 \mathcal{S}: X(z) \mapsto X(z+1)-X(z).
\end{align*}
To solve the equation \eqref{eq_lin_diff_eq_model} we need to invert the operator $\mathcal{S}$.
We can construct the following two formal solutions
\begin{align*}
 S^+[g](z):&= - \sum_{n\ges0} g(z+n) \\
 \text{and }\;\quad S^-[g](z):&= \sum_{n\ges1} g(z-n).
\end{align*}

Since $g$ is defined in a compact set around the origin, the above solutions have no analytic
meaning unless $g$ can be extended beyond its initial domain of definition. Towards this end
we have the following lemma.

\begin{lemma}[\cite{Gel99}]
 Let $h\in\mathcal{O}(D)$, $\chi$ be a Lipschitz continuous function of $\partial \mathcal{D}$ and
 \begin{align*}
  J_h=\frac{1}{2\pi}\int_{\partial \mathcal{D}} |h(\zeta)| |\dd \zeta| < \infty.
 \end{align*}
 
 Then the integral
 \begin{align*}
  H(z)=\frac{1}{2\pi \ii}\int_{\partial \mathcal{D}} \frac{h(\zeta) \chi(\zeta)}{\zeta-z} \dd \zeta
 \end{align*}
 defines two functions $H_{\text{int}}$ and $H_{\text{ext}}$ in the interior and the exterior
 of $\mathcal{D}$ respectively. Both functions admit continuous extensions onto the closure of
 their respective domains and
 \begin{align*}
  |H_{\text{int,ext}}|\les  (J_h+\|h\|_\infty) \|\chi\|_{\text{Lip}}.
 \end{align*}
 If $\text{supp}(\chi)\ne \partial \mathcal{D}$ then $H_{\text{int}}$ and $H_{\text{ext}}$
 define a single analytic function on $\C\backslash \text{supp}(\chi)$.
 
 Moreover let $\mathcal{D}$ be contained in a square of side length $R$. Then
 \begin{align*}
  |H_{\text{int,ext}}|\les  C\, \log(R)\, \|h\|_\infty\, \|\chi\|_{\text{Lip}}
 \end{align*}
 for some $C>0$.
\end{lemma}

We define the function $\chi^+:\partial D\to [0,1]$ to be Lipschitz continuous.
We also ask that $\chi^+$ has the value 1 on $\partial D \cap \{z\in\C : \re(z)<-\alpha/2 \}$
and $\chi^+$ has the value 0 on $\partial D \cap \{z\in\C : \re(z)>\alpha/2 \}$.
We also define $\chi^-(z) = 1- \chi^+(z)$ and let $L>0$ be such that
$\|\chi^+\|_{\text{Lip}},\|\chi^-\|_{\text{Lip}}\les  L $.
We define
\begin{align*}
 h^\pm(z)=\frac{1}{2\pi \ii}\int_{\partial \mathcal{D}} \frac{h(\zeta) \chi^\pm(\zeta)}{\zeta-z} \dd \zeta.
\end{align*}
The functions $h^+$ and $h^-$ are analytic on $\C\backslash \text{supp}(\chi^+)$ and $\C\backslash
\text{supp}(\chi^-)$ respectively and $h^+(z)+h^-(z)=h(z)$ when $z\in\mathring{D}$ because
of the Cauchy integral. 

With these we define
\begin{align*}
 \mathcal{S}^{-1}[h](z) = S[h](z) := \sum_{n\ges1} h^-(z-n) - \sum_{n\ges0} h^+(z+n).
\end{align*}
This solves the equation $X(z+1)-X(z)=h(z)$ if both sums are convergent.

In order to generalize this method we need to introduce a weigh function. Let
$\phi_a(z)=\me^{a z}+\me^{-a z}$ for some $a>0$ and we denote $\|\phi_a\|_D=\sup_{z\in D}|\phi_a(z)|$.
Then we repeat the above construction with $h(z)=\phi_a(z)\,g(z)$. We define
\begin{align*}
 g_a^\pm(z)=\frac{1}{2\pi \ii\,\phi_a(z)}\int_{\partial \mathcal{D}} \frac{\phi_a(\zeta) h(\zeta) \chi^\pm(\zeta)}{\zeta-z} \dd \zeta.
\end{align*}
By definition we have again $g_a^+(z)+g_a^-(z)=h(z)$ when $z\in\mathring{D}$. So we finally define
\begin{align*}
 S_a[g](z) := \sum_{n\ges1} g_a^-(z-n) - \sum_{n\ges0} g_a^+(z+n).
\end{align*}

\begin{lemma} 
 Let $h\in\mathcal{O}(D)$, $a\ges\frac{\pi}{3 \beta}$ and $r=\max\{ 2\alpha, 2\beta \}$.
 Then $ S_a:\mathcal{O}(D)\to \mathcal{O}(D)$ and
 \begin{align*}
  \|S_a\| \les C \, L \, (1+a^{-1}) \log(r) \|\phi_a\|_D
 \end{align*}
 for some $C>0$
 and $S_a[g]$ is a solution of equation \eqref{eq_lin_diff_eq_model}.
 \label{thm_lin_diff_eq_sol}
\end{lemma}

\begin{proof}
 It is trivial to check that formally $S_a[g]$ is a solution, so we only need to check
 that the sums converge and get the bound for the norm. For $z\in\mathring{D}$ and by
 the previous lemma we have
 \begin{align*}
  |S_a[g](z)| &\les \left| \sum_{n\ges1} g_a^-(z-n) \right| + \left| \sum_{n\ges0} g_a^+(z+n) \right| \\
  &\les   C\,L\, \log(r)\, \|\phi_a\|_D \, \|g\|_\infty \left(\left|\frac{1}{\phi_a(z)} \right|
  + \sum_{n\ges1} \left|\frac{1}{\phi_a(z-n)}  \right| + \sum_{n\ges1}  \left|\frac{1}{\phi_a(z+n)} \right| \right).
 \end{align*}
 Because $a\ges\frac{\pi}{3 \beta}$, $z$ stays far enough from the roots of $\phi_a$ so that
 $\phi_a(z)^{-1}$ stays bounded by $1$. Then both sums can be bounded by some
 constant times the integral $\int_0^\infty\me^{-a s}\dd s$ and from this the result follows.
\myqed
\end{proof}

\subsection{Approximation of fundamental solutions}
\label{ch_approx_of_fund_sols}

The first difference equation we are interested in is the one that the difference of the separatrices satisfies.
For $\delta(\varepsilon;\tau) = W^+(\varepsilon;\tau) - W^-(\varepsilon;\tau)$ and we have
\begin{align*}
 \delta(\varepsilon;\tau+1) &= W^+(\varepsilon;\tau+1)-W^-(\varepsilon;\tau+1) \\
 &= F_\varepsilon(W^+(\varepsilon;\tau)) - F_\varepsilon(W^-(\varepsilon;\tau)) \\
 &= \left( \int_0^1 F_\varepsilon'\left( s\,W^+(\varepsilon;\tau) + (1-s)\, W^-(\varepsilon;\tau)\dd s \right)  \right) \left( W^+(\varepsilon;\tau)-W^-(\varepsilon;\tau) \right) \\
 &= \left( \int_0^1 F_\varepsilon'\left( W^-(\varepsilon;\tau) + s\,\delta(\varepsilon;\tau)  \right) \dd s  \right) \delta(\varepsilon;\tau) ,
\end{align*}
so we write
\begin{align*}
 \delta(\varepsilon;\tau+1) &= A(\varepsilon;\tau)\, \delta(\varepsilon;\tau),
\end{align*}
with $A(\varepsilon;\tau)=\int_0^1 F_\varepsilon'\left( s\,W^+(\varepsilon;\tau) + (1-s)\, W^-(\varepsilon;\tau) \right) \dd s$.
We denote by $U(\varepsilon;\tau)$ the fundamental solution of this equation, i.e. a $2\times2 $ matrix that satisfies
\begin{align}
 U(\varepsilon;\tau+1) = A(\varepsilon;\tau)\cdot U(\varepsilon;\tau),
 \label{eq_variational_delta}
\end{align}
 $\det U(\varepsilon;\tau)=1$ and $ U \cdot (\begin{smallmatrix} 1 \\ 0 \end{smallmatrix} ) = \delta$.

For the second variational equation we define $D(\varepsilon;\tau)=
F_\varepsilon'\left(W^-(\varepsilon;\tau) \right)$ and we denote by
$V(\varepsilon;\tau)=(\Xi(\varepsilon;\tau),\dot{W}^-(\varepsilon;\tau))$ the fundamental solution, 
a $2\times2 $ matrix that satisfies
\begin{align}
 V(\varepsilon;\tau+1) = D(\varepsilon;\tau)\cdot V(\varepsilon;\tau)
 \label{eq_variational_Wminus}
\end{align}
$\det V(\varepsilon;\tau)=1$ and $ V \cdot  (\begin{smallmatrix} 0 \\ 1 \end{smallmatrix} ) = \dot{W}^-$.

The goal of this section is to prove that $U$ exists and can be
approximated by the same function as $V$ with errors that are of the same order.
To this end we denote by $R$ the $2\times 2$ matrix which satisfies
$$A(\varepsilon;\tau)=D(\varepsilon;\tau)+R(\varepsilon;\tau)$$
and we define
$$U(\varepsilon;\tau)=V(\varepsilon;\tau)(I+Q(\varepsilon;\tau))$$
for some $2\times 2$ matrix $Q$ such that $U$ is a fundamental solution of \eqref{eq_variational_delta}.
Then we have
\begin{align*}
 U(\varepsilon;\tau+1)&= V(\varepsilon;\tau+1) ( I+Q(\varepsilon;\tau+1) )\\
 &= D(\varepsilon;\tau) V(\varepsilon;\tau) ( I + Q(\varepsilon;\tau+1) ),\\
 A(\varepsilon;\tau) U(\varepsilon;\tau) &= D(\varepsilon;\tau) U(\varepsilon;\tau) + R(\varepsilon;\tau) U(\varepsilon;\tau) \\
 &= D(\varepsilon;\tau) V(\varepsilon;\tau)(I+Q(\varepsilon;\tau)) + R(\varepsilon;\tau)V(\varepsilon;\tau)(I+Q(\varepsilon;\tau)).
\end{align*}
Here $I$ denotes the identity matrix.
From these we get the equation
\begin{align}
 Q(\varepsilon;\tau+1) - Q(\varepsilon;\tau)= V^{-1}(\varepsilon;\tau)\cdot D^{-1}(\varepsilon;\tau)\cdot R(\varepsilon;\tau) \cdot V(\varepsilon;\tau)(I+Q(\varepsilon;\tau)).
 \label{eq_Q_local_eq}
\end{align}

\begin{definition}
 We define the domains
 \begin{align*}
  \mathcal{M}_0 :&= \left\{ \tau\in\C : |\re(\tau)|\les2, |\im(\tau)|\les \tfrac{\pi}{\varepsilon} - \tfrac{1}{\sqrt{\varepsilon}} \right\}, \\
  \mathcal{M}^\pm :&= \left\{ \tau\in\C : |\re(\tau)|\les2,  \tfrac{\pi}{\varepsilon} - \tfrac{1}{\sqrt{\varepsilon}} \les \pm\im(\tau) \les \tfrac{\pi}{\varepsilon} -\Lambda \right\},  \\
  \mathcal{M}:&= \mathcal{M}_0 \cup \mathcal{M}^+ \cup \mathcal{M}^-.
 \end{align*}
\end{definition}

\begin{definition}
 Let $M\in\C^\omega(\mathcal{M})^{2\times2}$. Then we define
 $$\|M\|_{\sup} = \max_{i,j\in\{1,2\}} \sup_{t\in\mathcal{M}} |M_{ij}(t)|.$$
\end{definition}

\begin{lemma}
 Let $n > 8$ and let $F_\varepsilon$ agree with the normal form up to order $n$.
 Then there exists $C_V>0$ such that
 $$\|V\|_{\sup}=\frac{C_V}{\varepsilon^4}\left( 1 + O(\varepsilon^{1/2}) \right).$$
 \label{thm_var_eq_sol_V_bound}
\end{lemma}

\begin{proof}
By writing $\Xi(\varepsilon;\tau)=(\xi_1(\varepsilon;\tau),\xi_2(\varepsilon;\tau))$
and $\dot{W}^-(\varepsilon;\tau)=(\zeta_1(\varepsilon;\tau),\zeta_2(\varepsilon;\tau))$ 
and using that $\det (\Xi(\varepsilon;\tau), \dot{W}^-(\varepsilon;\tau))=1$ we get 
$$ \xi_1(\varepsilon;\tau) = \frac{\zeta_1(\varepsilon;\tau)}{\zeta_2(\varepsilon;\tau)}
\xi_2(\varepsilon;\tau) + \frac{1}{\zeta_2(\varepsilon;\tau)}. $$
Substituting the above relation in equation \eqref{eq_variational_Wminus} we get the equation
\begin{align*}
 \xi_2(\varepsilon;\tau+1)= \left( D_{21}(\varepsilon;\tau) \frac{\zeta_1(\varepsilon;\tau)}{\zeta_2(\varepsilon;\tau)} + D_{22}(\varepsilon;\tau) \right)\xi_2(\varepsilon;\tau) + \frac{D_{21}(\varepsilon;\tau)}{\zeta_2(\varepsilon;\tau)}.
\end{align*}
Since $\zeta_2$ satisfies the homogeneous part of the above equation, we define
$\xi_2(\varepsilon;\tau)=C(\varepsilon;\tau) \, \zeta_2(\varepsilon;\tau)$ and by substitution we get 
\begin{align*}
 C(\varepsilon;\tau+1) - C(\varepsilon;\tau) = 
 \frac{D_{21}(\varepsilon;\tau)}{\zeta_2(\varepsilon;\tau+1)\,\zeta_2(\varepsilon;\tau)}=: K(\varepsilon;\tau).
\end{align*}

Combining the bounds we have for $\mathcal{D}_1$ and $\mathcal{D}_2$, for all $\tau\in\mathcal{M}_0 $ we 
have $ W^-(\varepsilon;\tau) = \formalSeparatrix_n(\varepsilon;\tau) + O(\varepsilon^{\tfrac{n+1}{2}})$.

We differentiate the approximation of $W^-(\varepsilon\tau)$ to get 
\begin{align*}
 \zeta_2(\varepsilon;\tau) &=  \frac{\varepsilon^2}{4 b_{0,0}}\text{sech}\left( \tfrac{\varepsilon \tau}{2} \right)
 + O(\varepsilon^3\,\tanh\left( \tfrac{\varepsilon\,\tau}{2} \right)^3), \\
 \zeta_1(\varepsilon;\tau) &= O(\varepsilon^3\,\tanh\left( \tfrac{\varepsilon\,\tau}{2} \right)^3) .
\end{align*}
This implies that in $\mathcal{M}_0$ the norm of $\dot{W}^-$ is bounded from below by
a constant independent of $\varepsilon$. We have
\begin{align*}
 |\zeta_2(\varepsilon;\tau)^{-1}| = \frac{C_0'}{\varepsilon^{2}}\left(1 + O(\varepsilon) \right)
\end{align*}
for some $C_0'>0$.
In order to bound $|\zeta_2(\varepsilon;\tau+1)|$ from below we repeat the above process
for $\tau\in\mathcal{M}_0+1$ and we see that only the constant changes, i.e.
\begin{align*}
 |\zeta_2(\varepsilon;\tau+1)^{-1}| = \frac{C_0''}{\varepsilon^2}\left(1 + O(\varepsilon) \right)
\end{align*}
for some $C_0''>0$.

To get a bound for $D_{21}(\varepsilon;\tau)$, we check that $\left[ F_\varepsilon'(x,y)
\right]_{21}= -2a_{0,1} \mu_1\varepsilon - 2b_{0,0} x +2 b_{0,0}^2 x y + O_3(\varepsilon,x,y)$ so
\begin{align*}
 |D_{21}(\varepsilon;\tau)|=C'''_0\varepsilon + O(\varepsilon^2\,\tanh\left( \tfrac{\varepsilon\,\tau}{2} \right))
\end{align*}
for some $C'''_0>0$ and since for any $\tau\in\mathcal{M}_0$ we have $\varepsilon \tanh\left(
\tfrac{\varepsilon\,\tau}{2} \right)=O(\varepsilon^{1/2})$ we have
\begin{align*}
 |K(\varepsilon;\tau)| = \frac{C_0}{\varepsilon^3}\left( 1 + O(\varepsilon^{1/2}) \right)
\end{align*}
for some $C_0>0$. 

For $\tau\in\mathcal{M}^+$ we need to use $\formalSeparatrixSing^-$ to get a bound.
From the bound in $\mathcal{D}_2$ we have $ W^-(\varepsilon;\tau) = \formalSeparatrixSing_n^-(\varepsilon;t)
+ O(\varepsilon^{\tfrac{n-1}{2}})$.

Recall that $\tau=t+\pi\ii/\varepsilon$, $W^-_0(t)=(0,-(b_{0,0}\,t)^{-1})+O(t^{-3})$ and $W^-_n(t)=O(t^{n-1})$.
Thus $\dot{W}^-_0(t)=(0,b_{0,0}^{-1}\,t^{-2})+O(t^{-4})$, $\dot{W}^-_1(t)=O(t^{-2})$ and $\dot{W}^-_n(t)=O(t^{n-2})$. 

For $\tau\in \mathcal{M}^+$, $|t|$ is bounded from above by $\varepsilon^{-\frac{1}{2}}$ and from below by $\Lambda$. So in order to bound $\zeta_2$ from below we need to estimate it for $\im(\tau)\approx \tfrac{\pi}{\varepsilon} - \tfrac{1}{\sqrt{\varepsilon}}$. In this region we get
\begin{align*}
 \dot{W}^-_0(t) = O(\varepsilon) \text{, } \dot{W}^-_1(t)= O(\varepsilon) \text{ and } \dot{W}^-_n(t) = O(\varepsilon^{\tfrac{n}{2}-1}) .
\end{align*}
Using these we get
\begin{align*}
 |\zeta_2(\varepsilon;\tau)^{-1}| = \frac{C_1'}{\varepsilon}\left(1 + O(\varepsilon^{1/2}) \right)
\end{align*}
for some $C_1'>0$. As above the same process on $\mathcal{M}^++1$ gives the same bound with
a different constant for $|\zeta_2(\varepsilon;\tau+1)^{-1}|$.

Finally, on $\mathcal{M}^+$ we have $|D_{21}(\varepsilon;\tau)|=C_1''(1+ O(\varepsilon))$ so we get
\begin{align*}
 |K(\varepsilon;\tau)| = \frac{C_1}{\varepsilon^2}\left( 1 + O(\varepsilon^{1/2}) \right)
\end{align*}
for some $C_1>0$. Due to the real symmetry we get exactly the same bounds on $\mathcal{M}^-$.

Now that we know that $K$ is bounded on $\mathcal{M}$ we can use Lemma \ref{thm_lin_diff_eq_sol}
to get the existence of $C$.
We set $a=\varepsilon/2$ and we have $r=2 \pi \varepsilon^{-1}$, so $\|S_a\|\les c' \varepsilon^{-2} $
and 
$| C(\varepsilon;\tau) |\les c''\, \varepsilon^{-5}$. From this we get that 
$$\xi_2(\varepsilon;\tau) = C(\varepsilon;\tau) \zeta_2(\varepsilon;\tau) = \frac{c_2}{\varepsilon^3}
\left( 1 + O(\varepsilon^{1/2}) \right)$$
and
\begin{align*}
\xi_1(\varepsilon;\tau) = \frac{1}{\zeta_2(\varepsilon;\tau)} + \frac{\xi_2(\varepsilon;\tau) \zeta_1(\varepsilon;\tau)}{\zeta_2(\varepsilon;\tau)} = \frac{c_1}{\varepsilon^4}\left( 1 + O(\varepsilon^{1/2}) \right). 
\end{align*} 
The maximum of these bounds gives the result.
\myqed
\end{proof}

\begin{lemma}
 Let $n\ges20$ and let $F_\varepsilon$ agree with the normal form up to order $n$.
 Then there exists $\varepsilon_0>0$ such that for all $0<\varepsilon\les\varepsilon_0$
 there exists a constant $C_Q>0$ such that $\left\|Q \right\|_{\sup} \les C_Q  \,
 \varepsilon^{\frac{n-19}{2}}\left( 1 + O(\varepsilon^{1/2}) \right)$ on $\mathcal{M}$.
 \label{thm_var_eq_sol_pol_close}
\end{lemma}

\begin{proof}
We know that $\|D^{-1}\|_{\sup}=1 + O(\varepsilon^{1/2})$, also since $\det V = 1$ we
have $\|V^{-1}\|_{\sup}= C_V \, \varepsilon^{-4}\left( 1 + O(\varepsilon^{1/2}) \right)$. We define
$$M=V^{-1} \cdot D^{-1} \cdot R  \cdot V.$$
Then equation \eqref{eq_Q_local_eq} becomes
\begin{align*}
 Q(\tau+1) - Q(\tau)=  M(\tau) + M(\tau) \cdot Q (\tau).
\end{align*}
From this we get
\begin{align*}
 Q(\tau)= S_a\big[ M \big](\tau) + S_a \big[M\cdot Q\big](\tau).
\end{align*}
We define
$$\mathbf{X}:Q\mapsto S_a\big[ M \big] + S_a \big[M \cdot Q\big].$$

If $W^+$ and $W^-$ coincide with the normal form up to order $n$, then there exists $C_n>0$
such that $\| R \|_{\sup}=  C_n \, \varepsilon^{\frac{n+1}{2}}\left( 1 + O(\varepsilon^{1/2}) \right)$.
We combine this with the bounds for $D$ and $V$ and we find that there exists $C_M>0$
such that $\| M \|_{\sup} =  C_M \, \varepsilon^{\frac{n-15}{2}}\left( 1 + O(\varepsilon^{1/2}) \right)$.
Recall that $\|S_a\|\les c' \varepsilon^{-2} $, so
$$ \left\|S_a\big[ M \big] \right\|_\infty \les C'_M \, \varepsilon^{\frac{n-19}{2}}
\left( 1 + O(\varepsilon^{1/2}) \right)$$
and
$$ \left\|S_a\big[ M \cdot Q\big] \right\|_\infty \les C'_M \, \varepsilon^{\frac{n-19}{2}}
\left( 1 + O(\varepsilon^{1/2}) \right) \left\|  Q \right\|_\infty . $$

Then for $n\ges20$ there exists a neighbourhood of the origin
$\mathcal{V}_c = \{ x\in \C^\omega(\mathcal{M})^{2\times2} : \|x\|_\infty \les c \left\|S_a [ M ] \right\|_\infty \}$
for big enough $c$ and  small enough $\varepsilon$ in which the operator $\mathbf{X}$ is a contraction.
From this the result follows.
\myqed
\end{proof}

\begin{corollary}
 For all $\tau\in\mathcal{M}$
 $$U(\varepsilon,\tau) = V(\varepsilon,\tau) + O(\varepsilon^{\frac{n-27}{2}}).$$
 Since $U(\varepsilon,\tau) =   (\Psi(\varepsilon,\tau),\Phi(\varepsilon,\tau))$
 this implies that 
 $\Phi(\varepsilon,\tau)  = \dot{W}^-(\varepsilon,\tau) + O(\varepsilon^{\frac{n-27}{2}}).$
 \label{thm_initial_bounds_for_fund_sol}
\end{corollary}

Combining the above corollary with Lemma \ref{thm_var_eq_sol_V_bound} we get a bound for $U$.

\begin{corollary}
 Let $n > 8$ and let $F_\varepsilon$ agree with the normal form up to order $n$.
 Then there exists $C_U>0$ such that
 $$\|U\|_{\sup}=\frac{C_U}{\varepsilon^4}\left( 1 + O(\varepsilon^{1/2}) \right).$$
 \label{thm_var_eq_sol_U_bound}
\end{corollary}

\section{Sharper bounds}
\label{ch_sharper_bounds}

\subsection{Exponentially small upper bound for the splitting}

With everything that is known up to this point we can prove that the splitting
admits an exponentially small upper bound.

\begin{lemma}
 Let $\mathcal D = \{ z\in\C :|\re(z)| \les2, |\im(z)|\les \tfrac12 \}$.
 For all $\tau \in \mathcal D $ there exists a constant $C>0$ such that
 $$|\delta(\varepsilon;\tau)| \les C \varepsilon^{-2} \me^{-\frac{ 2\pi^2}{\varepsilon} }. $$
 \label{thm_splitting_bound}
\end{lemma}
Recall that we have defined $\delta(\varepsilon;\tau) = W^+(\varepsilon;\tau)-W^-(\varepsilon;\tau) $.
Before we proceed to prove this lemma we need some results on real analytic periodic functions.

\paragraph{Real analytic periodic functions in a rectangular domain} ~

\begin{lemma}
Let $D=\{z\in\C : |\re(z)|\les \alpha, |\im(z)|\les \beta \} $ for some
$\alpha, \beta\ges1$ and let $g$ be a real analytic function on $D$ and
continuous on $\partial D$ such that $g(\tau+1)=g(\tau)$ when both $\tau$
and $\tau+1$ are in $D$. Moreover, we assume that there exists
$\tau_h\in [-\alpha,\alpha]$ such that $g(\tau_h) = 0$.
We write $g$ as a Fourier series:
$$g(\tau) = g_0 + \sum_{n\ges1} g_n\, \me^{-2\pi n \ii \tau} + \sum_{n\ges1} \overline{g_n}\, \me^{2\pi n \ii \tau},$$
with $g_n\in\C$. Then it is true that
$$ |g_n| \les \|g\|_\infty \me^{- 2\pi \beta n} $$
for all $n\in\N$ and
$$|g_0| \les 4 \|g\|_\infty \me^{- 2\pi \beta }.$$
\end{lemma}

\begin{proof}
By setting $\tau = \ii \beta$ we get
$$g(\ii \beta) = g_0 + \sum_{n\ges1} g_n\, \me^{2\pi \beta n} + \sum_{n\ges1} \overline{g_n}\, \me^{- 2\pi \beta n} $$
and this implies that
$$ |g_n| \les \|g\|_\infty \me^{- 2\pi \beta n} $$
for all $ n\ges1$.

From the equation $g(\tau_h) = 0$ we get
$$ |g_0| \les 2\sum_{n\ges1} |g_n|.$$
This sum is a geometric progression so 
$$ |g_0| \les 2\|g\|_\infty \me^{- 2\pi \beta }\frac{1}{1-\me^{- 2\pi \beta }}
\les 4 \|g\|_\infty \me^{- 2\pi \beta } . \qedhere $$
\myqed
\end{proof}

\begin{corollary}
 Let $g$ and $D$ be as described above. Then for all $\tau \in [-\alpha,\alpha]$ it is true that
 $$ |g(\tau)| \les 8\|g\|_\infty \me^{- 2\pi \beta }. $$
\label{thm_periodic_function_rect_domain}
\end{corollary}

\paragraph{The components of \texorpdfstring{$\delta$}{delta}} \hspace{0em}

Let $\Psi$ and $\Phi$ be such that $U(\varepsilon;\tau) = 
(\Psi(\varepsilon;\tau),\Phi(\varepsilon;\tau))$, with $U$ the fundamental
solution defined in Section \ref{ch_approx_of_fund_sols}.
Then there exist two functions $\Theta(\varepsilon;\tau)$ and $q(\varepsilon;\tau)$
such that
$$\delta(\varepsilon;\tau) = \Theta(\varepsilon;\tau)\,\Psi(\varepsilon;\tau)  + q(\varepsilon;\tau)\, \Phi(\varepsilon;\tau).$$
Then we have
\begin{equation}
 \Theta(\varepsilon;\tau) = \omega(\delta(\varepsilon;\tau),\Phi(\varepsilon;\tau))
 \label{eq_Theta_definition}
\end{equation}
and
\begin{align*}
 \Theta(\varepsilon;\tau+1) &= \omega(\delta(\varepsilon;\tau+1),\Phi(\varepsilon;\tau+1)) \\
 &= \omega(A(\varepsilon;\tau)\,\delta(\varepsilon;\tau),A(\varepsilon;\tau)\,\Phi(\varepsilon;\tau)) \\
 &= \omega(\delta(\varepsilon;\tau),\Phi(\varepsilon;\tau)) \\
 &= \Theta(\varepsilon;\tau).
\end{align*}
Similarly we get that $ q(\varepsilon;\tau + 1) = q(\varepsilon;\tau) $.

\begin{lemma}
 Let $\mathcal D = \{ z\in\C :|\re(z)| \les2, |\im(z)|\les \tfrac12 \}$.
 For all $\tau \in \mathcal D $ there exists a constant $C>0$ such that
 $$|\Theta(\varepsilon;\tau)|,|q(\varepsilon;\tau)| \les C \me^{-\frac{ 2\pi^2}{\varepsilon} }.$$
 \label{thm_splitting_comps_bound}
\end{lemma}

\begin{proof}

The map is area-preserving, so there has to be a homoclinic point $W^-(\varepsilon;\tau_h)$
such that $\delta(\varepsilon;\tau_h)=0$. Because $\Psi$ and $\Phi$ are linearly independent
this implies that $\Theta(\varepsilon;\tau_h) = q(\varepsilon;\tau_h)=0$.

Both $\Theta$ and $q$ are defined in a rectangular domain with $\alpha =2$ and $\beta =
\frac{\pi}{\varepsilon} - \Lambda$. We apply Corollary \ref{thm_periodic_function_rect_domain}
and we get that there exists a constant $C>0$ such that for all $\tau\in[-2,2] $ it holds that
$$|\Theta(\varepsilon;\tau)|,|q(\varepsilon;\tau)| \les C' \me^{-\frac{ 2\pi^2}{\varepsilon} }.$$
We can extend this bound to the whole $\mathcal{D}$ by increasing the constant since
$\mathcal{D}$ is independent of $\varepsilon$.
\end{proof}

\begin{proof}[Proof of Lemma \ref{thm_splitting_bound}] 
We combine Lemma \ref{thm_splitting_comps_bound} with Corollary \ref{thm_var_eq_sol_U_bound}.
\end{proof}

\subsection{Variational equations revisited}

In order to prove Lemma \ref{thm_var_eq_sol_pol_close} we used the fact that the stable
and unstable solutions can be approximated by the same formal series. This gives an error
that is polynomially small with $\varepsilon$. However, we saw in the previous section
that the splitting is actually exponentially small. We can now use this result to get a
sharper bound on the difference of the two fundamental solutions.

\begin{lemma}
 Let $\mathcal D = \{ z\in\C :|\re(z)| \les2, |\im(z)|\les \tfrac12 \}$.
 Then there exists $C>0$ such that on $ \mathcal D$ it is true that
 $$ \| U - V\|_{\sup}  \les C \varepsilon^{-16} \me^{-\frac{ 2\pi^2}{\varepsilon} } \big( 1 + O(\varepsilon^{1/2}) \big), $$
 where $U$ and $V$ are the fundamental solutions defined in Section \ref{ch_approx_of_fund_sols}.
 \label{thm_var_eq_sol_exp_bound}
\end{lemma}

\begin{proof}
The proof is essentially the same as the proof of Lemma \ref{thm_var_eq_sol_pol_close}. Here we restate the main points.

By definition we have
\begin{align*}
  A(\varepsilon;\tau) &=\int_0^1 F_\varepsilon'\left( s\,W^+(\varepsilon;\tau)+ (1-s)\, W^-(\varepsilon;\tau) \right) \dd s  \\
  & =\int_0^1 F_\varepsilon'\left( W^-(\varepsilon;\tau)+ s\, \delta(\varepsilon;\tau) \right) \dd s .
\end{align*}
Then
\begin{align*}
 R(\varepsilon;\tau) &= A(\varepsilon;\tau)- D(\varepsilon;\tau) \\
 & = \int_0^1 \Bigg( F_\varepsilon' \Big( W^-(\varepsilon;\tau)+ s\, \delta(\varepsilon;\tau) \Big) - F_\varepsilon' \Big( W^-(\varepsilon;\tau) \Big)  \Bigg) \dd s 
\end{align*}
and by using Taylor's theorem and the bound for $\delta$ we get that there exists $C>0$
such that for all $\tau\in\mathcal D$ it holds that
$$ |R(\varepsilon;\tau)| \les C \varepsilon^{-2} \me^{-\frac{ 2\pi^2}{\varepsilon} }. $$

We have
$$ M = V^{-1} \cdot D^{-1} \cdot R \cdot V$$
and since $M$ is a function defined on $\mathcal D$ we get
\begin{align*}
 \|V\|_{\sup},\|V^{-1}\|_{\sup} & \les C \; \varepsilon^{-4} \big( 1 + O(\varepsilon^{1/2}) \big), \\
 \|D^{-1}\|_{\sup} & \les 1 + O(\varepsilon^{1/2}), \\
 \|S_a\|_{\sup} &\les C \; \varepsilon^{-2}.
\end{align*}
Recall that we have set $ a = \varepsilon/2 $. Then
$$ \|S_a[M]\|_{\sup} \les C \varepsilon^{-12} \me^{-\frac{ 2\pi^2}{\varepsilon} } \big( 1 + O(\varepsilon^{1/2}) \big). $$
Now by the same contraction mapping argument we get 
$$ \| Q \|_{\sup} \les C \varepsilon^{-12} \me^{-\frac{ 2\pi^2}{\varepsilon} } \big( 1 + O(\varepsilon^{1/2}) \big), $$
which implies
$$ \| U - V\|_{\sup} = \| V\cdot Q \|_{\sup} \les C \varepsilon^{-16} \me^{-\frac{ 2\pi^2}{\varepsilon} }
\big( 1 + O(\varepsilon^{1/2}) \big). \qedhere $$
\myqed
\end{proof}

\section{Asymptotic expansion of the separatrix splitting}
\label{ch_asym_exp_of_splitting}

Recall that in \eqref{eq_Theta_definition} we have defined the periodic function
$$\Theta(\varepsilon;\tau) = \omega(\delta(\varepsilon;\tau),\Phi(\varepsilon;\tau)).$$
We define
$$\Theta^-(\varepsilon;\tau) = \omega(\delta(\varepsilon;\tau),\dot{W}^-(\varepsilon;\tau)).$$
Note that unlike $\Theta$, $\Theta^-$ is not periodic, as $\delta$ and $\dot{W}^-$
do not satisfy the same equation. However, we will see that $\Theta^-$ is close enough
to $\Theta$, so that they can be approximated by the same asymptotic series.

We write $\Theta$ as a Fourier series
\begin{align*}
 \Theta(\varepsilon;\tau) = c_0 + \sum_{n\ges1} c_n(\varepsilon) \me^{-2\pi n \ii\tau} + \sum_{n\ges1} \overline{c_n(\varepsilon)} \me^{2\pi n \ii\tau}.
\end{align*}
Then
\begin{align}
 \Theta(\varepsilon;t+\tfrac{\pi}{\varepsilon}\ii) =c_0 + \sum_{n\ges1} c_n(\varepsilon) \me^{\frac{2\pi^2n}{\varepsilon}} \me^{-2\pi n \ii t} + \sum_{n\ges1} \overline{c_n(\varepsilon)} \me^{-\frac{2\pi^2n}{\varepsilon}} \me^{2\pi n \ii t}.
 \label{eq_Theta_fourier_high}
\end{align}

\subsection{Asymptotic series for $\Theta$}

We define
\begin{equation}
 L_1(\nu)=\{t\in\C:\im(t)=-\nu \text{ and } |\re(t)|\les\tfrac12 \}
 \label{eq_L1_definition}
\end{equation}
and we fix $\nu = - (M+2) (2\pi)^{-1} \log(\varepsilon)$.
Then we have $\me^{2\pi\ii t} = O(\varepsilon^{-M-2})$.
We will estimate $\Theta$ and $\Theta^-$ on that line.

\begin{lemma}
 There exists a series formal in $\varepsilon$
 $$ \tilde{\Theta}(\varepsilon;t) = \sum_{n\ges0}\varepsilon^n \zeta_n(t)  $$
 with $\zeta_n$ functions analytic in the semistrip $|\re(t)|\les2$, $\im(t)\les \Lambda$,
 such that for all $ t\in L_1(\nu)$ it holds that \footnote
{
Recall that $N=6M+39$.
}
 $$ \Theta(\varepsilon;t+\tfrac{\pi}{\varepsilon}\ii) = \sum_{n\ges0}^N\varepsilon^n \zeta_n(t) + O(\varepsilon^{2M+3}). $$
 \label{thm_Theta_asymptotic_series}
\end{lemma}

\begin{proof}
 We define the formal series
 $$\tilde{\delta}(\varepsilon;t) = \sum_{n\ges0} \varepsilon^n \delta_n(t),$$
 with $\delta_n(t)= W_n^+(t)-W_n^-(t)$, where $W_n^\pm$ were defined in Section \ref{ch_borel_trans}. 
 We denote by $\tilde{\delta}_N$ the sum of the first $N+1$ terms.
 
 Corollary \ref{thm_initial_bounds_for_fund_sol} implies that for $t\in L_1(\nu)$ it holds
 $$\Phi(\varepsilon;t+\tfrac{\pi}{\varepsilon}\ii) = \dot{W}^-(\varepsilon;t+
 \tfrac{\pi}{\varepsilon}\ii)  + O(\varepsilon^{3M+6}).$$

 By Lemma \ref{thm_local_lemma_matching}, for $t\in L_1(\nu)$ we also know that 
 $$W^\pm(\varepsilon;t+\tfrac\pi\varepsilon\ii) = \sum_{n=0}^N \varepsilon^n W_n^\pm(t) + O(\varepsilon^{3M+19}).$$
 Subtracting $W^-$ from $W^+$ we get
 $$\delta(\varepsilon;t+\tfrac{\pi}{\varepsilon}\ii)=\tilde{\delta}_N(\varepsilon;t)  + O(\varepsilon^{3M+19}).$$
 We know from Lemma \ref{thm_borel_trans_form_sol_sing} that $\delta_n(t)=O(t^{n+2} \me^{-2\pi\ii t})$.
 Then for $t\in L(\nu)$ we have $\varepsilon^n\delta_n(t)=O(\varepsilon^{-M-3+n})$ which implies that for all $N\in\N$
 $$ \tilde{\delta}_N(\varepsilon;t) = O(\varepsilon^{-M-3}). $$
 We saw that $\formalSeparatrixSing_n(t)=O(t^{n-1})$. The function $W^-_n$ is the Borel-Laplace
 sum of $\formalSeparatrixSing_n$ and this implies that it can be decomposed into a polynomial
 of degree $n-1$ and a function in the class $O(t^{-1})$. This implies that $\dot{W}_n^-(t) = O(|t|^{n-2})$.
 Using this fact and the above bounds, we get
 \begin{align*}
  \Theta(\varepsilon;t+\tfrac{\pi}{\varepsilon}\ii) &= \omega(\delta(\varepsilon;t+\tfrac{\pi}{\varepsilon}\ii),\Phi(\varepsilon;t+\tfrac{\pi}{\varepsilon}\ii)) = \\
  &= \omega\left(\tilde{\delta}_N(\varepsilon;t) + O(\varepsilon^{3M+19})  , \sum_{n=0}^N \varepsilon^n \dot{W}_n^-(t) + O(\varepsilon^{3M+6}) \right) \\
  &= \omega\left(\tilde{\delta}_N(\varepsilon;t) , \sum_{n=0}^N \varepsilon^n \dot{W}_n^-(t) \right) + \omega\left(\tilde{\delta}_N(\varepsilon;t) ,   O(\varepsilon^{3M+6}) \right) \\
  & \hphantom{=}\; + \omega\left( O(\varepsilon^{3M+19})  , \sum_{n=0}^N \varepsilon^n \dot{W}_n^-(t) \right) + \omega\left( O(\varepsilon^{3M+19})  , O(\varepsilon^{3M+6}) \right) \\
  &= \omega\left(\tilde{\delta}_N(\varepsilon;t) , \sum_{n=0}^N \varepsilon^n \dot{W}_n^-(t) \right) + O(\varepsilon^{2M+3}) \\
  &= \sum_{n=0}^N \varepsilon^n \sum_{m=0}^n \omega\left( \delta_m(t), \dot{W}_{n-m}^-(t)  \right) + O(\varepsilon^{2M+3}).
 \end{align*}
 We define $\zeta_n(t) = \sum_{m=0}^n\omega(\delta_m(t),\dot{W}_{n-m}^-(t))$, which is defined
 on the above mentioned semistrip since both $\delta_i$ and $W_{i}^-$ are defined there.
\myqed
\end{proof}

By the definition of $\tilde{\Theta}$ we get
\begin{align*}
 \tilde{\Theta}(\varepsilon;t+1) &= \omega\left(\tilde{\delta}(\varepsilon;t+1),
 \dot{\formalSeparatrixSing}^-(\varepsilon;t+1) \right) \\
 &= \omega\left( F_\varepsilon(\formalSeparatrixSing^+(\varepsilon;t)) - 
 F_\varepsilon(\formalSeparatrixSing^-(\varepsilon;t)),  F_\varepsilon'(\formalSeparatrixSing^-(\varepsilon;t))
 \cdot \dot{\formalSeparatrixSing}^-(\varepsilon;t) \right) \\
 &= \omega\left( F_\varepsilon'(\formalSeparatrixSing^-(\varepsilon;t)) \cdot \tilde{\delta}(\varepsilon;t),
 F_\varepsilon'(\formalSeparatrixSing^-(\varepsilon;t)) \cdot \dot{\formalSeparatrixSing}^-(\varepsilon;t) \right) \\
 &\quad+ \omega\left(\tilde{\mathfrak{V}}(\varepsilon;t) ,  F_\varepsilon'(\formalSeparatrixSing^-(\varepsilon;t))
 \cdot \dot{\formalSeparatrixSing}^-(\varepsilon;t) \right) \\
 &= \tilde{\Theta}(\varepsilon;t) + \omega\left(\tilde{\mathfrak{V}}(\varepsilon;t) , 
 \dot{\formalSeparatrixSing}^-(\varepsilon;t+1) \right),
\end{align*}
with
\begin{align*}
  \tilde{\mathfrak{V}}(\varepsilon;t) & =  F_\varepsilon(\formalSeparatrixSing^+(\varepsilon;t)) -
  F_\varepsilon(\formalSeparatrixSing^-(\varepsilon;t)) -  F_\varepsilon'(\formalSeparatrixSing^-(\varepsilon;t))
  \cdot \tilde{\delta}(\varepsilon;t)\\
  & = F_\varepsilon(\formalSeparatrixSing^-(\varepsilon;t)+\tilde{\delta}(\varepsilon;t)) - 
  F_\varepsilon(\formalSeparatrixSing^-(\varepsilon;t)) -  F_\varepsilon'(\formalSeparatrixSing^-(\varepsilon;t))
  \cdot \tilde{\delta}(\varepsilon;t).
\end{align*}

\begin{lemma}
 $\tilde{\mathfrak{V}}$ can be written as
\begin{align*}
 \tilde{\mathfrak{V}}(\varepsilon;t)=\sum_{n\ges0}\varepsilon^n\mathscr{V}_n(t)
\end{align*}
where the coefficients are analytic in the semistrip $|\re(t)|\les2$, $\im(t)\les \Lambda$
and admit an upper bound of the form $\mathscr{V}_n(t) = O(t^{n+4}\me^{-4\pi\ii t})$.
\label{thm_V_existence_and_bound}
\end{lemma}

Using this lemma we can write
\begin{align}
 \zeta_n(t+1) = \zeta_n(t) + \sum_{m=0}^n \omega(\mathscr{V}_m(t),\dot{W}_{n-m}^-(t+1)).
 \label{eq_zeta_recurence_relations}
\end{align}

\begin{proof}
For the proof we need to show that in a sense the following relation holds:
$$ F_\varepsilon(\formalSeparatrixSing^- +\tilde{\delta} ) = F_\varepsilon(\formalSeparatrixSing^- )
+ F_\varepsilon'(\formalSeparatrixSing^- )\cdot \tilde{\delta}  + O\big(\tilde{\delta}^2 \big). $$
This relation would be a direct corollary of Taylor's theorem if every quantity involved was a function.
However, since $\formalSeparatrixSing^-$ and $\tilde{\delta}$ are formal series, we need to
construct a more careful argument.

For the proof we need to introduce some new notation.
Let $H:\C^2\to\C^2$ be analytic in a neighbourhood of the origin. We write its Taylor series as
$$H(W+v)=\sum_{n\ges0}\frac{1}{n!}H^{(n)}(W;\underbrace{v,\dots,v}_{n-\text{times}}) , $$
where $H^{(n)}$ has to be viewed as a symmetric tensor. Notice that the tensor is not
linear with respect to its first argument.

Using this notation we write $H'(W)\cdot u$ as $H^{(1)}(W;u)$ and we have
$$H^{(1)}(W+v;u)=\sum_{n\ges0}\frac{1}{n!}H^{(n+1)}(W;\underbrace{v,\dots,v}_{n-\text{times}},u).$$
In general, it holds
$$H^{(n)}(W+u,v_1,\dots,v_n)=H^{(n)}(W,v_1,\dots,v_n)+H^{(n+1)}(W,v_1,\dots,v_n,u)+O(v^n u^2).$$

We also need a slight generalization of the multi-index notation. We define the set
$$\mathcal{P}(n,m):=\left\{(k_1,\dots,k_m)\in\N^{k}:\sum_{i=1}^{m} k_i=n\right\},$$
which is the set of all $m$-tuples of non-negative integers whose sum is $n$. Then we define the sets
$$\hat{\mathcal{P}}(n,m):=\left\{(j;k_1,\dots,k_m)\in\N^{m}:\sum_{i=1}^{m} k_i=n, 1\les j\les m\right\},$$
$$\hat{\mathcal{P}}_o(n,m):=\left\{(m;k_1,\dots,k_{m-1},k_{m})\in\N^{m}:
\sum_{i=1}^{m+1} k_i=n\right\}\subset \hat{\mathcal{P}}(n,m).$$

Using these we define
$$ \wcomb_{(k_1,\dots,k_m)}=(\dot{W}_{k_1}^-,\dots,\dot{W}_{k_m}^-),$$
$$ \wcomb_{(j;k_1,\dots,\hat{k}_j,\dots,k_m)}=(\dot{W}_{k_1}^-,\dots,\delta_{k_j},\dots,\dot{W}_{k_m}^-).$$

With this notation and having in mind that for any bounded bilinear map $A$ because of
\eqref{eq_bound_for_deltaN}, it holds $A(\delta_i(t), \delta_j(t))=O(t^{i+j+4}\,\me^{-4\pi\ii t})$.
We expand in Taylor series dropping terms that are quadratic in any $\delta_i$ and we get
\begin{align*}
  F_\varepsilon(\formalSeparatrixSing^-+\tilde{\delta}) &= \sum_{n\ges0} \varepsilon^n \Bigg(  F_n(W_0^- + \delta_0) \\
  &\hphantom{=}+ \sum_{m=0}^{n-1} \sum_{k=1}^{n-m}\sum_{p\in\mathcal{P}(n-m,k)} \frac{1}{k!}  F_m^{(k)}\left( W_0^- + \delta_0; \wcomb_p \right)  \\
  &\hphantom{=}+ \sum_{m=0}^{n-1} \sum_{k=1}^{n-m}\sum_{p\in\hat{\mathcal{P}}(n-m,k)} \frac{1}{k!}  F_m^{(k)}\left( W_0^- + \delta_0; \wcomb_p \right) \\
  &\hphantom{=}+ O(t^{n+4}\me^{-4\pi\ii t}) \Bigg),
\end{align*}
\begin{align*}
  F_\varepsilon(\formalSeparatrixSing^-(\varepsilon;t)) &=\sum_{n\ges0} \varepsilon^n \Bigg(  F_n(W_0^-)  \\
  &\hphantom{=} + \sum_{m=0}^{n-1} \sum_{k=1}^{n-m}\sum_{p\in\mathcal{P}(n-m,k)} \frac{1}{k!}  F_m^{(k)}\left( W_0^-; \wcomb_p \right) \\
  &\hphantom{=}+ O(t^{n+4}\me^{-4\pi\ii t}) \Bigg)
\end{align*}
and
\begin{align*}
  F_\varepsilon'(\formalSeparatrixSing^-(\varepsilon;t))\cdot \tilde{\delta}(\varepsilon;t)  = \sum_{n\ges0} \varepsilon^n \Bigg(\sum_{m=0}^n F_m^{(1)}(W_0^-; \delta_{n-m}) \\
 + \sum_{m=0}^{n-1} \sum_{k=1}^{n-m}\sum_{p\in\mathcal{P}(n-m,k)} \frac{1}{k!}  F_m^{(k+1)}\left( W_0^-; \wcomb_p,\delta_0 \right) \\
 + \sum_{m=0}^{n-1} \sum_{k=2}^{n-m}\sum_{p\in\hat{\mathcal{P}}_o(n-m,k)} \frac{1}{(k-1)!}  F_m^{(k)}\left( W_0^-; \wcomb_p \right) \\
 + O(t^{n+4}\me^{-4\pi\ii t}) \Bigg).
\end{align*}

We fix $n$, $m$, $k$ and $p\in\mathcal{P}(n-m,k)$. Then by Taylor's theorem we have that
\begin{align*}
 \frac{1}{k!}  F_m^{(k)}\left( W_0^- + \delta_0; \wcomb_p \right) -
 \frac{1}{k!}  F_m^{(k)}\left( W_0^-; \wcomb_p \right) - \frac{1}{k!} 
 F_m^{(k+1)}\left( W_0^-; \wcomb_p, \delta_0 \right)
\end{align*}
is of order $\delta_0^2$.

Then we fix $n$, $m$, set $k=1$ and similarly we get that
\begin{align*}
  F_m^{(1)}(W_0^- + \delta_0;\delta_{n-m}) -  F_m^{(1)}(W_0^-; \delta_{n-m})
\end{align*}
is of order $\delta_0 \delta_{n-m}$.

Finally we fix $n$, $m$, $k>1$ and since $F_m^{(k)}$ is a symmetric tensor we have 
\begin{align*}
 &\sum_{p\in\hat{\mathcal{P}}(n-m,k)} \frac{1}{k!}  F_m^{(k)}\left( W_0^-
 + \delta_0; \wcomb_p \right) - \sum_{p\in\hat{\mathcal{P}}_o(n-m,k)} \frac{1}{(k-1)!}
 F_m^{(k)}\left( W_0^-; \wcomb_p \right) \\
 &=\sum_{p\in\hat{\mathcal{P}}_o(n-m,k)} \frac{1}{(k-1)!}  F_m^{(k)}\left( W_0^- + \delta_0; \wcomb_p \right)
 - \frac{1}{(k-1)!}  F_m^{(k)}\left( W_0^-; \wcomb_p \right).
\end{align*}
This implies that each term of the sum is of order $\delta_0 \delta_{j}$ for some $j\in\{1,\dots,n-m\}$.

The above arguments show that for all $n$, $\mathscr{V}_n(t) = O(t^{n+4}\me^{-4\pi\ii t})$.
\myqed
\end{proof}

\subsection{The first Fourier coefficient of \texorpdfstring{$\Theta$}{Theta}}

We can compute the fist Fourier coefficient $c_1(\varepsilon)$ of the function $\Theta$ using equation
\eqref{eq_Theta_fourier_high} in the form of the integral 
\begin{equation}
 \theta(\varepsilon):=\int_{L_1(\nu)}\me^{2\pi\ii t} \Theta(\varepsilon;t+\tfrac{\pi}{\varepsilon}\ii)
 \dd t = c_1(\varepsilon) \me^{\frac{2\pi^2}{\varepsilon}},
 \label{eq_first_fourier_coeff}
\end{equation}
where $L_1(\nu)$ is the line segment defined in \eqref{eq_L1_definition}. Note that the
value of this integral is independent of the choice of $\nu$ as long as $L_1(\nu)$ remains
inside the domain of analyticity of $\Theta$.

Lemma \ref{thm_Theta_asymptotic_series} implies that
\begin{align}
\begin{split}
 \theta(\varepsilon) &= \int_{L_1(\nu)}\me^{2\pi\ii t} \left( \tilde{\Theta}_N(\varepsilon;t)
 + O(\varepsilon^{2M+3}) \right) \dd t \\
 &= \sum_{n=0}^N\varepsilon^n \left( \int_{L_1(\nu)}\me^{2\pi\ii t} \zeta_n(t) \dd t \right)
 + O(\varepsilon^{M+1}).
\end{split}
\label{eq_iner_label_01}
\end{align}

\begin{lemma}
 We define
 \begin{align*}
  \rho_n(t) = \int_{t}^{t+1}\me^{2\pi\ii s}\zeta_n(s) \dd s
 \end{align*}
 for all $t$ in the semistrip $\im t < -\Lambda$ and $|\re t|\les2$. 
 It is true that
 \begin{align*}
  \rho_n(t) = \theta_n + O(t^{n+2}\me^{-2\pi\ii t}).
 \end{align*}
\end{lemma}

\begin{proof}
We define
$$L_1^-(\mu) := \{ t\in\C : \im t\les\mu, |\re t| \les \tfrac12 \} $$
and
$$L^-(\mu) := L_1^-(\kappa)\cup \big(L_1^-(\kappa)+1\big). $$

Equation \eqref{eq_zeta_recurence_relations} implies that
\begin{align*}
 \me^{2\pi\ii (t+1)}\zeta_n(t+1) = \me^{2\pi\ii t}\zeta_n(t) + r_n(t),
\end{align*}
with
\begin{align*}
 r(t) = \sum_{m=0}^n \omega(\me^{2\pi\ii t}\mathscr{V}_m(t),\dot{W}_{n-m}^-(t+1))
\end{align*}
and due to Lemma \ref{thm_V_existence_and_bound}, $r_n(t) = O(t^{n+2}\me^{-2\pi\ii t})$.
All of the above functions are analytic in $L^-(\nu)$.

Then $\rho_n$ satisfies the equation
\begin{align*}
 \rho_n(t+1) = \rho_n(t) + \int_{t}^{t+1}r_n(s) \dd s,
\end{align*}
which has as a solution
\begin{align*}
 \rho_n(t) = \theta_n + \int_{-\ii \infty}^{t}r_n(s) \dd s
\end{align*}
for some constant $\theta_n$.

Since we know the bound for $r_n$ and 
\begin{align*}
 \int_{\infty}^{|t|} s^{n+2} \me^{-2\pi s} \dd s \les C_{n} |t|^{n+2} \me^{-2\pi |t|},
\end{align*}
we get that for all  $t\in L_1^-(\nu)$
\begin{align*}
 \rho_n(t) = \theta_n + O(t^{n+2}\me^{-2\pi\ii t}).
\end{align*}
\myqed
\end{proof}

\begin{remark}
 For the first constant $\theta_0$ we have
 $$ \theta_0 = \lim_{\nu\to\infty}\int_0^1 \omega\big(\delta_0(t),\dot{W}_0^-(t)\big) \dd t $$
 which is the Stokes constant of the resonant map.
\end{remark}

\begin{lemma}
 There exist constants $\theta_i\in\C$ such that
 \begin{align*}
  \theta(\varepsilon)= \sum_{n=0}^M \varepsilon^n \theta_n + O(\varepsilon^{M+1}).
 \end{align*}
\end{lemma}

\begin{proof}
We use the previous lemma with equation \eqref{eq_iner_label_01}.
As $\nu$ was chosen such that $\me^{-2\pi\ii t} = O(\varepsilon^{M+2})$ for $t\in L_1(\nu)$,
then for any $n\in\N$ it holds $t^{n+2} \me^{-2\pi\ii t} = O(\varepsilon^{M+1})$. 
This gives
\begin{align*}
 \rho_n(t) = \theta_n + O(\varepsilon^{M+1}),
\end{align*}
which we can combine with the equation \eqref{eq_iner_label_01} to get
\begin{align*}
 \theta(\varepsilon) =&  \sum_{n=0}^M\varepsilon^n \theta_n + O(\varepsilon^{M+1}). \qedhere
\end{align*}
\myqed
\end{proof}

\subsection{The constant term of \texorpdfstring{$\Theta$}{Theta}}

\begin{lemma}
 There exists $C>0$ such that
 $$|c_0| \les C \varepsilon^{-18} \me^{-\frac{ 4\pi^2}{\varepsilon} } \big( 1 + O(\varepsilon^{1/2}) \big). $$
\end{lemma}

\begin{proof}
Let $\tau_h$ be such that $W^+(\varepsilon;\tau_h) = W^-(\varepsilon;\tau_h)$.
Then $W^+(\varepsilon;\tau_h+1) = W^-(\varepsilon;\tau_h+1)$.
 Let $A$ be the signed area enclosed by these two pieces of the separatrices.
 Using Green's formula to calculate the area we get that
\begin{align*}
 A = \frac{1}{2}\int_0^1 \omega \big(W^+(\varepsilon;\tau_h+s),\dot{W}^+(\varepsilon;\tau_h+s)\big)
 - \omega \big(W^-(\varepsilon;\tau_h+s),\dot{W}^-(\varepsilon;\tau_h+s)\big) \dd s.
\end{align*}
It holds that $W^+ = W^- + \delta$ so we have
\begin{align*}
 \omega \big(W^+ ,\dot{W}^+ \big) - \omega \big(W^- ,\dot{W}^- \big) &= \omega
 \big(\delta ,\dot{W}^- \big) +  \omega \big(W^- ,\dot{\delta} \big)  +  \omega \big(\delta ,\dot{\delta} \big) \\
 &= \frac{\dd}{\dd t}\omega \big( W^- , \delta \big) + 2\omega \big(\delta ,\dot{W}^-
 \big) - \omega \big(\delta ,\dot{\delta} \big).
\end{align*}
We define
$$\sigma_0 = \int_0^1 \omega \big( \delta(\varepsilon;\tau_h+s), \dot{W}^-(\varepsilon;\tau_h+s) \big) \dd s $$
and we have
$$ A - \sigma_0 = \frac{1}{2} \omega \big(  W^-(\varepsilon;\tau_h+s) \big),
\delta(\varepsilon;\tau_h+s) \Big|_{s=0}^1 - \frac{1}{2}\int_0^1 \omega \big(\delta(\varepsilon;\tau_h+s)
,\dot{\delta}(\varepsilon;\tau_h+s) \big) \dd s .$$
Since the map is area-preserving, $A=0$ and since $\delta(\varepsilon;\tau_h)=\delta(\varepsilon;\tau_h+1)=0$,
we get that
$$\sigma_0 = \frac{1}{2}\int_0^1 \omega \big(\delta(\varepsilon;\tau_h+s)  ,
\dot{\delta}(\varepsilon;\tau_h+s) \big) \dd s. $$
Using Lemma \ref{thm_splitting_bound} for $\delta$ and $\dot{\delta}$ we get $|\sigma_0|
\les C \varepsilon^{-4} \me^{-\frac{ 4\pi^2}{\varepsilon}} $.

Since $c_0$ is the constant term of a periodic function, it is true that
$$c_0 = \int_0^1 \omega \big( \delta(\varepsilon;\tau_h+s), \Phi(\varepsilon;\tau_h+s) \big) \dd s. $$
So
$$ |c_0 - \sigma_0| \les \int_0^1 \big|\omega \big( \delta(\varepsilon;\tau_h+s),
\Phi(\varepsilon;\tau_h+s) - \dot{W}^-(\varepsilon;\tau_h+s) \big)\big| \dd s . $$
Using Lemmas \ref{thm_splitting_bound} and \ref{thm_var_eq_sol_exp_bound} we get that
$$ |c_0 - \sigma_0| \les C \varepsilon^{-18} \me^{-\frac{ 4\pi^2}{\varepsilon} } \big( 1 + O(\varepsilon^{1/2}) \big). $$

Combining the above estimates concludes the proof.
\myqed
\end{proof}

\subsection{Asymptotic series  for the homoclinic invariant}

Now we have all the ingredients we need in order to prove the asymptotic series for the Lazutkin homoclinic invariant.

\begin{lemma}
 There exist real numbers $\vartheta_n$ such that
 $$\Omega(\varepsilon) = \Bigg( \sum_{n=0}^M \vartheta_n \varepsilon^n  + O(\varepsilon^{M+1})
 \Bigg)\me^{-\frac{2\pi^2}{\varepsilon}}.$$
 Moreover, $\vartheta_0 = 4 \pi |\theta_0|$, where $\theta_0$ is the Stokes constant of the resonant map.
\end{lemma}

\begin{proof}
 
 By Lemma \ref{thm_var_eq_sol_exp_bound} we have
 \begin{align*}
  \Phi(\varepsilon;\tau) - \dot{W}^-(\varepsilon;\tau) = O(\varepsilon^{-16}\me^{-\frac{2\pi^2}{\varepsilon})}).
 \end{align*}
 This implies that since 
 \begin{align*}
  \Theta(\varepsilon;\tau) - \Theta^-(\varepsilon;\tau) = \omega\big( \delta(\varepsilon;\tau) ,
  \Phi(\varepsilon;\tau) - \dot{W}^-(\varepsilon;\tau) \big),
 \end{align*}
 using the bound of Lemma \ref{thm_splitting_bound} for real $\tau$, we get 
 \begin{align*}
  \Theta(\varepsilon;\tau) - \Theta^-(\varepsilon;\tau) = O(\varepsilon^{-18}\me^{-\frac{4\pi^2}{\varepsilon}}).
 \end{align*}
 By the Fourier expansion \eqref{eq_Theta_fourier_high} and equation \eqref{eq_first_fourier_coeff} we have
 \begin{align*}
  \Theta(\varepsilon;\tau) &= c_0 + \theta(\varepsilon) \me^{-\frac{2\pi^2}{\varepsilon}}
  \me^{-2\pi\ii\tau} + \overline{\theta(\varepsilon)} \me^{-\frac{2\pi^2}{\varepsilon}}
  \me^{2\pi\ii\tau} + O(\me^{-\frac{4\pi^2}{\varepsilon}}) .
 \end{align*}
 We know that $|c_0| \les C \varepsilon^{-18} \me^{-\frac{ 4\pi^2}{\varepsilon} }
 \big( 1 + O(\varepsilon^{1/2}) \big)$, so we get
 \begin{align*}
  \Theta^-(\varepsilon;\tau) = 2|\theta(\varepsilon)| \, \me^{-\frac{2\pi^2}{\varepsilon}}\,
  \cos\big(2\pi\tau - \arg(\theta(\varepsilon))\big) + O(\varepsilon^{-18}\me^{-\frac{4\pi^2}{\varepsilon}}).
 \end{align*}

 Let $W^+(\tau_h)$ be a homoclinic point, then $\Theta^-(\varepsilon;\tau_h) = 0$.
 From the above relation we get that
 $$ \cos\big(2\pi \tau_h - \arg(\theta(\varepsilon))\big) = O(\varepsilon^{-18}\me^{-\frac{2\pi^2}{\varepsilon}}). $$
 This implies that
 $$ \sin \big(2\pi \tau_h - \arg(\theta(\varepsilon))\big) = 1 + O(\varepsilon^{-36}\me^{-\frac{4\pi^2}{\varepsilon}}). $$
 So
 \begin{align*}
  \dot\Theta^-(\varepsilon;\tau_h) &= 4\pi|\theta(\varepsilon)| \, \me^{-\frac{2\pi^2}{\varepsilon}}\,
  \sin \big(2\pi\tau - \arg(\theta(\varepsilon))\big) + O(\varepsilon^{-18}\me^{-\frac{4\pi^2}{\varepsilon}}) \\
  &= 4\pi|\theta(\varepsilon)| \, \me^{-\frac{2\pi^2}{\varepsilon}}\,
  \Big( 1 + O(\varepsilon^{-36}\me^{-\frac{4\pi^2}{\varepsilon}}) \Big)
  + O(\varepsilon^{-18}\me^{-\frac{4\pi^2}{\varepsilon}}) \\
  &= 4\pi|\theta(\varepsilon)| \, \me^{-\frac{2\pi^2}{\varepsilon}}
  + O(\varepsilon^{-18}\me^{-\frac{4\pi^2}{\varepsilon}}).
 \end{align*}
 Differentiating the relation $\Theta^-(\varepsilon;\tau) =
 \omega(\delta(\varepsilon;\tau),\dot W^-(\varepsilon;\tau))$ we get
 \begin{align*}
  \dot\Theta^-(\varepsilon;\tau) = \omega \big(\dot\delta(\varepsilon;\tau),
  \dot W^-(\varepsilon;\tau) \big) + \omega \big(\delta(\varepsilon;\tau),
  \ddot W^-(\varepsilon;\tau) \big).
 \end{align*}
 Since $\delta(\varepsilon;\tau_h)=0$ we get
 \begin{align*}
  \dot\Theta^-(\varepsilon;\tau_h) &= \omega \big(\dot\delta(\varepsilon;\tau_h),
  \dot W^-(\varepsilon;\tau_h) \big) \\
  &=\omega \big(\dot W^+(\varepsilon;\tau_h),\dot W^-(\varepsilon;\tau_h) \big),
 \end{align*}
 which is by definition the homoclinic invariant.
 
 Finally, in order to prove the lemma, we use the fact that $ \theta(\varepsilon)=
 \sum_{n=0}^M \varepsilon^n \theta_n + O(\varepsilon^{M+1})$. This implies that
 \begin{align*}
  4\pi|\theta(\varepsilon)| = \sum_{n=0}^M \vartheta_n \varepsilon^n  + O(\varepsilon^{M+1})
 \end{align*}
 for some real constants $\vartheta_n$.
\myqed
\end{proof}

\section*{Acknowledgments}

I would like to thank Vassili Gelfreich for having suggested this question
to me and for having generously shared his ideas.

\bibliographystyle{plain}
\bibliography{APM31unfold}

\end{document}